\documentclass[11pt]{amsart}
\usepackage{amsmath,amssymb}
\usepackage{graphicx}
\usepackage{xcolor}
\usepackage{xkeyval}
\usepackage{calc}
\usepackage{ifthen}
\usepackage{algorithm}
\usepackage{algpseudocode}
\usepackage{mathrsfs}
\usepackage{fullpage}
\usepackage[pdftex,colorlinks,backref=page,citecolor=blue]{hyperref}
\usepackage{enumerate}
\usepackage{pgf,tikz}
\usepackage{todonotes}
\usepackage{appendix}
\usetikzlibrary{arrows}
\usetikzlibrary{positioning}
\usepackage{array}

\def\COMMENT#1{}
\def\TASK#1{}

\def\noproof{{\unskip\nobreak\hfill\penalty50\hskip2em\hbox{}\nobreak\hfill%
        $\square$\parfillskip=0pt\finalhyphendemerits=0\par}\goodbreak}
\def\endproof{\noproof\bigskip}
\newdimen\margin   
\def\textno#1&#2\par{%
    \margin=\hsize
    \advance\margin by -4\parindent
           \setbox1=\hbox{\sl#1}%
    \ifdim\wd1 < \margin
       $$\box1\eqno#2$$%
    \else
       \bigbreak
       \hbox to \hsize{\indent$\vcenter{\advance\hsize by -3\parindent
       \sl\noindent#1}\hfil#2$}%
       \bigbreak
    \fi}

\makeatletter \providecommand\@dotsep{5} \def\listtodoname{List of Todos} \def\listoftodos{\@starttoc{tdo}\listtodoname} \makeatother

\newtheorem{thm}{Theorem}[section]
\newtheorem{define}[thm]{Definition}

\newtheorem{lem}[thm]{Lemma}
\newtheorem{claim}[thm]{Claim}
\newtheorem{fact}[thm]{Fact}
\newtheorem{col}[thm]{Corollary}
\newtheorem{conj}[thm]{Conjecture}
\newtheorem{prop}[thm]{Proposition}

\newtheorem{ques}[thm]{Question}
\newtheorem{obs}[thm]{Observation}
\newtheorem*{thm*}{Theorem}
\newtheorem*{define*}{Definition}
\newtheorem*{examp*}{Example}
\newtheorem*{lem*}{Lemma}
\newtheorem*{claim*}{Claim}
\newtheorem*{fact*}{Fact}
\newtheorem*{col*}{Corollary}
\newtheorem*{conj*}{Conjecture}

\newcommand{\eps}{\varepsilon}

\newtheorem*{claim:non-degen}{Claim~\ref{claim:non-degen}}
\newtheorem*{claim:degenfull}{Claim~\ref{claim:degenfull}}
\newtheorem*{claim:non-degensym}{Claim~\ref{claim:non-degensym}}
\newtheorem*{claim:degenfullsym}{Claim~\ref{claim:degenfullsym}}
\newtheorem*{claim:q_2}{Claim~\ref{claim:q_2}}
\newtheorem*{prop:m2h1h2}{Proposition~\ref{prop:m2h1h2}}
\newtheorem*{claim:q_1}{Claim~\ref{claim:q_1}}
\newtheorem*{lemma:errorvalid}{Lemma~\ref{lemma:errorvalid}}
\newtheorem*{claim:polylog}{Claim~\ref{claim:polylog}}
\newtheorem*{claim:conclusion2}{Claim~\ref{claim:conclusion2}}

\title{Proof of the Kohayakawa--Kreuter conjecture for the majority of cases}
\author{Candida Bowtell\textsuperscript{c}}
\author{Robert Hancock\textsuperscript{r}}
\author{Joseph Hyde\textsuperscript{j}}
\thanks{c: School of Mathematics, University of Birmingham, United Kingdom, {\tt c.bowtell@bham.ac.uk}, research supported in parts by ERC Starting Grant 947978, Philip Leverhulme Prize PLP--2020--183 and Leverhulme Trust Early Career Fellowship ECF--2023--393. \\
\indent r: 
Department of Theoretical Computer Science, Faculty of Information Technology,
Czech Technical University in Prague, Th\'akurova 7, Prague, 160 00, Czech Republic.
\texttt{robert.hancock@cvut.cz}. Research supported by a Humboldt Research Fellowship and ERC Advanced Grant 883810.\\
\indent j: Department of Mathematics, Strand Building, 5th floor, Strand Campus, Strand, London, WC2R 2LS, {\tt joseph.hyde@kcl.ac.uk}, research supported by the UK Research and Innovation Future Leaders Fellowship MR/S016325/1 and ERC Advanced Grant 101020255.\\
}

\begin{document}

\maketitle

\begin{abstract}
For graphs $G, H_1,\dots,H_r$, write $G \to (H_1, \ldots, H_r)$ to denote the property that whenever we $r$-colour the edges of $G$, there is a monochromatic copy of $H_i$ in colour $i$ for some $i \in \{1,\dots,r\}$.
Mousset, Nenadov and Samotij proved an upper bound on the threshold function for the property that $G_{n,p} \to (H_1,\dots,H_r)$, thereby resolving the $1$-statement of the Kohayakawa--Kreuter conjecture. 
We reduce the $0$-statement of the Kohayakawa--Kreuter conjecture to a natural deterministic colouring problem and resolve this problem for almost all cases,
which in particular includes (but is not limited to) 
when $H_2$ is strictly $2$-balanced and either has density greater than $2$ or is not bipartite. 
In addition, we extend our reduction to hypergraphs, proving the colouring problem in almost all cases there as well.
\end{abstract}


An extended abstract describing this work appeared in the proceedings of EUROCOMB2023~\cite{bhh}.


\section{Introduction}\label{sec:intro}

Let $r \in \mathbb{N}$ and $G, H_1, \ldots, H_r$ be $k$-uniform hypergraphs, hereafter referred to as `$k$-graphs'.
We write $G \to (H_1, \ldots, H_r)$ to denote the property that whenever we colour the edges of $G$ with colours from the set $[r] := \{1, \ldots, r\}$ there exists $i \in [r]$ and a copy of $H_i$ in $G$ monochromatic in colour $i$. 
Ramsey~\cite{r} famously showed, for $n$ sufficiently large, that $K_n \to (H_1, \ldots, H_r)$. 
On inspection one may wonder if Ramsey's theorem is true because $K_n$ is very dense, 
but Folkman~\cite{f}, and in a more general setting Ne\v{s}et\v{r}il and R\"{o}dl~\cite{nr}, 
proved that for any graph $H$ there are locally sparse graphs $G = G(H)$ such that $G \to (H_1, \ldots, H_r)$ when $H_1 = \cdots = H_r = H$. 
By considering Ramsey properties in the random setting, we discover that such graphs $G$ are in fact very common. Let $G^k_{n,p}$ be the binomial random graph with $n$ vertices and edge probability $p$. Throughout this paper we write `$G_{n,p}$' for `$G^2_{n,p}$'. 
To speak about Ramsey properties in random hypergraphs we must introduce the notion of a \emph{threshold}. 
Let $P$ be a monotone (increasing) property of hypergraphs, that is, if $H \in P$ and $H \subseteq G$, then $G \in P$.\footnote{By $H \in P$ we mean that $H$ has property $P$. Indeed, a property for hypergraphs can be thought of as the collection of all hypergraphs that have that property.} In this paper, we say that a function $f:\mathbb{N} \to \mathbb{R}$ is a {\it threshold for $P$ in $G^k_{n,p}$} if there exist constants $b, B \in \mathbb{N}$ such that \[\lim\limits_{n \to \infty} \mathbb{P}[G^k_{n,p} \in P] = 
\begin{cases}
    0       & \quad \text{if } p \leq bf(n),\\
    1       & \quad \text{if } p \geq Bf(n).
\end{cases}\] Bollob\'{a}s and Thomason~\cite{bt} and Friedgut~\cite{f}, for $k = 2$ and general $k \in \mathbb{N}$, respectively, showed that every non-trivial\footnote{$P$ is non-trivial if there exist (hyper)graphs $H$ and $G$ such that $H\in P$ and $G\not\in P$, i.e.\ $P$ does not hold for all graphs.} monotone property $P$ for $G^k_{n,p}$ has a threshold.

\L{}uczak, Ruci\'{n}ski and Voigt~\cite{lrv}, improving on earlier work of Frankl and R\"{o}dl~\cite{fr}, proved that $p = n^{-1/2}$ is a threshold for the property $G_{n,p} \to (K_3, K_3)$. 
Building on this, R\"{o}dl and Ruci\'{n}ski~\cite{rr1, rrrandom, rr2} determined a threshold for the general symmetric case in graphs. 
Before stating this result, we give the following definitions concerning density conditions, including the definition for all $k$-graphs, as we shall need it in this form later. 
Let $H$ be a graph. We define 
\begin{align*}
d(H) := 
    \begin{cases}
        e(H)/v(H)         & \quad \text{if } v(H) \geq 1,\\
        0               & \quad \text{otherwise};
    \end{cases}
\end{align*}
and define the \emph{$m$-density} of $H$ to be
\begin{align*}
m(H) & := \max\{d(J): J \subseteq H\}.
\end{align*}
We say that a graph $H$ is \emph{balanced w.r.t $d$}, or just \emph{balanced}, if we have $d(H) = m(H)$. 
Moreover, we say $H$ is \emph{strictly balanced} if for every proper subgraph $J \subsetneq H$, we have $d(J) < m(H)$.

We define 
\begin{align*}
d_1(H) & := 
    \begin{cases}
        e(H)/(v(H)-1)         & \quad \text{if } v(H) \geq 2,\\
        0               & \quad \text{otherwise};
    \end{cases}
\end{align*}
and the \emph{arboricity} (also known as the $1$-density) of $H$ to be
\begin{align*}    
ar(H)=m_1(H) & := \max\{d_1(J): J \subseteq H\}.
\end{align*}


For $k\geq 2$ and a $k$-graph $H$, we define \[d_k(H) := 
    \begin{cases}
        (e(H) - 1)/(v(H) - k)         & \quad \text{if $H$ is non-empty with} \ v(H) \geq k+1,\\
        1/k                         & \quad \text{if} \ H \cong K_k,\\    
        0                           & \quad \text{otherwise};
    \end{cases}\] and the \emph{$k$-density of $H$} to be \[m_k(H) := \max\{d_k(J): J \subseteq H\}.\] 
    
We say that a $k$-graph $H$ is \emph{$k$-balanced} if $d_k(H) = m_k(H)$, and \emph{strictly $k$-balanced} if for all proper subhypergraphs $J \subsetneq H$, we have $d_k(J) < m_k(H)$.

\begin{thm}[R\"{o}dl and Ruci\'{n}ski~\cite{rr2}]\label{thm:rr}
Let $r \geq 2$ and $H$ be a non-empty graph such that at least one component of $H$ is not a star or, when $r = 2$, a path on $3$ edges. 
Then there exist positive constants $b, B > 0$ such that \[\lim\limits_{n \to \infty} \mathbb{P}[G_{n,p} \to (\underbrace{H, \ldots, H}_{r \ times})] = 
\begin{cases}
    0       & \quad \text{if } p \leq bn^{-1/m_2(H)},\\
    1       & \quad \text{if } p \geq Bn^{-1/m_2(H)}.
\end{cases}\] 
\end{thm}

As is common in this study, we call the statement for $p \leq bn^{-1/m_2(H)}$ in Theorem~\ref{thm:rr} a \emph{$0$-statement} and the statement for $p \geq Bn^{-1/m_2(H)}$ a \emph{$1$-statement}. 

The assumptions put on $H$ in Theorem~\ref{thm:rr} are necessary. 
If every component of $H$ is a star then $G_{n,p} \to (\underbrace{H, \ldots, H}_{r \ times})$ as soon as sufficiently many vertices of degree $(\Delta(H) - 1)r + 1$ appear in $G_{n,p}$. 
A threshold for this property in $G_{n,p}$ is $f(n) = n^{-1-1/((\Delta(H) - 1)r + 1)}$, but $m_2(H) = 1$.
For the case when $r = 2$ and at least one component of $H$ is a path on 3 edges while the others are stars, the $0$-statement of Theorem~\ref{thm:rr} does not hold. 
Indeed, if $p = cn^{-1/m_2(P_3)} = cn^{-1}$ for some $c > 0$, then one can show that the probability that $G_{n,p}$ contains a cycle of length $5$ with an edge pending at every vertex is bounded from below by a positive constant $d = d(c)$.
One can check that every colouring of the edges of this augmented $5$-cycle with $2$ colours yields a monochromatic path of length $3$, and hence $p = o(n^{-1/m_2(H)})$ is required for the 0-statement to hold in this case.
This construction was missed in \cite{rr2}, and was eventually observed by Friedgut and Krivelevich~\cite{fk}, who corrected the $0$-statement to have the assumption $p = o(n^{-1/m_2(H)})$ instead. 

The intuition behind the threshold in Theorem~\ref{thm:rr} is as follows: Firstly, assume $H$ is $2$-balanced. The expected number of copies of a graph $H$ in $G_{n,p}$ is $\Theta(n^{v(H)}p^{e(H)})$ and the expected number of edges is $\Theta(n^2p)$. 
For $p = n^{-1/m_2(H)}$, these two expectations are of the same order since $H$ is $2$-balanced. 
That is to say, if the expected number of copies of $H$ containing some fixed edge is smaller than some small constant $c$, then we can hope to colour without creating a monochromatic copy of $H$: 
very roughly speaking, each copy will likely contain an edge not belonging to any other copy of $H$, so by colouring these edges with one colour and all other edges with a different colour we avoid creating monochromatic copies of $H$. 
If the expected number of copies of $H$ containing some fixed edge is greater than some large constant $C$ then a monochromatic copy of $H$ may appear in any $r$-colouring since the copies of $H$ most likely overlap heavily.

Note that Nenadov and Steger~\cite{ns} produced a short proof of Theorem~\ref{thm:rr} using the hypergraph container method (see \cite{bms, st}).
Indeed, the proof of the $1$-statement is reduced to a mere two pages and demonstrates a fundamental application of the container method.

Finally, before we move onto the asymmetric setting, it is worth mentioning that sharp thresholds for $G_{n,p} \to (H,H)$ have been obtained when $H$ is a tree \cite{fk}, a triangle \cite{frrt} or a strictly $2$-balanced graph that can be made bipartite by the removal of some edge \cite{ss}.
Very recently, in work which generalises the non-tree cases above, sharp thresholds were obtained for strictly $2$-balanced graphs which are `collapsible'~\cite{fkss}; see \cite{fkss} for details, as well as further discussion on sharp thresholds in general random Ramsey-style problems. 

\subsection{The Kohayakawa--Kreuter conjecture: asymmetric case}

Asymmetric properties of random graphs were first considered by Kohayakawa and Kreuter~\cite{kk}. In this paper, we will prove the $0$-statement of a conjecture from \cite{kk} for almost all finite collections of $k$-graphs.
This conjecture is an asymmetric analogue of Theorem~\ref{thm:rr}. 
To state the conjecture we require some more nomenclature, where we again make the definitions for $k$-graphs.

For $k$-graphs $H_1$ and $H_2$ with $m_k(H_1) \geq m_k(H_2)$, we define \[d_k(H_1,H_2) := 
    \begin{cases}
        \frac{e(H_1)}{v(H_1) - k + \frac{1}{m_k(H_2)}}         & \quad \text{if $H_2$ is non-empty and} \ v(H_1) \geq k,\\
        0                         & \quad \text{otherwise};
    \end{cases}\] and the \emph{asymmetric $k$-density of the pair $(H_1, H_2)$} to be \[m_k(H_1, H_2) := \max\left\{d_k(J,H_2): J \subseteq H_1\right\}.\]
    
We say that $H_1$ is \emph{balanced with respect to $d_k(\cdot, H_2)$} if we have $d_k(H_1, H_2) = m_k(H_1, H_2)$ and \emph{strictly balanced with respect to $d_k(\cdot, H_2)$} if for all proper subgraphs $J \subsetneq H_1$ we have $d_k(J, H_2) < m_k(H_1, H_2)$. 
Note that we always have $m_k(H_1) \geq m_k(H_1, H_2) \geq m_k(H_2)$ (see Proposition~\ref{prop:m2h1h2}).

\begin{conj}[Kohayakawa and Kreuter~\cite{kk}]\label{conj:kk}
Let $r \geq 2$ and suppose that $H_1, \ldots, H_r$ are non-empty graphs such that $m_2(H_1) \geq m_2(H_2) \geq \cdots \geq m_2(H_r)$ and $m_2(H_2) > 1$.
Then there exist constants $b, B > 0$ such that \[\lim\limits_{n \to \infty} \mathbb{P}[G_{n,p} \to (H_1, \ldots, H_r)] =   
\begin{cases}
    0       & \quad \text{if } p \leq bn^{-1/m_2(H_1,H_2)},\\
    1       & \quad \text{if } p \geq Bn^{-1/m_2(H_1,H_2)}.
\end{cases}\] 
\end{conj}

Observe that, since $H_2$ is non-empty, we would always need $m_2(H_2) \geq 1$ as an assumption, otherwise $m_2(H_2) = 1/2$ (that is, $H_2$ is the union of a matching and some isolated vertices, and $m_2(H_2)$ is maximised by $K_2 \subseteq H_2$) and we would have that $m_2(H_1, H_2) = e(J)/v(J)$ for some non-empty subgraph $J \subseteq H_1$.
Then for any constant $B > 0$, and $p = Bn^{-1/m_2(H_1, H_2)}$, the probability that $G_{n,p}$ contains no copy of $H_1$ exceeds a positive constant $C = C(B)$; see, e.g. \cite{jlr}. 
We include the assumption of Kohayakawa, Schacht and Sp\"{o}hel~\cite{kss}, that $m_2(H_2) > 1$, so that we avoid possible complications arising from $H_2$ (and/or $H_1$) being certain forests, such as those excluded in the statement of Theorem~\ref{thm:rr}.
Note that if $m_2(H_2) = 1$ there is at least one case for which Conjecture~\ref{conj:kk} does not hold and which is not covered by Theorem~\ref{thm:rr}: Let $H_1 = K_3$ and $H_2 = K_{1,2}$.
One can calculate that $m_2(H_1, H_2) = 3/2$.\footnote{Note that $H_2$ is strictly $2$-balanced and $H_1$ is strictly balanced with respect to $m_2(\cdot, H_2)$.}
Now, consider the graph $G$ with $9$ edges and $6$ vertices constructed by taking a copy $T$ of $K_3$ and appending a triangle to each edge of $T$ (see Figure~\ref{fig:k3k12badgraph}). We have $m(G) = 3/2 = m_2(H_1, H_2)$, but $G \to (H_1, H_2)$. 
Thus we at least require that $p = o(n^{-1/m_2(H_1, H_2)})$ to ensure a $0$-statement holds in this case.

\begin{figure}[!ht]
\begin{center}
\begin{tikzpicture}
\draw [line width=2pt] (3,3)-- (2,1.5);
\draw [line width=2pt] (4,1.5)-- (3,3);
\draw [line width=2pt] (4,1.5)-- (2,1.5);
\draw [line width=2pt] (2,1.5)-- (1,0);
\draw [line width=2pt] (1,0)-- (3,0);
\draw [line width=2pt] (3,0)-- (2,1.5);
\draw [line width=2pt] (4,1.5)-- (5,0);
\draw [line width=2pt] (5,0)-- (3,0);
\draw [line width=2pt] (3,0)-- (4,1.5);
\begin{scriptsize}
\draw [fill=black] (3,3) circle (2.5pt);
\draw [fill=black] (2,1.5) circle (2.5pt);
\draw [fill=black] (4,1.5) circle (2.5pt);
\draw [fill=black] (3,0) circle (2.5pt);
\draw [fill=black] (1,0) circle (2.5pt);
\draw [fill=black] (5,0) circle (2.5pt);
\end{scriptsize}
\end{tikzpicture}
\caption{Every red/blue edge colouring of the `triforce' graph yields either a monochromatic copy of $K_3$ in red or a monochromatic copy of $K_{1,2}$ in blue.}\label{fig:k3k12badgraph}
\end{center}
\end{figure}

The intuition behind the threshold in Conjecture~\ref{conj:kk} is most readily explained in the case of $r = 3$, $H_2 = H_3$ and when $m_2(H_1) > m_2(H_1, H_2)$. (The following explanation is adapted from \cite{gnpsst}.)
Firstly, observe that we can assign colour $1$ to every edge that does not lie in a copy of $H_1$. 
Since $m_2(H_1) > m_2(H_1, H_2)$, we expect that, with $p = \Theta(n^{-1/m_2(H_1,H_2)})$, the copies of $H_1$ in $G_{n,p}$ do not overlap much (by similar reasoning as in the intuition for the threshold in Theorem~\ref{thm:rr}). 
Hence the number of edges left to be coloured is of the same order as the number of copies of $H_1$, which is $\Theta(n^{v(H_1)}p^{e(H_1)})$. 
If we further assume that these edges are randomly distributed (which although not correct, does give good intuition) then we get a random graph $G^*$ with edge probability $p^* =  \Theta(n^{v(H_1) - 2}p^{e(H_1)})$. 
Now we colour $G^*$ with colours $2$ and $3$, and apply the intuition from the symmetric case (as $H_2 = H_3$): if the copies of $H_2$ are heavily overlapping then we cannot hope to colour without getting a monochromatic copy of $H_2$, but if not then we should be able to colour.
As observed earlier, a threshold for this property is $p^* = n^{-1/m_2(H_2)}$. Solving $n^{v(H_1) - 2}p^{e(H_1)} = n^{-1/m_2(H_2)}$ for $p$ then yields $p = n^{-1/m_2(H_1,H_2)}$, the conjectured threshold. 

Following some partial progress by 
~\cite{gnpsst, hst, kk, kss, msss}, the $1$-statement of Conjecture~\ref{conj:kk} was fully solved by Mousset, Nenadov and Samotij~\cite{mns}.
We therefore now focus on the $0$-statement.

\subsection{The $0$-statement of the Kohayakawa--Kreuter Conjecture}

We first consider a natural reduction for proving the $0$-statement of Conjecture~\ref{conj:kk}.
Let $H_1, H_2$ be graphs with $m_2(H_1) \geq m_2(H_2)$. We call $(H_1,H_2)$ a \emph{heart} if 
\begin{itemize}
\item $H_2$ is strictly $2$-balanced,
\item when $m_2(H_1)=m_2(H_2)$, $H_1$ is strictly $2$-balanced,
\item when $m_2(H_1)>m_2(H_2)$, $H_1$ is strictly balanced with respect to $d_2(\cdot, H_2)$.
\end{itemize}
It is easy to show that for any pair of graphs $(H_1,H_2)$ with $m_2(H_1) \geq m_2(H_2)$, 
there exists a heart $(H_1',H_2')$ with \begin{itemize}
    \item $H_i' \subseteq H_i$ \ \mbox{for each} \ $i \in [2]$,
    \item $m_2(H_2')=m_2(H_2)$,
    \item $m_2(H_1',H_2')=m_2(H_1,H_2)$ \ \mbox{if} \ $m_2(H_1) > m_2(H_2)$, and
    \item $m_2(H_1')=m_2(H_1)$ \ \mbox{if} \ $m_2(H_1) = m_2(H_2)$.
\end{itemize} 
We call this pair a \emph{heart of $(H_1,H_2)$}. The relevance of hearts for $0$-statements comes from the following simple observation.
\begin{obs}\label{obs:heart}
Let $H_1,H_2$ be graphs. In order to prove a $0$-statement at $n^{-1/m_2(H_1, H_2)}$ for $(H_1,H_2)$, it suffices to prove the $0$-statement at $n^{-1/m_2(H_1', H_2')}$ for some heart $(H_1',H_2')$ of $(H_1,H_2)$.
\end{obs}
The observation follows immediately since any colouring avoiding a monochromatic copy of a subgraph of some graph $H$ clearly avoids a monochromatic copy of $H$ itself, and in all cases $m_2(H_1', H_2') = m_2(H_1, H_2)$. 

Previously obtained results on the $0$-statement of Conjecture~\ref{conj:kk} are summed up in the following theorem.
\begin{thm}\label{thm:done}
The $0$-statement of Conjecture~\ref{conj:kk} holds for $(H_1,\dots,H_r)$ in each of the following cases.
 
For some heart $(H_1',H_2')$ of $(H_1,H_2)$, we have:
 \begin{enumerate}
 \item[(i)] $H_1'$ and $H_2'$ are both cycles {\rm(\cite{kk})};
 \item[(ii)] $H_1'$ and $H_2'$ are both cliques {\rm(\cite{msss})};
 \item[(iii)] $H_1'$ is a clique and $H_2'$ is a cycle {\rm(\cite{lmms})};
 \item[(iv)] $H_1'$ and $H_2'$ are a pair of regular graphs, excluding the cases when
 (a) $H_1'$ is a clique and $H_2'$ is a cycle; (b) $H_2'$ is a cycle and $v(H_1') \geq v(H_2')$; (c) $(H_1',H_2')=(K_3,K_{3,3})$ 
 {\rm(\cite{h})}; or,
 \end{enumerate}
we have:
 \begin{enumerate}
 \item[(v)] $m_2(H_1)=m_2(H_2)$ {\rm(\cite{ks})}.
 \end{enumerate}
\end{thm}
Note that (i)--(iii) above were only stated for $(H_1,H_2)$ of the precise form of $(H_1',H_2')$ stated (i.e. for (i), with $H_1$ and $H_2$ themselves both cycles). 
However, note that each such pair is a heart itself, so the theorem, via Observation~\ref{obs:heart}, extends to the cases indicated. 

Also note that (v) above was proven in a much stronger sense.
Let $\mathcal{F}$ be a finite family of graphs and let $r \geq 2$. 
Call a graph $G$ \emph{$r$-list-Ramsey with respect to $\mathcal{F}$} if
there exists an  assignment of lists of $r$ colours to each edge (possibly different lists for each edge)  
such that however we choose colours for each edge from its list,
there is a monochromatic copy of some $F \in \mathcal{F}$.
Define $m_2(\mathcal{F}):=\min \{ m_2(F): F \in \mathcal{F} \}$.
Kuperwasser and Samotij's result~\cite[Theorem 1.5]{ks} states that $n^{-1/m_2(\mathcal{F})}$ is the threshold for $G_{n,p}$ being $r$-list-Ramsey with respect to $\mathcal{F}$.
Now, forcing the lists to be $[r]$ for each edge and setting $\mathcal{F}=\{H_1,H_2\}$ for graphs $H_1,H_2$ with $m_2(H_1)=m_2(H_2)$ recovers Theorem~\ref{thm:done}(v).

In this paper, we vastly extend the number of cases for which the $0$-statement holds, increasing the proportion of known cases sharply from $o(1)$ to $1-o(1)$. In particular, case (ii) of Theorem \ref{thm:main} shows that we prove the $0$-statement for all collections $(H_1, H_2, \ldots, H_r)$ in which $H_2$ is not bipartite, which already satisfies this. The six additional cases then present nice characterisations of a large number of other graphs for which Conjecture~\ref{conj:kk} holds. In fact, the breadth of cases covered by this result might best be seen by inspecting which cases are not covered by our theorem.  

\begin{thm}\label{thm:main}
The $0$-statement of Conjecture~\ref{conj:kk} holds for all but a $o(1)$ proportion of cases. In particular, it holds for $(H_1,\dots,H_r)$ in all of the following cases.
We have:
\begin{enumerate}
\item[(i)] $m_2(H_1)=m_2(H_2)$ (independent proof from~\cite{ks}); or,
\end{enumerate} 
for some heart $(H_1',H_2')$ of $(H_1,H_2)$, we have:
\begin{enumerate}
\item[(ii)] $\chi(H_2') \geq 3$;
\item[(iii)] $m(H_2')>2$; 
\item[(iv)] $ar(H_2')>2$;
\item[(v)] $H_1'=K_a$ for some $a \geq 3$;
\item[(vi)] $H_1'=K_{a,b}$ for some $a,b \geq 3$;
\item[(vii)] there exists $\ell \in \mathbb{N}$ such that $\ell<ar(H_1')$ and $m_2(H_1,H_2) \leq \ell+\frac{1}{2}$;
\item[(viii)] there exists $\ell \in \mathbb{N}$ such that $\ell<m(H_1')$ and $m_2(H_1,H_2) \leq \ell+1$, and $H_2'$ is not a cycle.
\end{enumerate}
\end{thm}

We remark that case (iii) is redundant since all graphs $H$ satisfy $ar(H) \geq m(H)$, and so it is covered by case (iv). 
However we still wish to highlight it since,
as we will see later, the required colouring result for (iii) (see Lemma~\ref{lem:maincolouring}) has a natural proof which generalises to the hypergraph case, which we now discuss.

\subsection{Random Ramsey results for hypergraphs}
It is natural to ask for generalisations of both Theorem~\ref{thm:rr} and Conjecture~\ref{conj:kk} to the setting of $k$-graphs for any fixed $k \geq 2$.
In these versions, we replace graphs with $k$-graphs, `$G_{n,p}$' with `$G^k_{n,p}$', `$m_2(H)$' with `$m_k(H)$' and `$m_2(H_1,H_2)$' with `$m_k(H_1,H_2)$'. 
Note that the intuition provided for why $n^{-1/m_2(H)}$ and $n^{-1/m_2(H_1,H_2)}$ are the respective thresholds carries through to the hypergraph setting, as does the definition of a heart and Observation~\ref{obs:heart}.

R\"{o}dl and Ruci\'{n}ski~\cite{rrhyper} conjectured that a $k$-graph version of the $1$-statement of Theorem~\ref{thm:rr} should hold.
They proved this conjecture when $H$ is the complete $3$-graph on $4$ vertices and there are $2$ colours and, together with Schacht~\cite{rrs},
extended their result to $k$-partite $k$-uniform hypergraphs and $r$ colours.
In 2010, Freidgut, R\"{o}dl and Schacht~\cite{frs} and, independently, Conlon and Gowers~\cite{cg} resolved R\"{o}dl and Ruci\'{n}ski's conjecture for all $k$-graphs. 

\begin{thm}[\cite{cg, frs}]\label{thm:rrhyper}
    Let $H$ be a $k$-graph with maximum degree at least $2$ and let $r\geq 2$. There exists a constant $B > 0$ such that for $p = p(n)$ satisfying $p \geq Bn^{-1/m_k(H)}$, we have

    \[\lim_{n\to\infty}\mathbb{P}[G^k(n,p) \to (\underbrace{H, \ldots, H}_{r \ times})] = 1.\]
\end{thm}

A corresponding $0$-statement for Theorem~\ref{thm:rrhyper} was proved when $H$ is a complete hypergraph in~\cite{npss} and for some other special classes in~\cite{t}, 
while in~\cite{gnpsst} the authors proved that the threshold is not $n^{-1/m_k(H)}$ for certain cases.
See Section~\ref{sec:conclude} for more discussion regarding what is proved in~\cite{t,gnpsst} and what is remaining to prove for a hypergraph version of Theorem~\ref{thm:rr}.

Regarding the asymmetric case for hypergraphs, 
a $1$-statement at $n^{-1/m_k(H_1,H_2)}$ was proven for asymmetric cliques in~\cite{gnpsst}, and was proven for all cases by Mousset, Nenadov and Samotij~\cite{mns}.\footnote{Mousset, Nenadov and Samotij  observe this in the last paragraph of their paper.}

We make the first progress on finding a matching $0$-statement for a pair of hypergraphs $H_1, H_2$ with $H_1 \not=H_2$. In fact, we show for all $r \geq 2$, as in the graph case, that a $1-o(1)$ proportion of all tuples $(H_1, H_2, \dots, H_r)$ satisfy the matching $0$-statement.

Note that $m$-density and arboricity have exactly the same definitions for $k$-graphs. 
For a $k$-graph $H$, we use the \emph{weak chromatic number} (denoted by $\chi(H)$),
which is the minimum number $t$ of colours required such that there exists a $t$-colouring of the vertices which does not produce a monochromatic edge.
This generalises the chromatic number of graphs.

\begin{thm}\label{thm:mainhyp}
Let $r \geq 2$, $k \geq 3$ and suppose that $H_1, \ldots, H_r$ are non-empty $k$-graphs such that $m_k(H_1) \geq m_k(H_2) \geq \cdots \geq m_k(H_r)$ and $m_k(H_2) > 1$.
There exists a constant $b > 0$ such that \[\lim\limits_{n \to \infty} \mathbb{P}[G_{n,p} \to (H_1, \ldots, H_r)] =   
    0      \quad \text{if } p \leq bn^{-1/m_k(H_1,H_2)},
\] 
in each of the following cases. For some heart $(H_1',H_2')$ of $(H_1,H_2)$ we have:
\begin{enumerate}
\item[(i)] $\chi(H_2') \geq k+1$;
\item[(ii)] $m(H_2')>\binom{v_1}{k-2}+1$. 
\end{enumerate}
\end{thm}

\subsection{Resolution for all $k$-graphs up to proving a deterministic colouring result}
In this section, we wish to highlight a result which demonstrates that the answer to the problem of determining whether $n^{-1/m_k(H_1,H_2)}$ is a threshold for $G^k_{n,p} \to (H_1,\dots,H_r)$ depends entirely on the answer to a certain deterministic colouring question. Indeed, this result, Theorem~\ref{thm:coloursuffices} below, is the main result of this paper.

Let $G, H_1,\dots,H_r$ be $k$-graphs with $m_k(H_1) \geq \cdots \geq m_k(H_r)$ and $m_k(H_2)>1$, and 
suppose that $G$ has constant size and $m(G) \leq m_k(H_1,H_2)$. 
For any constant $b>0$ it is easy to show for $p=bn^{-1/m_k(H_1,H_2)}$ that $G$ will appear as a subhypergraph of $G^k_{n,p}$ with at least constant probability.
Therefore, if $n^{-1/m_k(H_1,H_2)}$ is to be a threshold function for $G^k_{n,p} \to (H_1,\dots,H_r)$,
it is certainly necessary that $G \not \to (H_1,\dots,H_r)$. 
It is natural to ask whether this is in fact the only obstruction in proving a $0$-statement i.e.
is the statement
`$m(G) \leq m_k(H_1,H_2)$ implies $G \not \to (H_1,\dots,H_r)$' 
a sufficient condition for proving a $0$-statement at the threshold $n^{-1/m_k(H_1,H_2)}$?
Observe that we only need to prove such a statement for $(H_1,H_2)$ which are hearts.
Further motivation for this question arises from observing that for all known cases for which $n^{-1/m_k(H_1,H_2)}$ is not a threshold for $(H_1,H_2)$ given earlier (including the cases where $H_1=H_2$), there exists a $k$-graph $G$ with $m(G) \leq m_k(H_1,H_2)$ such that $G \to (H_1,\dots,H_r)$.

In the symmetric setting, the answer to this question is yes. 
In the case of graphs this result can be found within both the original works of R\"{o}dl and Ruci\'{n}ski~\cite{rr1, rrrandom, rr2} and also in the short proof by Nenadov and Steger~\cite{ns}.
In the symmetric $k$-graphs setting, the result appears in~\cite{npss,t}.
Additionally, in~\cite{npss},  Nenadov et al. showed that this same phenomenon occurs for a number of other Ramsey-style properties. 
Therefore, naturally, there have been attempts to answer this question in the asymmetric setting.
After previous results of~\cite{gnpsst, h},
we resolve this question in the following sense.
\begin{thm}\label{thm:coloursuffices}
Let $H_1,\dots,H_r$ be $k$-graphs with $m_k(H_1) \geq \cdots \geq m_k(H_r)$ and $m_k(H_2) > 1$ and suppose $(H_1,H_2)$ is a heart.
If for all $k$-graphs $G$ we have that $m(G) \leq m_k(H_1,H_2)$ implies $G \not \to (H_1,H_2)$,
then 
there exists a constant $b > 0$ such that \[\lim\limits_{n \to \infty} \mathbb{P}[G^k_{n,p} \to (H_1, \ldots, H_r)] =   
    0      \quad \text{if } p \leq bn^{-1/m_k(H_1,H_2)}.
\] 
\end{thm}
This means that in order to prove the $0$-statement at $n^{-1/m_k(H_1,H_2)}$ for a pair $(H_1,H_2)$ (in particular, in order to prove Conjecture~\ref{conj:kk} for the remaining open cases)
it suffices to prove the colouring result contained in Theorem~\ref{thm:coloursuffices} -- `for all $k$-graphs $G$ we have that $m(G) \leq m_k(H_1, H_2)$ implies $G \not\to (H_1, H_2)$' -- for a heart $(H_1',H_2')$ of $(H_1,H_2)$.

Observe that for $r \geq 3$, showing that it suffices to prove `for all $k$-graphs $G$ we have that $m(G) \leq m_k(H_1,H_2)$ implies $G \not \to (H_1,\dots,H_r)$' would improve upon  Theorem~\ref{thm:coloursuffices}; 
however, note that Conjecture~\ref{conj:kk} states that the threshold should only depend on $H_1$ and $H_2$, so Theorem~\ref{thm:coloursuffices} is natural as stated. Indeed, if some $H_1,H_2,H_3$ with $m_2(H_1)\geq m_2(H_2) \geq m_2(H_3)$ and $m_2(H_2)>1$ satisfy that for all $G$ with $m(G) \leq m_2(H_1,H_2)$ we have $G \not \to (H_1,H_2,H_3)$, while simultaneously for one such $G'$ we have $G' \to (H_1,H_2)$, then the pair $(H_1, H_2)$ is a counterexample to Conjecture~\ref{conj:kk}.
\COMMENT{RH: New paragraph, please check carefully!}

Note that in the symmetric setting, the colouring result contained in Theorem~\ref{thm:coloursuffices} holds in all cases\footnote{That is, for all graphs $H$ satisfying the natural condition $m_2(H) > 1$, in order to avoid those cases discussed after the statement of Theorem~\ref{thm:rr}.} and has a short proof (see e.g. Theorem 3.2 in~\cite{ns}).
In the asymmetric setting, we prove the colouring result for graphs in the following cases.
\begin{lem}\label{lem:maincolouring}
For all graphs $G,H_1,H_2$ with 
$m_2(H_1) \geq m_2(H_2)>1$, $m(G) \leq m_2(H_1,H_2)$ and $(H_1, H_2)$ a heart\footnote{Note that we will need the assumption that $(H_1, H_2)$ is a heart in most (but not all) cases below. In particular (i) and (vi) and, implicitly, (iii) and (v), since they each rely on one of (i) or Lemma~\ref{lem:hypergraphcolouring}(i).},
we have $G \not \to (H_1,H_2)$ if any of the following conditions are satisfied.
\begin{enumerate}
\item[(i)] We have $m_2(H_1)=m_2(H_2)$;
\item[(ii)] We have $\chi(H_2) \geq 3$;
\item[(iii)] We have $m(H_2) >2$;
\item[(iv)] We have $ar(H_2)>2$;
\item[(v)] We have $H_1=K_a$ for any $a \geq 3$;
\item[(vi)] We have $H_1=K_{a,b}$ for any $a,b \geq 2$;
\item[(vii)] There exists $\ell \in \mathbb{N}$ such that $\ell<ar(H_1)$ and $m_2(H_1,H_2) \leq \ell+\frac{1}{2}$;
\item[(viii)] There exists $\ell \in \mathbb{N}$ such that $\ell<m(H_1)$ and $m_2(H_1,H_2) \leq \ell+1$, and $H_2$ is not a cycle.
\end{enumerate}
\end{lem}
Theorem~\ref{thm:main} immediately follows from Theorem~\ref{thm:coloursuffices} combined with Lemma~\ref{lem:maincolouring} and Observation~\ref{obs:heart}.

We also prove the colouring result for hypergraphs in the following cases. 
\begin{lem}\label{lem:hypergraphcolouring}
Let $k \geq 2$. 
For all $k$-graphs $G, H_1, H_2$ with $m_k(H_1) \geq m_k(H_2)>1$, $m(G) \leq m_k(H_1,H_2)$ and $(H_1, H_2)$ a heart\footnote{We take $(H_1, H_2)$ to be a heart in order to prove case (i) below.}, 
we have $G \not \to (H_1,H_2)$ if either of the following conditions are satisfied: 
\begin{enumerate}
\item[(i)] $\chi(H_2) \geq k+1$;
\item[(ii)] $m(H_2) > \binom{v(H_1)}{k-2}+1$.
\end{enumerate}
\end{lem}
Theorem~\ref{thm:mainhyp} now immediately follows from Theorem~\ref{thm:coloursuffices} combined with Lemma~\ref{lem:hypergraphcolouring} and the equivalent of Observation~\ref{obs:heart} for hypergraphs.
Note that Lemma~\ref{lem:hypergraphcolouring} is a generalisation of cases~(ii) and~(iii) of Lemma~\ref{lem:maincolouring}.
\COMMENT{RH: Do we need the footnote here? I think best to include it to be clear. CB: I actually think we should remove the footnote... Isn't it definition rather than convention -- there is exactly one way to choose nothing, just as there is exactly one way to choose everything? RH: Agreed! I was thinking of $\binom{0}{0}$ which is less obvious.}

\subsection{Additional note} 
Shortly before submitting this paper, we learned of the simultaneous and independent work of Kuperwasser, Samotij and Wigderson~\cite{ksw}. 
They also resolve the Kohayakawa--Kreuter conjecture for the vast majority of cases. 
They too prove Theorem~\ref{thm:coloursuffices} for graphs, along with a large collection of (not entirely overlapping) colouring results similar to Lemma~\ref{lem:maincolouring}. Rather than extending their result to hypergraphs as we do, they instead extend their results to families of graphs, i.e. one aims to avoid a finite family  $\mathcal{H}_1$ of graphs in red and family $\mathcal{H}_2$ of graphs in blue. Finally, they include an interesting conjecture on graph partitioning (\cite[Conjecture 1.5]{ksw}), which if true, would prove the Kohayakawa--Kreuter conjecture in full.

\COMMENT{CB: just adding comment here so that easier to see when flicking through the draft that this bit will need updating and the `TO DO' deleting.}

\subsection{Extra additional note}
Since the first submission of this paper, Christoph, Martinsson, Steiner and Wigderson~\cite{cmsw} proved the following colouring statement: Let $H_1$, $H_2$ be graphs with $m_2(H_1) > m_2(H_2) > 1$. If a graph $G$ satisfies $G \to (H_1,H_2)$, then $m(G) > m_2(H_1, H_2)$.
This result, coupled with Theorem~\ref{thm:coloursuffices}
and Lemma~\ref{lem:maincolouring}(i), completes the resolution of the Kohayakawa--Kreuter conjecture.
In particular, they proved this result via proving a slight weakening of the aforementioned graph partitioning conjecture of Kuperwasser, Samotij and Wigderson,
along with an extra graph partitioning result.

\section{Organisation} 

The paper is organised as follows.
In Section~\ref{sec:notation} we give any remaining notation we will use and a couple of results which will be useful throughout.
In Section~\ref{sec:overview} we give an overview of the different components used to prove Theorem~\ref{thm:coloursuffices}.
This includes discussion of the relevance of -- as well as how we employ -- the central class of graphs important to this paper, $\hat{\mathcal{B}}(H_1, H_2, \eps)$ (defined formally in Section~\ref{sec:bhatdef}), together with a proof sketch of the finiteness of $\hat{\mathcal{B}}(H_1, H_2, \eps)$ (for any $\eps \geq 0$ and heart $(H_1, H_2)$).
In Section~\ref{sec:asymedgecolb}, we introduce the algorithm \textsc{Asym-Edge-Col-$\hat{\mathcal{B}}_{\eps}$}, a process that, if it does not run into an error, produces a colouring of $G^{k}_{n,p}$ which yields $G^{k}_{n,p} \not \to (H_1,H_2)$, thus reducing the proof of Theorem~\ref{thm:coloursuffices} to proving that \textsc{Asym-Edge-Col-$\hat{\mathcal{B}}_{\eps}$} does not run into an error asymptotically almost surely (a.a.s.) (see Lemma~\ref{lemma:noerror}).
In Section~\ref{sec:bhatdef}, we define $\hat{\mathcal{B}}(H_1, H_2, \eps)$ when $m_k(H_1) > m_k(H_2)$.
(The case that $m_2(H_1) = m_2(H_2)$ can be found in \cite[Appendix~\ref{app:mh1mh2equal}]{bhh_kk_arxiv}.) 
In Section~\ref{sec:grow}, we prove Lemma~\ref{lemma:noerror} using an auxiliary algorithm \textsc{Grow-$\hat{\mathcal{B}}_{\eps}(H_1, H_2, \eps)$}. 
In Section~\ref{sec:connectivity}, we prove a connectivity result for $k$-graphs that we will utilise to prove that $\hat{\mathcal{B}}(H_1, H_2, \eps)$ is finite in Section~\ref{sec:bfinite}. In Section~\ref{sec:colour}, we prove Lemmas~\ref{lem:maincolouring} and \ref{lem:hypergraphcolouring}. We then give several concluding remarks in Section~\ref{sec:conclude}.


In the arXiv version of this paper~\cite{bhh_kk_arxiv}, we include four appendices. The work included in those is sufficiently similar to previous work carried out by the third author in~\cite{h}, and so is provided separately for completeness. A guide to these appendices can be found in Section 12 of \cite{bhh_kk_arxiv}. 

\section{Notation and useful results}\label{sec:notation}

For a $k$-graph $G=(V,E)$, we denote the number of vertices in $G$ by $v(G)=v_G:=|V(G)|$ and the number of edges in $G$ by $e(G)=e_G:=|E(G)|$.
Moreover, for $k$-graphs $H_1$ and $H_2$ we let $v_1:=|V(H_1)|$, $e_1:=|E(H_1)|$, $v_2:=|V(H_2)|$ and $e_2:=|E(H_2)|$.
Since $d(\cdot)$ is used for density, we will always denote the ($1$-)degree of a vertex $u$ by $\deg(u)$. We use $\delta(H)$ for the minimum ($1$-)degree of a (hyper)graph $H$. For $k$-graphs $H_1,\dots,H_r$, we write $\delta_i:=\delta(H_i)$ for brevity.

The following result elucidates the relationship between the one and two argument $m_k$ densities.
\COMMENT{RH: Might want to reword this/change it - how often do we actually need it before the appendix? (Does it need to have a label Prop, or could it even just be a labelled equation?) JH: Yeah, this one is tricky. Its used used a couple of times deep into the appendices, but I sort-of like it here. It feels like a necessary part of explaining $k$-density and asymmetric $k$-density. It does immediately give us why `strictly balanced with respect to $d_k(\cdot, H_2)$' isn't relevant for the symmetric $k$-densities case, which I think is meaningful and clues up the reader to how the symmetric and asymmetric cases are going to be very different in their proof, even if the former is relegated almost entirely to the appendices.}
Observe that $m_k(\cdot, \cdot)$ is not symmetric in both arguments. 

\begin{prop}\label{prop:m2h1h2}
Suppose that $H_1$ and $H_2$ are non-empty $k$-graphs with $m_k(H_1) \geq m_k(H_2)$. 
Then we have \[m_k(H_1) \geq m_k(H_1, H_2) \geq m_k(H_2).\]
Moreover, \[m_k(H_1) > m_k(H_1, H_2) > m_k(H_2) \ \mbox{whenever} \ m_k(H_1) > m_k(H_2).\]
\end{prop}

Note that if $m_k(H_1) = m_k(H_2)$ and $H_1$ and $H_2$ are non-empty $k$-graphs, then $H_1$ cannot be strictly balanced with respect to $d_k(\cdot, H_2)$ unless $H_1$ is an edge. 
Indeed, otherwise, by Proposition~\ref{prop:m2h1h2} we would then have that \[m_k(H_2) = m_k(H_1, H_2) > d_k(K_k, H_2) = m_k(H_2).\] 

Generally, 
one can show that $H_1$ being (strictly) $k$-balanced implies it is (strictly) $k$-balanced with respect to $d_k(\cdot,H_2)$.
However, the converse does not hold in general, that is, the notions are not equivalent. For example, let $H_1$ be the graph formed from a copy of $C_4$ with a triangle appended onto two adjacent edges of this $C_4$, and let $H_2=C_7$. Then $H_1$ is strictly balanced with respect to $d_2(\cdot,H_2)$, but is not $2$-balanced.

%
%
The following result will be useful for us.

\begin{fact}\label{fact:ineq}
For $a,c,C \in \mathbb{R}$ and $b,d > 0$, we have 
\[(i) \ \frac{a}{b} \leq C \land \frac{c}{d} \leq C \implies \frac{a+c}{b+d} \leq C, \ \mbox{and} \ (ii) \ \frac{a}{b} \geq C \land \frac{c}{d} \geq C \implies \frac{a+c}{b+d} \geq C,\] and similarly, if also $b > d$, 
\[(iii) \ \frac{a}{b} \leq C \land \frac{c}{d} \geq C \implies \frac{a-c}{b-d} \leq C, \ \mbox{and} \ (iv) \ \frac{a}{b} \geq C \land \frac{c}{d} \leq C \implies \frac{a-c}{b-d} \geq C.\]
\end{fact}

\section{Overview of the proof of Theorem~\ref{thm:coloursuffices}}\label{sec:overview}

\subsection{A refinement of Theorem~\ref{thm:coloursuffices}}

In Section~\ref{sec:bhatdef}, for the $m_k(H_1) > m_k(H_2)$ case, we define for each heart $(H_1, H_2)$ and constant $\eps \geq 0$ a class of $k$-graphs $\hat{\mathcal{B}}(H_1, H_2, \eps)$ which will be crucial to our proof of Theorem~\ref{thm:coloursuffices}.
Indeed, Theorem~\ref{thm:coloursuffices} follows by combining the following two theorems, and using that all graphs $G \in \hat{\mathcal{B}}(H_1,H_2,\eps)$ will satisfy $m(G) \leq m_k(H_1,H_2)+\eps$ (see Claim~\ref{claim:mgbound}).
Throughout the rest of this paper, we will denote $\hat{\mathcal{B}}(H_1, H_2, 0) =: \hat{\mathcal{B}}(H_1, H_2)$ by $\hat{\mathcal{B}}$ and $\hat{\mathcal{B}}(H_1, H_2, \eps)$ by $\hat{\mathcal{B}}_{\eps}$.  

\begin{thm}\label{thm:colourb}
Let $H_1,\dots,H_r$ be $k$-graphs with $m_k(H_1) \geq \cdots \geq m_k(H_r)$ and $m_k(H_2) > 1$ and suppose $(H_1,H_2)$ is a heart. If
\COMMENT{JH: Technically, at the moment we do not define the graphs $G \in \hat{\mathcal{B}}$ to satisfy that $m(G) \leq m_k(H_1,H_2)$. In fact, nowhere in the proof need we reference $m(G)$ for $G \in \hat{\mathcal{B}}$, which seems quite weird. Ok, so the `contrapositive' is true: if $G \not\in \hat{\mathcal{B}}$ and $G$ is constructable by Grow-$\hat{\mathcal{B}}$(-Alt), then $m(G) > m_k(H_1, H_2)$. The positive statement is not true currently. We could impose this restriction on $\hat{\mathcal{B}}$, but probably that's only relevant to the colouring statement. For ease of proofs where we use the meaning of $G \in \hat{\mathcal{B}}$, I think we don't want the extra condition $m(G) \leq m_k(H_1, H_2)$, but, this is something that we could do as a refinement of $\hat{\mathcal{B}}$ later. Ultimately, I'm pretty sure we don't want the $m(G) \leq m_k(H_1, H_2)$ condition in the algorithm itself. (I'm rambling\ldots this is something we should chat about as its quite subtle.) } 
\begin{itemize}
\item $\hat{\mathcal{B}}$ is finite, and
\item for all $G \in \hat{\mathcal{B}}$, we have $G \not \to (H_1,H_2)$, \COMMENT{RH: Also might end up being easier to say there exists $\delta>0$ such that true for all $G \in \hat{\mathcal{B}}(H_1,H_2,\delta)$ just because of there not being any graphs in here which actually have density greater than $m_k(H_1,H_2)$... maybe it'll get explained later.... JH: I think we can get rid of $\eps$ from this theorem and do all the explanation of about removing $\eps$ elsewhere.}
\end{itemize}
then there exists a constant $b > 0$ such that \[\lim\limits_{n \to \infty} \mathbb{P}[G^k_{n,p} \to (H_1, \ldots, H_r)] =   
    0      \quad \text{if } p \leq bn^{-1/m_k(H_1,H_2)}.
\] 
\end{thm}

\begin{thm}\label{thm:bfinite}
For all $k$-graphs $H_1,\dots,H_r$ satisfying the conditions in Theorem~\ref{thm:colourb},
and for all $\eps\geq 0$, 
the family $\hat{\mathcal{B}}_{\eps}$ is finite.
\end{thm}
Let us speak more on this class of $k$-graphs $\hat{\mathcal{B}}_{\eps}$. It will be clear from construction and later results -- namely Claims~\ref{claim:non-degen} and Claim~\ref{claim:degenfull} when applied in the context of constructing $\hat{\mathcal{B}}_{\eps}$ -- that for any constants $\eps' > \eps \geq 0$ we have \begin{equation}\hat{\mathcal{B}}_{\eps} \subseteq \hat{\mathcal{B}}_{\eps'}.\end{equation}
Hence, given  Theorem~\ref{thm:bfinite}, we introduce the following notational definition. 
\begin{define}\label{def:eps*}
For all $k$-graphs $H_1,H_2$ satisfying the conditions in Theorem~\ref{thm:colourb}, we denote by $\eps^* = \eps^*(H_1,H_2) > 0$ any $\eps$ such that  
$\hat{\mathcal{B}}_{\eps} = \hat{\mathcal{B}}_{0} = \hat{\mathcal{B}}(H_1, H_2)$.
\end{define}
That is, throughout this paper, wherever there is an occurrence of $\eps^*$, this is simply a choice of some $\eps>0$ such that $\hat{\mathcal{B}}_{\eps} = \hat{\mathcal{B}}_{0}$. That such a value $\eps^*$ always exists follows from the proof of Theorem~\ref{thm:bfinite} and Claim~\ref{claim:mgbound}. Furthermore, the reason for introducing $\eps^*$ is that in our proof of Lemma~\ref{lemma:noerror} later (upon which Theorem~\ref{thm:colourb} relies) we will need that a parameter $\lambda := \lambda(H_1, H_2, \eps)$ associated with $\eps$ is strictly positive, which occurs only if $\eps > 0$. 
Indeed, in the proof of Claim~\ref{claim:conclusion1} we will conclude that certain $k$-graphs $F$ constructed from a finite number of $k$-graphs in $\hat{\mathcal{B}}_{\eps}$ must have $m(F) > m_k(H_1, H_2) + \eps$. If $\hat{\mathcal{B}}(H_1,H_2,\eps)$ is finite, we may then simply apply Markov's inequality to conclude that $G = G^k_{n,p}$ contains none of these $k$-graphs a.a.s. 
Therefore, thinking of $\eps^*$ as a `dummy variable' to assist with the proof, we will instead prove Theorem~\ref{thm:colourb} with $\hat{\mathcal{B}}_{\eps^*}$ in place of $\hat{\mathcal{B}}$.

We define $\hat{\mathcal{B}}(H_1, H_2,\eps)$ using one of two algorithms, \textsc{Grow-$\hat{\mathcal{B}}_{\eps}$} and \textsc{Grow-$\hat{\mathcal{B}}_{\eps}$-Alt}, the former handling the case when $m_k(H_1) > m_k(H_2)$ and the latter handling the case when $m_k(H_1) = m_k(H_2)$. 
We present algorithm \textsc{Grow-$\hat{\mathcal{B}}_{\eps}$} and describe its output in detail in Section~\ref{sec:bhatdef}. For a full description of algorithm \textsc{Grow-$\hat{\mathcal{B}}_{\eps}$-Alt} and its output please see \cite[Appendix~\ref{app:mh1mh2equal}]{bhh_kk_arxiv}.  
\COMMENT{RH: need comment on why $m_k(H_1)=m_k(H_2)$ case in appendix.}

\begin{define}\label{def:bhatdef}
For all $\eps \geq 0$ and all $k$-graphs $H_1,H_2$ satisfying the conditions in Theorem~\ref{thm:colourb},
we define $\hat{\mathcal{B}}(H_1, H_2,\eps)$ to be the set of graphs $G$ such that 
\begin{itemize}
\item If $m_k(H_1) > m_k(H_2)$, then $G$ is an output of algorithm\COMMENT{JH: Maybe better to say `procedure' or `construction procedure' here (and elsewhere).} \textsc{Grow-$\hat{\mathcal{B}}_{\eps}$}$(H_1, H_2,\eps)$, and
\item If $m_k(H_1) = m_k(H_2)$, then $G$ is an output of algorithm \textsc{Grow-$\hat{\mathcal{B}}_{\eps}$-Alt}$(H_1, H_2,\eps)$.
\end{itemize}
\end{define}

Theorem~\ref{thm:colourb} is a refinement of the work of the third author, who proved a similar result for the $k=2$ case with the family $\hat{\mathcal{A}}(H_1, H_2, \eps)$ in place of $\hat{\mathcal{B}}(H_1,H_2,\eps)$  (see Theorem~1.9 in~\cite{h}). For the interested reader, the definition of $\hat{\mathcal{A}}(H_1, H_2, \eps)$ and why in this paper we use $\hat{\mathcal{B}}(H_1, H_2, \eps)$ instead of $\hat{\mathcal{A}}(H_1, H_2, \eps)$, is discussed in \cite[Appendix~\ref{app:hydeconj}]{bhh_kk_arxiv}.

Much of the proof of Theorem~\ref{thm:colourb} goes through in the same way as in the proof of Theorem~1.9 from~\cite{h}; as such we defer many of the places where it is almost exactly the same to the appendices.

\subsection{Proof overviews of Theorems~\ref{thm:colourb} and~\ref{thm:bfinite}}
\COMMENT{RH: I've written both of these without referring to other works... I guess its important to point out that your work and cliques vs cliques did something very similar. Perhaps we should recycle some pieces of your proof sketches and include them. However we want to be brief and get to point about what is going on, and I think what I've written here does that? Think it might be very helpful to refer to technical names (and precise later lemmas) in here, without giving definitions of anything too technical.}
\COMMENT{RH: I think when referring to stuff, we only want to refer to the $m_k(H_1)>m_k(H_2)$ case, as we don't want there to be multiple pointers definitions/lemmas/etc in the appendix! (And I don't think we need a separate proof sketch for the appendix either.) JH: Yes, completely agree.}
For this section, we assume $m_k(H_1) > m_k(H_2)$. 
The ideas for the case $m_k(H_1)=m_k(H_2)$ are almost identical. We also introduce the following notion synonymous with not having the Ramsey property for two $k$-graphs and two colours, which will always be red and blue in what follows.
\begin{define}
\textnormal{Let $H_1$ and $H_2$ be non-empty $k$-graphs. We say that a $k$-graph $G$ has a \emph{valid edge-colouring for $H_1$ and $H_2$} if there exists a red/blue colouring of the edges of $G$ that does not produce a red copy of $H_1$ or a blue copy of $H_2$.}
\end{define} 

To prove Theorem~\ref{thm:colourb},
we will design an algorithm \textsc{Asym-Edge-Col-$\hat{\mathcal{B}}_{\eps^*}$}, (for any $\eps^*$ as per Definition~\ref{def:eps*} above),
which, given an input of $G^k_{n,p}$,
a.a.s. outputs a valid edge-colouring for $H_1$ and $H_2$.
In particular, it will tell us explicitly how to red/blue-colour all edges,
up to using the known existence 
of a valid edge-colouring for 
members of $\hat{\mathcal{B}}_{\eps^*}$,
\footnote{Since $\hat{\mathcal{B}}(H_1,H_2,\eps^*)$ is finite and all graphs in it have valid edge-colourings, one could of course find all $k$-graphs in $\hat{\mathcal{B}}(H_1,H_2,\eps^*)$ and  explicit red/blue-colourings for each by brute force, leading to an explicit colouring for $G^{k}_{n,p}$.}
 and will get stuck if it comes across a substructure it does not know how to colour (see Lemma~\ref{lemma:errorvalid}), which a.a.s. does not happen (see Lemma~\ref{lemma:noerror}).
Roughly, the argument goes as follows.
Suppose that an edge of $G^{k}_{n,p}$ does not lie in any copies of $H_1$;
then clearly we can colour it red and forget about it. 
A similar observation holds for edges not in any copies of $H_2$.
Building on this (by using stronger observations than just the above), 
we can colour everything up to subhypergraphs of $G^k_{n,p}$ with a certain structure.
We show that a.a.s. everything left is an edge-disjoint union of $k$-graphs from $\hat{\mathcal{B}}_{\eps^*}$, such that each copy of $H_1$ or $H_2$ lies entirely inside a $k$-graph from $\hat{\mathcal{B}}_{\eps^*}$
(see Definitions~\ref{def:avgraph} and~\ref{def:h1h2sparse}),
and so being able to colour $k$-graphs from $\hat{\mathcal{B}}_{\eps^*}$ allows us to deal with these copies of $H_1$ and $H_2$ (see Lemma~\ref{lemma:avcolour}) and so complete a valid edge-colouring of $G^k_{n,p}$.

The proof of Lemma~\ref{lemma:noerror} involves looking more carefully at the other possible structures which could remain. These are of the following form
(and are the structures which, by design, the algorithm \textsc{Asym-Edge-Col-$\hat{\mathcal{B}}_{\eps^*}$} gets stuck on): 
\begin{itemize}
\item[(i)] an edge-disjoint union of $k$-graphs from $\hat{\mathcal{B}}_{\eps^*}$, such that there exists a copy of either $H_1$ or $H_2$ which has edges in at least two of these $k$-graphs; (See Definition~\ref{def:special}.)
\item[(ii)] a pair of $k$-graphs from $\hat{\mathcal{B}}_{\eps^*}$ intersecting at one or more edges; (See Definition~\ref{def:special}.)
\item[(iii)] the output of another algorithm called \textsc{Grow}. (See Definition~\ref{def:growoutputs}.)
\end{itemize}

We show that each $k$-graph occurring in (i) or (ii) does not occur in $G^k_{n,p}$ a.a.s. (see Claim~\ref{claim:conclusion1});
finiteness of $\hat{\mathcal{B}}_{\eps^*}$ then allows us to deal with these cases, since then a union bound can be used to say all such possibilities do not appear a.a.s. 
For (iii), we show that the whole family of possible outputs of~\textsc{Grow} do not appear as subhypergraphs of $G^k_{n,p}$ a.a.s. (see Lemma~\ref{claim:conclusion2}).
Note that, as we saw in the colouring algorithm outlined above, certain edges of $G^{k}_{n,p}$ are `easy to colour'. Roughly speaking, for any $\eps \geq 0$, algorithm \textsc{Grow-$\hat{\mathcal{B}}_{\eps}$} outputs $k$-graphs that have no easy edges to colour. 
Starting from a copy of $H_1$, the algorithm \textsc{Grow-$\hat{\mathcal{B}}_{\eps}$} builds up $k$-graphs entirely by attaching in each step a finite number of copies of $H_1$ and $H_2$ onto an appropriately chosen edge. 
Very roughly speaking, in some step of this process we may increase the number of easy to colour edges by attaching a so-called \emph{fully open flower} (defined in Section \ref{sec:bhatdef}).
We can show that in each step $i$ of the algorithm, one of two possibilities occurs:
\begin{itemize}
\item[(i)] The resulting $k$-graph becomes denser\COMMENT{Roughly speaking.}, and so less likely to appear in $G^{k}_{n,p}$, and the number of fully open flowers does not decrease by too much;
\item[(ii)] The resulting $k$-graph has the same density as in the previous step.
\end{itemize}
Crucially, the number of fully open flowers strictly increases after two such iterations of type (ii) are performed in a row. Since by the end we want no edges to be easy to colour, we require that there are no fully open flowers on output. 
We also want the $k$-graph outputted to be not too dense (otherwise it will have $m$-density larger than $m_k(H_1,H_2)+\eps$).
Combining these requirements, it is easy to see that we can only have a finite number of iterations of each kind, leading to a finite number of outputs of \textsc{Grow-$\hat{\mathcal{B}}_{\eps}$} overall.
\COMMENT{RH: Not linked anything in this sketch to particular lemmas/definitions, but I do not think we need to, since it is a brief/more imprecise sketch CB: I have linked a couple of things that felt useful when I read through.}
\COMMENT{RH: again there might be some way of linking this to what is really going on/comparing the two grow algorithms...not sure. JH: I think this is sufficient. I'm also not sure if much more needs to be said for the \textsc{Grow} algorithms than that \textsc{Grow-$\hat{\mathcal{B}}_{\eps}$} produces all the possible graphs that \textsc{Grow} can get stuck on as well as every graph that is a building block in one of the special cases. These are the bad graphs (we could even say `B' for `bad').}

\section{Setup to prove Theorem~\ref{thm:colourb}}\label{sec:asymedgecolb}

Assume the conditions in Theorem~\ref{thm:colourb} hold, that is, $(H_1, H_2)$ is a heart, $m_k(H_1) \geq m_k(H_2) > 1$ and the two conditions listed in the bullet points hold. 
All of the definitions which follow in this section through to Section~\ref{sec:grow} can be defined for any $\eps := \eps(H_1, H_2) \geq 0$, and so we state them in that way. 
For each result, we make it clear whether any $\eps \geq 0$ works, or if only $\eps^*$ as given by Definition~\ref{def:eps*} works.
\COMMENT{RH: I want to be explicit about this; hopefully this is clear? JH: Seems clear to me.}

Let us now begin our setup for proving Theorem~\ref{thm:colourb}. Firstly, we need the following definitions relating to copies of $k$-graphs from $\hat{\mathcal{B}}_{\eps}$ that appear in some $k$-graph $G$. (These are analogous to definitions relating to $\hat{\mathcal{A}}$ in \cite{h}.)  

For any $k$-graph $G$, we define \[\mathcal{S}^B_G := \{S \subseteq G: S \cong B \in \hat{\mathcal{B}}_{\eps} \land \nexists S' \supset S \ \mbox{with} \ S' \subseteq G, \ S' \cong B' \in \hat{\mathcal{B}}_{\eps}\},\] that is, the family $\mathcal{S}^B_G$ contains all maximal subhypergraphs of $G$ isomorphic to a member of $\hat{\mathcal{B}}_{\eps}$. 
Hence, there are no two members $S_1, S_2 \in \mathcal{S}^B_G$ such that $S_1 \subsetneq S_2$.
For any edge $e \in E(G)$, let \[\mathcal{S}^B_G(e) := \{S \in \mathcal{S}^B_G : e \in E(S)\}.\]

\begin{define}\label{def:avgraph}
    We call $G$ a \emph{$\hat{\mathcal{B}}_{\eps}$-graph} if, for all $e \in E(G)$, we have \[|\mathcal{S}^B_G(e)| = 1.\] 
\end{define}
In particular, a $\hat{\mathcal{B}}_{\eps}$-graph is an edge-disjoint union of $k$-graphs from $\hat{\mathcal{B}}_{\eps}$.
For a $\hat{\mathcal{B}}_{\eps}$-graph $G$, 
\COMMENT{RH: Previously looked like definition was only for $\hat{\mathcal{B}}_{\eps}$ graphs only. But it is for all $G$ technically. JH: I think I disagree. It makes the sentence not true as $G$ could have copies of $H_1$ or $H_2$ that have edges that are not in any $k$-graph from $\mathcal{S}^B_G$. Moreover, it is precisely this observation that it used to choose the edge at the start of Grow that is not contained in any $k$-graph from $\mathcal{S}^B_G$. The sparsity only matters if you're a $\hat{\mathcal{B}}_{\eps}$-graph, as not being sparse is covered by special, as is the case when you're not a $\hat{\mathcal{B}}_{\eps}$-graph and the reason for that is there's some edge in two different $k$-graphs from $\mathcal{S}^B_G$.}
a copy of $H_1$ or $H_2$ can be a subhypergraph of $G$ in two particular ways: 
either it is a subhypergraph of some $S \in \mathcal{S}^B_G$ or it is a subhypergraph with edges in at least two different $k$-graphs from $\mathcal{S}^B_G$. 
The former we call \emph{trivial} copies of $H_1$ and $H_2$, and we define \[\mathcal{T}^B_G := \left\{T \subseteq G : (T \cong H_1 \vee T \cong H_2) \land \left|\bigcup_{\substack{e \in E(T)}}\mathcal{S}^B_G(e)\right| \geq 2\right\}\] to be the family of all \emph{non-trivial} copies of $H_1$ and $H_2$ in $G$.
\begin{define}\label{def:h1h2sparse}
    We say that a $k$-graph $G$ is \emph{$(H_1, H_2)$-sparse} if $\mathcal{T}^B_G = \emptyset$.
\end{define} 

Our next lemma asserts that $(H_1, H_2)$-sparse $\hat{\mathcal{B}}_{\eps}$-graphs are easily colourable, provided that there exists a valid edge-colouring for $H_1$ and $H_2$ for all graphs $G \in \hat{\mathcal{B}}_{\eps}$.

\begin{lem}\label{lemma:avcolour}
Let $\eps\geq 0$. 
Suppose that for all $B \in \hat{\mathcal{B}}_{\eps}$, we have $B \not \to (H_1,H_2)$.
Then there exists a procedure \textsc{B-Colour($G,\eps$)} that returns for any $(H_1, H_2)$-sparse $\hat{\mathcal{B}}_{\eps}$-graph $G$ a valid edge-colouring for $H_1$ and $H_2$.
\end{lem}

\begin{proof} 
There exists a valid edge-colouring for $H_1$ and $H_2$ of every $B \in \hat{\mathcal{B}}_{\eps}$. Using this we define a procedure \textsc{B-Colour}$(G,\eps)$ as follows:
Assign a valid edge-colouring for $H_1$ and $H_2$ to every subhypergraph $S \in \mathcal{S}^B_G$ locally, that is, regardless of the structure of $G$. 
Since $G$ is an $(H_1, H_2)$-sparse $\hat{\mathcal{B}}_{\eps}$-graph, we assign a colour to each edge of $G$ without producing a red copy of $H_1$ or a blue copy of $H_2$, and the resulting colouring is a valid edge-colouring for $H_1$ and $H_2$ of $G$.\end{proof}

In particular, the conditions of Lemma~\ref{lemma:avcolour} are met when $\eps:=\eps^*$, by the second condition in Theorem~\ref{thm:colourb} and the definition of $\eps^*$.
Note that we did not use that $\hat{\mathcal{B}}_{\eps^*}$ is finite  (as given by the other condition in Theorem~\ref{thm:colourb}) in our proof of Lemma~\ref{lemma:avcolour}, only that every $k$-graph in $\hat{\mathcal{B}}_{\eps^*}$ has a valid edge-colouring for $H_1$ and $H_2$. 
The finiteness of $\hat{\mathcal{B}}_{\eps^*}$ will be essential later for the proofs of Claims~\ref{claim:conclusion1} and \ref{claim:conclusionsym}. 

We now need a hypergraph version of algorithm \textsc{Asym-Edge-Col} from \cite{h}. (See Figure \ref{asymedgecolfig} for the algorithm and Section \ref{subsec_alg_det1} for a detailed description.)
The algorithm \textsc{Asym-Edge-Col-$\hat{\mathcal{B}}_{\eps}$} below works in essentially the same way as \textsc{Asym-Edge-Col} does in \cite{h}.
For this algorithm, we need the following definition.
For any $k$-graph $G$ we define 
the families \[\mathcal{R}_G := \{R \subseteq G : R \cong H_1\} \ \mbox{and}\ \mathcal{L}_G := \{L \subseteq G : L \cong H_2\}\] of all copies of $H_1$ and $H_2$ in $G$, respectively. 

\begin{figure} 
\begin{algorithmic}[1]
\Procedure{\sc Asym-Edge-Col-$\hat{\mathcal{B}}_{\eps}$}{$G = (V,E), \eps$}
    \State $s\gets$ {\sc empty-stack}()
    \State $E'\gets E$
    \State $\mathcal{L}\gets \mathcal{L}_G$
    \While{$G' = (V, E')$ is not an $(H_1, H_2)$-sparse  $\hat{\mathcal{B}}_{\eps}$-graph}\label{line:while1start}
            \If{$\exists e \in E' \ \mbox{s.t.} \ \nexists (L,R) \in \mathcal{L} \times \mathcal{R}_{G'}: E(L) \cap E(R) = \{e\}$}\label{line:edgeremoval}
                \ForAll {$L \in \mathcal{L}: e \in E(L)$}\label{line:Lall1}
                    \State $s$.{\sc push}($L$)\label{line:Lpush1}
                \State $\mathcal{L}$.{\sc remove}($L$) 
                \EndFor 
            \State $s$.{\sc push}($e$)
            \State $E'$.{\sc remove}($e$)\label{line:edgeremove}
            \Else
                \If{$\exists L \in \mathcal{L}: \ \exists e \in E(L) \ \mbox{s.t.} \ \nexists R \in \mathcal{R}_{G'} \ \mbox{with} \ E(L) \cap E(R) = \{e\} $}\label{line:L*check}
                    \State $s$.{\sc push}($L$)\label{line:Lpush2}
                    \State $\mathcal{L}$.{\sc remove}($L$)
                \Else
                    \State {\bf error} {``stuck"}\label{line:error}
                \EndIf
            \EndIf
    \EndWhile\label{line:while1end}
    \State {\sc B-colour}($G' = (V,E'),\eps$)\label{line:acolourcall}
    \While{$s \neq \emptyset$}\label{line:while2start}
        \If{$s$.{\sc top}() is an edge}
            \State $e \gets s$.{\sc pop}()\label{line:edgepop}
            \State $E'$.{\sc add}($e$)
            \State $e$.{\sc set-colour}(blue)
        \Else
            \State $L \gets s$.{\sc pop}()\label{line:Lpop}
            \If{$L$ is entirely blue}\label{line:colourswapstart}
                \State $f \gets$ any $e \in E(L)$ s.t. $\nexists R \in \mathcal{R}_{G'}: E(L) \cap E(R) = \{e\}$\label{line:getf}
                \State $f$.{\sc set-colour}(red)\label{line:colourswapend}
                
            \EndIf
        \EndIf
    \EndWhile \label{line:while2end}
\EndProcedure
\end{algorithmic}
\caption{The implementation of algorithm \textsc{Asym-Edge-Col-$\hat{\mathcal{B}}_{\eps}$}.}\label{asymedgecolfig}
\end{figure}

Let us describe the algorithm \textsc{Asym-Edge-Col-$\hat{\mathcal{B}}_{\eps}$} which if successful outputs a valid edge-colouring of $G$. 
In \textsc{Asym-Edge-Col-$\hat{\mathcal{B}}_{\eps}$}, edges are removed from and then inserted back into a working copy $G'~=~(V, E')$ of $G$.
Each edge is removed in the first while-loop only when it is not the unique intersection of the edge sets of some copy of $H_1$ and some copy of $H_2$ in $G'$ (line~\ref{line:edgeremoval}). 
It is then `pushed\footnote{For clarity, by `push' we mean that the object is placed on the top of the stack $s$.}' onto a stack $s$ such that when we reinsert edges (in reverse order) in the second while-loop we can colour them to construct a valid edge-colouring for $H_1$ and $H_2$ of $G$;
if at any point $G'$ is an $(H_1, H_2)$-sparse $\hat{\mathcal{B}}_{\eps}$-graph, then we combine the colouring of these edges with a valid edge-colouring for $H_1$ and $H_2$ of $G'$ provided by \textsc{B-Colour$(G,\eps)$}. 
We also keep track of the copies of $H_2$ in $G$ and push abstract representations of some of them (or all of them if $G'$ is never an $(H_1, H_2)$-sparse $\hat{\mathcal{B}}_{\eps}$-graph during \textsc{Asym-Edge-Col-$\hat{\mathcal{B}}_{\eps}$}) onto $s$ (lines~\ref{line:Lpush1} and \ref{line:Lpush2}) to be used later in the colour swapping stage of the second while-loop (lines~\ref{line:colourswapstart}-\ref{line:colourswapend}). 

The following two lemmas combined prove Theorem~\ref{thm:colourb}.
The first confirms that our colouring process in the second while-loop of \textsc{Asym-Edge-Col-$\hat{\mathcal{B}}_{\eps^*}$}
produces a valid edge-colouring for $H_1$ and $H_2$ of $G$. The proof can be found in \cite[Appendix~\ref{sec:omittedproofs}]{bhh_kk_arxiv}. The second assures us that it enters this loop, thereby succeeding, a.a.s.

\begin{lem}\label{lemma:errorvalid}
Algorithm \textsc{Asym-Edge-Col-$\hat{\mathcal{B}}_{\eps^*}$} either terminates with an error in line~\ref{line:error} or finds a valid edge-colouring for $H_1$ and $H_2$ of $G$.
\end{lem}



\begin{lem}\label{lemma:noerror}
There exists a constant $b = b(H_1,H_2) > 0$ such that for $p \leq bn^{-1/m_k(H_1,H_2)}$ algorithm \textsc{Asym-Edge-Col-$\hat{\mathcal{B}}_{\eps^*}$} terminates on $G^k_{n,p}$ without error a.a.s.
\end{lem}
The proof of Lemma~\ref{lemma:noerror} in the case when $m_k(H_1) > m_k(H_2)$ is given in Section~\ref{sec:grow}. The proof in the case that $m_k(H_1) = m_k(H_2)$ is given in \cite[Appendix~\ref{app:mh1mh2equal}]{bhh_kk_arxiv}. The proof of the latter is similar to the former, but requires a slightly different definition $\hat{\mathcal{B}}_{\eps}$, resulting from a slightly different algorithm in this case. 

\subsection{Detailed description of \textsc{Asym-Edge-Col-$\hat{\mathcal{B}}_{\eps}$}} \label{subsec_alg_det1}
We finish this section by describing \textsc{Asym-Edge-Col-$\hat{\mathcal{B}}_{\eps}$} in detail. In line~\ref{line:while1start}, we check whether $G'$ is an $(H_1, H_2)$-sparse $\hat{\mathcal{B}}_{\eps}$-graph.
If not, then we enter the first while-loop. In line~\ref{line:edgeremoval}, we choose an edge $e$ which is not the unique intersection of the edge sets of some copy of $H_1$ and some copy of $H_2$ in $G'$ (if such an edge $e$ exists).
Then in lines~\ref{line:Lall1}-\ref{line:edgeremove} we push each copy of $H_2$ in $G'$ that contains $e$ onto $s$ before pushing $e$ onto $s$ as well.
Now, if every edge $e \in E'$ is the unique intersection of the edge sets of some copy of $H_1$ and some copy of $H_2$ in $G'$,
then we push onto $s$ a copy $L$ of $H_2$ in $G'$ which contains an edge that is not the unique intersection of the edge set of $L$ and the edge set of some copy of $H_1$ in $G'$.
If no such copies $L$ of $H_2$ exist, then the algorithm has an error in line~\ref{line:error}.
If \textsc{Asym-Edge-Col-$\hat{\mathcal{B}}_{\eps}$} does not run into an error, then we enter the second while-loop with input $G'$.
Observe that $G'$ is either the empty $k$-graph on vertex set $V$ or some $(H_1, H_2)$-sparse $\hat{\mathcal{B}}_{\eps}$-graph.
Now for $\eps:=\eps^*$,  Lemma~\ref{lemma:avcolour} asserts $G'$ has a valid edge-colouring for $H_1$ and $H_2$.
The second while-loop successively removes edges (line~\ref{line:edgepop}) and copies of $L$ (line~\ref{line:Lpop}) from $s$ in the reverse order in which they were added onto $s$, with the edges added back into $E'$.
Each time an edge is added back it is coloured blue, and if a monochromatic blue copy $L$ of $H_2$ is constructed, we make one of the edges of $L$ red (lines~\ref{line:colourswapstart}-\ref{line:colourswapend}).
This colouring process is then repeated until we have a valid edge-colouring for $H_1$ and $H_2$ of $G$.

\section{Definition of \texorpdfstring{$\hat{\mathcal{B}}_{\eps}$: $m_k(H_1) > m_k(H_2)$}{B family, strict inequality} case}\label{sec:bhatdef}

Let $\eps \geq 0$. We now introduce \textsc{Grow-$\hat{\mathcal{B}}_{\eps}$}, whose outputs form the set $\mathcal{B}(H_1, H_2, \eps)$.
Recall that \textsc{Grow-$\hat{\mathcal{B}}_{\eps}$} is used when $m_k(H_1) > m_k(H_2)$. In order to state \textsc{Grow-$\hat{\mathcal{B}}_{\eps}$} we require several definitions. For any $k$-graph $F$, we define \[\lambda(F) := v(F) - \frac{e(F)}{m_k(H_1,H_2)}.\] This definition is motivated by the fact that the expected number of copies of $F$ in $G^k_{n,p}$ with $p = bn^{-1/m_k(H_1,H_2)}$ has order of magnitude \[n^{v(F)}p^{e(F)} = b^{e(F)}n^{\lambda(F)}.\] We define the following sets of $k$-graphs in a $k$-graph $G$. 
\COMMENT{RH: moved $\mathcal{L}_G$ and $\mathcal{R}_G$ up. Wondering if it is necessary to highlight them in a definition? I find it slightly hard to read in sense that it is not italics so not sure when it ends!}
For any $k$-graph $G$ we define 
\[\mathcal{L}^*_G := \{L \in \mathcal{L}_G : \forall e \in E(L) \ \exists R \in \mathcal{R}_G \ \mbox{s.t.} \ E(L) \cap E(R) = \{e\}\} \subseteq \mathcal{L}_G,\] the family of copies $L$ of $H_2$ in $G$ with the property that for every edge $e$ in $L$ there exists a copy $R$ of $H_1$ such that the edge sets of $L$ and $R$ intersect uniquely at $e$; 
\[\mathcal{C} = \mathcal{C}(H_1, H_2) := \{G = (V,E) : \forall e \in E\ \exists(L,R) \in \mathcal{L}_G \times \mathcal{R}_G\ \mbox{s.t.}\ E(L) \cap E(R) = \{e\}\},\] the family of $k$-graphs $G$ where every edge is the unique edge-intersection of some copy $L$ of $H_2$ and some copy $R$ of $H_1$;
and 
\[\mathcal{C}^* = \mathcal{C}^*(H_1, H_2) := \{G = (V,E): \forall e \in E \ \exists L \in \mathcal{L}^*_{G} \ \mbox{s.t.} \ e \in E(L)\},\] the family of $k$-graphs $G$ where every edge is contained in a copy $L$ of $H_2$ which has at each edge $e$ some copy $R$ of $H_1$ attached such that $E(L) \cap E(R) = \{e\}$.

Define \begin{equation}\label{eq:gamma}\gamma = \gamma(H_1,H_2,\eps) := \frac{1}{m_k(H_1, H_2)} - \frac{1}{m_k(H_1, H_2) + \eps} \geq 0.\end{equation}

Let $F$ be a $k$-graph. We define an edge $f$ to be \emph{open in $F$} if there exists no $L \in \mathcal{L}^*_{G}$ such that $f \in E(L)$.
Now attach a copy $L$ of $H_2$ to $F$ at some edge $e \in E(F)$ such that $E(L) \cap E(F) = e$ and $V(L) \cap V(F) = \cup_{v\in e} v$.
Further, attach a copy $L_{e'}$ of $H_1$ at each edge $e' \in E(R)\setminus e$ such that $E(R_{e'}) \cap (E(R)\cup E(H)) = e'$, $V(L_{e'}) \cap (V(R)\cup V(F)) =  \cup_{v\in e'} v$ and for each pair of edges $e', e'' \in  E(R)\setminus e$ we have $V(R_{e'}) \cap V(R_{e''}) = e' \cap e''$.
That is, the copies of $H_1$ attached are in some sense `pairwise-disjoint'.
Call the graph constructed in this process $F'$.
We call the $k$-graph induced by the collection of edges in $J := \cup_{e'\in R_{e'}} E(R_{e'}) \cup e$ a \emph{fully open flower} and the edges belonging to $\cup_{e'\in R_{e'}} E(R_{e'})\setminus e'$ the \emph{petal edges of $J$}.
If the construction of $F'$ from $F$ is different at all from the above, that is, if $J$ intersects $F$ in more than one edge or more than $k$ vertices, then $J$ is not a fully open flower.
Later, in Section~\ref{sec:bfinite}, we will see that the petal edges are in fact all open, hence the name `fully open flower'.
For now, all we need observe is that the construction process of $F'$ yields that one can tell by inspection whether an edge belongs
to the petal edges of some fully open flower or not without knowing anything about how $F$ and $F'$ were constructed.
This may seem enigmatic currently, but this observation will allow us to define a function \textsc{Eligible-Edge}, 
used in algorithm \textsc{Grow-$\hat{\mathcal{B}}_{\eps}$} below and 
\textsc{Grow}$=$\textsc{Grow($G', \eps,n$)}
(see its definition in Section~\ref{sec:grow}), 
in a way that does not increase the number of outputs of either algorithm.\COMMENT{JH: Not sure how happy I am with this paragraph.}\COMMENT{JH: I wonder how fast \textsc{Eligible-Edge} runs? In the previous paper this didn't matter as the version of \textsc{Eligible-Edge} there was only used in the theoretical algorithm \textsc{Grow}, but here we want to say something like `plug this algorithm into a computer and use it'. I think this \textsc{Eligible-Edge}, to guarantee there are no open edges that do not belong to fully open flowers, would need to preform a lot of checks. It would still be polynomial. I guess some of this can be alleviated by keeping track of open edges and closed edges, and what type of open edges there are available in each iteration. But still, seems like a lot. Perhaps a section in the concluding remarks discussing how one could implement \textsc{Grow-B} would be valuable, rather than just claiming such a polynomial time algorithm exists (which is what I was thinking previously). Maybe we could even work out the specific order? Of course, we've got limited time, and this could well be introduced after a version is put on arXiv.}

\begin{figure}
\begin{algorithmic}[1]
\Procedure{Grow-$\hat{\mathcal{B}}_{\eps}$}{$H_1, H_2, \eps$}
    \State $F_0 \gets H_1$\label{line:growanyrV}
    \State $i \gets 0$\label{line:growi0V}
    \While {$\forall \tilde{F} \subseteq F_i: \lambda(\tilde{F}) > -\gamma$}\label{line:growwhileconditionV}
        \If {$\exists e \in E(F_i)$ s.t. $e$ is open in $F_i$}\label{line:growifopeneV}
            \State \textbf{either} choose $k$ vertices $v_1, \ldots, v_k \in V(F_i)$\label{line:vwopen}
                \State \ \ \ \ \ $F_{i+1} \gets$ \textsc{Attach-$H_1(F_i, v_1 \ldots, v_k)$}\label{line:attachh1open}
                \State \ \ \ \ \ $i \gets i + 1$\label{line:h1openito1+1}
            \State \textbf{or}   $e \gets \textsc{Eligible-Edge}(F_i)$\label{line:aveligible-edge}
                \State \ \ \ \ \ $F_{i+1} \gets$ \textsc{Close-$e$}($F_i, e$)\label{line:F_i+1getsextendV}
                \State \ \ \ \ \ $i \gets i + 1$ \label{line:eopenito1+1}
           \State \textbf{end either}
        \Else 
             \State \Return{$F_i$}\label{line:returnfi}
             \State \textbf{either} choose $k$ vertices $v_1, \ldots, v_k \in V(F_i)$ \label{line:vwclosed}
                \State \ \ \ \ \ $F_{i+1} \gets$ \textsc{Attach-$H_1(F_i, v_1, \ldots, v_k)$}
                \State \ \ \ \ \ $i \gets i + 1$\label{line:h1closeditoi+1}
             \State \textbf{or} $e \gets$ any edge of $F_i$ \label{line:growifclosedeV}
                \State \ \ \ \ \ $F_{i+1} \gets$ \textsc{Close-$e$-closededge}($F_i, e$)
                \State \ \ \ \ \ $i \gets i + 1$\label{line:ecloseditoi+1}
            \State \textbf{end either}
        \EndIf
    \EndWhile
\EndProcedure
\end{algorithmic}\smallskip\smallskip

\begin{algorithmic}[1]
\Procedure{Close-$e$}{$F,e$}
    \State $L \gets$ a copy of $H_2$ containing $e$ such that for all $e' \in E(L)\setminus E(F)$ we have $|e' \cap V(F)| < k$ \label{line:copyofh2V}
    \State $F' \gets F \cup L$
        \ForAll{$e' \in E(L)\setminus E(F)$}
                \State $R_{e'} \gets$ a copy $R$ of $H_1$ such that $E(L) \cap E(R) = \{e'\}$ and $|V(R) \cap V(F)| \leq k-1$. \label{line:copyofh1V}
                \State $F' \gets F' \cup R_{e'}$
        \EndFor
        \State \Return {$F'$}
\EndProcedure 
\end{algorithmic}\smallskip\smallskip

\begin{algorithmic}[1]
\Procedure{Close-$e$-closededge}{$F,e$}
    \State Perform procedure \textsc{Close-$e$}($F, e$), but ensure that at least one new edge is added.\label{line:closededgeclosee}
\EndProcedure 
\end{algorithmic}\smallskip\smallskip

\begin{algorithmic}[1]
\Procedure{Attach-$H_1$}{$F, v_1, \ldots, v_k$}
    \State $R \gets$ a copy of $H_1$ containing $v_1, \ldots, v_k$ such that $E(R) \nsubseteq E(F)$
    \State $F' \gets F \cup R$.
    \State \Return{$F'$}
\EndProcedure 
\end{algorithmic}
\caption{The implementation of algorithm \textsc{Grow-$\hat{\mathcal{B}}_{\eps}$}.}\label{fig:growb}
\end{figure}

Let us describe procedure \textsc{Grow-$\hat{\mathcal{B}}_{\eps}$}, given in Figure~\ref{fig:growb} (note that a detailed description is given in Section~\ref{sec:detaileddescriptiongrowb}).
Firstly, observe that \textsc{Grow-$\hat{\mathcal{B}}_{\eps}$} takes as input two $k$-graphs $H_1$ and $H_2$, with properties as described in Definition~\ref{def:bhatdef}, and a constant $\eps \geq 0$.
That is, there is no underlying $k$-graph determining at step $i$ what the possible structure of $F_i$ could be, as will be the case in \textsc{Grow} (in Section~\ref{sec:grow}) later. 
Thus whenever we have a choice in step $i$ of how to attach a $k$-graph we must consider every possible structure $F_{i+1}$ that could possibly result. 
Utilising Claims~\ref{claim:non-degen} and \ref{claim:degenfull},  
we will see that the majority of these choices result in making $F_{i+1}$ less likely to appear in $G^k_{n,p}$ than $F_{i}$.
This is the reason behind line~\ref{line:growwhileconditionV}:
While $\lambda(F) \geq 0$ for all subhypergraphs $F \subseteq F_i$, it is possible that $F_i$ appears in $G^k_{n,p}$ with positive probability, however once $\lambda(F) \leq -\gamma$ for some subhypergraph $F \subseteq F_i$,
we can show immediately using Markov's inequality that $F_i$ does not exist in $G^k_{n,p}$ a.a.s.

The outputs $F$ of \textsc{Grow-$\hat{\mathcal{B}}_{\eps}$} have several properties that will be useful to us. 
Firstly, all their edges are \emph{not} open; that is, $F \in \mathcal{C}^{*}(H_1 H_2)$. 
Secondly, we will see later that whatever the input $G$ is to algorithm \textsc{Grow} (see Section~\ref{sec:grow} and Section~\ref{sec:guidetoapp})\COMMENT{JH: I'm thinking that eventually we'll have a section that lays out what small changes we make in the Asym-Edge-Col, Grow, etc parts that are in the appendices.}, every $k$-graph $\hat{G}$ constructable by \textsc{Grow} 
 is constructable using \textsc{Grow-$\hat{\mathcal{B}}_{\eps}$}.
Notably, we will see that \textsc{Grow} cannot construct any $k$-graph which is an output of \textsc{Grow-$\hat{\mathcal{B}}_{\eps}$}. This will be particularly useful in the proofs of Claims~\ref{claim:growv} and \ref{claim:conclusion1}.
\COMMENT{JH: These two results are the ones where the proofs have to change to accommodate the new Eligible-Edge definition.
The first is the result which says that \textsc{Grow} always adds an edge in each iteration, and the second is the beginning of the proof of the final 'whp $G^k_{n,p}$ doesn't contain any of the $k$-graphs that are constructed by \textsc{Grow}' result.}
In some sense, \textsc{Grow-$\hat{\mathcal{B}}_{\eps}$} outputs those $k$-graphs which the while-loop of \textsc{Grow} could get stuck on and which are the building blocks for those subhypergraphs of $G^k_{n,p}$ which one cannot trivially colour, using the while-loop of \textsc{Asym-Edge-Col-$\hat{\mathcal{B}}_{\eps}$} or otherwise.
\COMMENT{JH: Come back to this at some point to write properly something about why $\eps = 0$ can be `assumed'/`used' when thinking about \textsc{Asym-Edge-Col}, referring back to the discussion in Section~\ref{sec:overview}. Future JH: Nah, this is wrong. \textsc{Asym-Edge-Col} needs to feed Grow something that's using $\eps > 0$ so that $\gamma > 0$ and we can use this to apply Markov's inequality correctly e.g. in Claim~\ref{claim:conclusion1}.}
\COMMENT{JH: I'm not sure how happy I am referring so much to \textsc{Grow($G$)} here. But this is crucial to understanding why algorithm \textsc{Grow-$\hat{\mathcal{B}}$} looks like it does. Perhaps some of this could be moved down to where \textsc{Grow} is.}
Thirdly, we make the following claim.
\begin{claim}\label{claim:mgbound}
Let $F$ be an output of \textsc{Grow-$\hat{\mathcal{B}}_{\eps}$}. Then  $m(F) \leq m_k(H_1, H_2) + \eps$. 
\end{claim}
\begin{proof}
Assume for a contradiction that there is some non-empty subgraph $F' \subseteq F$ such that $d(F') > m_k(H_1, H_2) + \eps$. Since $F$ was an output of \textsc{Grow-$\hat{\mathcal{B}}_{\eps}$}, we must have $\lambda(F')>-\gamma$. 
Combining these two facts yields
\begin{align*}
\frac{e(F')}{m_k(H_1,H_2)+\eps}-\frac{e(F')}{m_k(H_1,H_2)} & > v(F')-\frac{e(F')}{m_k(H_1,H_2)}  = \lambda(F') \\
& > -\gamma = \frac{1}{m_k(H_1, H_2) + \eps} - \frac{1}{m_k(H_1, H_2)},
\end{align*} 
which, if $\eps>0$, rearranges to say $e(F')<1$, a contradiction. If $\eps=0$, we obtain the contradiction $0>0$. 
\COMMENT{RH: Actually, we used two strict equalities here, but we only needed one... so actually does this mean we could claim $m(F)<m_k(H_1,H_2) + \eps$?? JH: Yes, but not for $\eps = 0$, as you'd be dividing by $0$. Hence I don't think saying $m(F)<m_k(H_1,H_2) + \eps$ in the claim would be useful, even if it is true for $\eps > 0$. RH: good point! Actually need to be slightly careful with $\eps=0$ case: added extra line.}
\end{proof}

\subsection{Detailed description of \textsc{Grow-$\hat{\mathcal{B}}_{\eps}$}}\label{sec:detaileddescriptiongrowb}
We finish the section by describing \textsc{Grow-$\hat{\mathcal{B}}_{\eps}$} in detail. In lines~\ref{line:growanyrV} and \ref{line:growi0V} we initialise the $k$-graph as a copy of $H_1$ and initialise $i$ as $0$.
We have $\lambda(H_1) = v_1 - e_1/m_k(H_1, H_2) = k - \frac{1}{m_k(H_2)} > 0$, since $H_1$ is (strictly) balanced with respect to $d_k(\cdot, H_2)$ and $m_k(H_2) > 1$, so we enter the while-loop on line~\ref{line:growwhileconditionV}. 
In the $i$th iteration of the while-loop, we first check in line~\ref{line:growifopeneV} whether there are any open edges in $F_i$.
If not, then every edge in $F_i$ is not open and we output this $k$-graph in line~\ref{line:returnfi};
if so, then in each iteration the algorithm makes a choice (arbitrarily) to do one of two operations. 
The first, in lines~\ref{line:vwopen}-\ref{line:h1openito1+1}, is to take $k$ vertices of $F_i$ (again, chosen arbitrarily) in line~\ref{line:vwopen} and perform \textsc{Attach-$H_1$} in line~\ref{line:attachh1open}, 
which takes a copy $R$ of $H_1$ and attaches it to $F_i$ such that $R$ contains $v_1, \ldots, v_k$ and also such that $F_{i+1}$ has at least one new edge compared to $F_i$.
Note that this copy $R$ of $H_1$ could possibly intersect anywhere in $F_{i}$. 
In line~\ref{line:h1openito1+1}, we iterate $i$ for the next iteration of the while-loop.

The second operation, in lines~\ref{line:aveligible-edge}-\ref{line:eopenito1+1}, 
invokes a procedure \textsc{Eligible-Edge} which functions as follows: for a $k$-graph $F$, 
\textsc{Eligible-Edge($F$)} returns a specific open edge $e \in E(F)$ that does not belong to a fully open flower or, if no such edge exists, just returns a specific edge $e \in E(F)$ which is open. \textsc{Eligible-Edge} selects this edge $e$ to be \textit{unique up to isomorphism of $F_i$}, that is, for any two isomorphic $k$-graphs $F$ and $F'$, there exists an isomorphism $\phi$ with $\phi(F) = F'$ such that \[\phi(\textsc{Eligible-Edge}(F)) = \textsc{Eligible-Edge}(F').\] \textsc{Eligible-Edge} can make sure it always takes the same specific edge of $F$ regardless of how it was constructed using a theoretical look-up table of all possible $k$-graphs,
that is, \textsc{Eligible-Edge} itself does not contribute in anyway to the number of outputs of \textsc{Grow-$\hat{\mathcal{B}}_{\eps}$}.
As observed earlier, whether an edge in $F$ is in a fully open flower or not is verifiable by inspection, hence \textsc{Eligible-Edge} is well-defined.\COMMENT{JH: Need I say more here?}\COMMENT{A reader familiar with \cite{h} will observe that this definition of \textsc{Eligible-Edge} differs from that the procedure of the same name in \cite{h}. Indeed, this is a subtle change we need to introduce, and will discuss this more in Section~\ref{sec:guidetoapp}. JH: This is very specific and I'm not sure any reader would care, so maybe this footnote should just be removed?}
In line~\ref{line:F_i+1getsextendV}, we invoke procedure \textsc{Close-$e$($F_i, e$)}, using the edge we just selected with \textsc{Eligible-Edge}.
Procedure \textsc{Close-$e$($F_i, e$)} attaches firstly a copy $L$ of $H_2$ at this edge, possibly intersecting over vertices of $F_i$,
but making sure that no edge of $L$ that is new in $F_{i+1}$ contains $k$ vertices of $F_i$. Then for each edge $e'$ in $L$ that was not already in $F$,
\textsc{Close-$e$($F_i, e$)} attaches a copy $R_{e'}$ of $H_1$ so that the edge-intersection of $R_{e'}$ with $L$ is precisely the edge $e'$ and the vertex-intersection of $R_{e'}$ and $F$ is at most $k-1$.
Note that $e$ was open when \textsc{Eligible-Edge} chose it, and hence $F_{i+1}$ must have at least one more edge than $F_i$.
In line~\ref{line:eopenito1+1}, we iterate $i$ in preparation for the next iteration of the while-loop. 
Overall, observe that procedure \textsc{Close-$e$} acts in such a way that no copy of $H_1$ it appends could have been appended in lines~\ref{line:vwopen}-\ref{line:h1openito1+1} (or lines~\ref{line:vwclosed}-\ref{line:h1closeditoi+1}).
This strong distinguishing between \textsc{Attach-$H_1$} and \textsc{Close-$e$} is primarily designed to replicate how the while loop of \textsc{Grow} functions later (see Section~\ref{sec:grow}) and will be essential in arguing that \textsc{Grow} always has an open edge at the start of each iteration of its while-loop.

If we enter line~\ref{line:returnfi}, then afterwards we again do one of two operations. The first, in lines~\ref{line:vwclosed}-\ref{line:h1closeditoi+1}, is identical to that in lines~\ref{line:vwopen}-\ref{line:h1openito1+1}. The second, in lines~\ref{line:growifclosedeV}-\ref{line:ecloseditoi+1}, is almost identical to that in lines~\ref{line:aveligible-edge}-\ref{line:eopenito1+1}, except we make sure that at least one new edge is added. We need to stipulate this since the edge $e$ we select in line~\ref{line:growifclosedeV} is closed.

\section{Proof of Lemma~\ref{lemma:noerror}: \texorpdfstring{$m_k(H_1) > m_k(H_2)$}{strict inequality} case }\label{sec:grow}

We will prove Case~1 of Lemma~\ref{lemma:noerror} using two auxiliary algorithms \textsc{Special}$=$\textsc{Special($G', \eps,n$)} and \textsc{Grow}$=$\textsc{Grow($G', \eps,n$)}, which are both defined for any $\eps \geq 0$ (see Figures~\ref{specialfig} and \ref{growvfig})
and each take input $G'$ to be the $k$-graph which \textsc{Asym-Edge-Col-$\hat{\mathcal{B}}_{\eps}$} got stuck on.
Note that if \textsc{Asym-Edge-Col-$\hat{\mathcal{B}}_{\eps}$} gets stuck, then it did not ever use \textsc{B-Colour($G',\eps$)}, so the processes in these algorithms are still well-defined for all $\eps \geq 0$, even though \textsc{B-Colour($G',\eps$)} may not exist for the given value of $\eps$. 
\COMMENT{RH: We really want a comment here about how both algorithms want to be defined for all $\eps \geq 0$, somehow circumnavigating that this description is `subject' to the existence of \textsc{B-Colour}($G,\eps$), which we only know exists for $\eps^*$, yet this doesn't affect validity of descriptions....: RH: response to myself - if asym-edge-col gets stuck then it never reaches B-Colour, so we are okay! Let me know if this new paragraph makes sense. JH: The paragraph makes sense, but part of me is unsure if its the full point. Is the idea overall to only use $\eps* > 0$ when its really needed, to illustrate which parts are dependent on $\eps*>0$ and which parts can just use $\eps\geq 0$. I guess this is somewhat brought across in the paragraph.}

If \textsc{Asym-Edge-Col-$\hat{\mathcal{B}}_{\eps^*}$} has an error, then either \textsc{Special($G',\eps^*,n$)} outputs one of a constant number of subhypergraphs of $G'$ which are too dense to appear in $G^k_{n,p}$ a.a.s.~(with $p$ as in Lemma~\ref{lemma:noerror}), or, if \textsc{Special($G',\eps^*,n$)} does not output any subhypergraph, we use \textsc{Grow($G',\eps^*,n$)} to compute a subhypergraph $F \subseteq G'$ which is either too large in size or too dense to appear in $G^k_{n,p}$ a.a.s.
Indeed, letting $\mathcal{F}_0$ be the class of all $k$-graphs which can possibly be returned by algorithm \textsc{Special($G',\eps^*,n$)} and $\tilde{\mathcal{F}}$ be the class of all $k$-graphs that can possibly be returned by algorithm \textsc{Grow($G',\eps^*,n$)}, we will show that the expected number of copies of $k$-graphs from $\mathcal{F} = \mathcal{F}_0 \cup \tilde{\mathcal{F}}$ contained in $G^k_{n,p}$ is $o(1)$, which with Markov's inequality implies that $G^k_{n,p}$ a.a.s.~contains no $k$-graph from $\mathcal{F}$. 
This in turn implies Case 1 of Lemma~\ref{lemma:noerror} by contradiction.

\subsection{Definitions of \textsc{Grow} and \text{Special} procedures}

\COMMENT{RH: commented out the full text of repeated definitions for now. JH: Yeah, I think what's written here now is sufficient.}
To state -- and discuss --- \textsc{Special} and \textsc{Grow}, recall the definitions of 
$\lambda(G)$, 
$\mathcal{R}_G$, 
$\mathcal{L}_G$,
$\mathcal{L}^*_G$,
$\mathcal{C}$,
$\mathcal{C}^*$,
$\mathcal{T}^B_G$,
$\gamma(H_1, H_2, \eps)$ and an edge being \emph{open in $G$} from previous sections. 

Algorithm \textsc{Special} has as input the $k$-graph $G' \subseteq G$ that \textsc{Asym-Edge-Col-$\hat{\mathcal{B}}_{\eps}$} got stuck on, as well as the constant $\eps \geq 0$ that was an input into \textsc{Asym-Edge-Col-$\hat{\mathcal{B}}_{\eps}$}
and $n$ (the $n$ from $G^k_{n,p}$). 
\COMMENT{RH: commented out saying $\eps>0$ here; I think it is really only necessary that $\eps>0$ for Markov application (to get the necessary space to apply it) right at the end.}
Let us consider the properties of $G'$ when \textsc{Asym-Edge-Col-$\hat{\mathcal{B}}_{\eps}$} got stuck. Because the condition in line~\ref{line:edgeremoval} of \textsc{Asym-Edge-Col-$\hat{\mathcal{B}}_{\eps}$} fails, $G'$ is in the family $\mathcal{C}$. In particular, every edge of $G'$ is contained in a copy $L \in \mathcal{L}_{G'}$ of $H_2$, and, because the condition in line~\ref{line:L*check} fails, we can assume in addition that $L$ belongs to $\mathcal{L}^*_{G'}$.
Hence, $G'$ is actually in the family $\mathcal{C}^*$.
Lastly, $G'$ is not an $(H_1, H_2)$-sparse $\hat{\mathcal{B}}_{\eps}$-graph because \textsc{Asym-Edge-Col-$\hat{\mathcal{B}}_{\eps}$} ended with an error.

We now outline algorithm \textsc{Special} which checks whether either of two special cases occur. 
The first case corresponds to when $G'$ is a $\hat{\mathcal{B}}_{\eps}$-graph, which is not $(H_1, H_2)$-sparse as it is a $k$-graph that \textsc{Asym-Edge-Col-$\hat{\mathcal{B}}_{\eps}$} got stuck on.
The second case happens if there are two $k$-graphs in $\mathcal{S}^B_{G'}(e)$ that overlap in (at least) the edge $e$. 
The outputs of these two special cases are $k$-graphs which the while-loop of the succeeding algorithm \textsc{Grow} could get stuck on. 


\begin{figure}[!ht]
\begin{algorithmic}[1]
\Procedure{Special}{$G' = (V,E), \eps, n$}
    \If {$\forall e \in E: |\mathcal{S}^B_{G'}(e)| = 1$}\label{line:growspecial1}
        \State $T \gets$ any member of $\mathcal{T}^B_{G'}$\label{line:growt}
        \State \Return $\bigcup_{e \in E(T)}\mathcal{S}^B_{G'}(e)$
    \EndIf
    \If{$\exists e \in E : |\mathcal{S}^B_{G'}(e)| \geq 2$}\label{line:growspecial2}
        \State $S_1, S_2 \gets$ any two distinct members of $\mathcal{S}^B_{G'}(e)$\label{line:growanytwo}
        \State \Return $S_1 \cup S_2$\label{line:growspecial2end}
    \EndIf\label{line:endif}
\EndProcedure 
\end{algorithmic}
\caption{The implementation of algorithm \textsc{Special}.}\label{specialfig}
\end{figure}

If algorithm \textsc{Special} does not return a subhypergraph, we utilise algorithm \textsc{Grow}. \textsc{Grow} again takes as input the $k$-graph $G' \subseteq G$ that \textsc{Asym-Edge-Col-$\hat{\mathcal{B}}_{\eps}$} got stuck on, as well as the constant $\eps \geq 0$ that was an input into \textsc{Asym-Edge-Col-$\hat{\mathcal{B}}_{\eps}$}
and $n$ (the $n$ from $G^k_{n,p}$).  Let us outline algorithm \textsc{Grow}. Since algorithm \textsc{Special} did not return a subhypergraph, in line~\ref{line:growanye} we may choose an edge $e$ that does not belong to any $k$-graph in $\mathcal{S}^B_{G'}$; note that this further implies that $e$ does not belong to any copy of any $k$-graph $B\in \hat{\mathcal{B}}_{\eps}$ in $G'$.\COMMENT{JH: Do I need to elaborate on this more? RH: I don't think so...}
Crucially, algorithm \textsc{Grow} chooses a $k$-graph $R \in \mathcal{R}_{G'}$  which contains such an edge $e$ and makes it the seed $F_0$ for a growing procedure (line~\ref{line:growanyr}). 
This choice of $F_0$ will allow us to conclude later that there always exists an open edge (see the proof of Claim~\ref{claim:growv}),
that is, the while-loop of \textsc{Grow} operates as desired and does not get stuck. 

\begin{figure}[!ht]
\begin{algorithmic}[1]
\Procedure{Grow}{$G' = (V,E), \eps, n$}
    \State $e \gets$ any $e \in E : |\mathcal{S}^B_{G'}(e)| = 0$\label{line:growanye}
    \State $F_0 \gets$ any $R \in \mathcal{R}_{G'}:e \in E(R)$\label{line:growanyr}
    \State $i \gets 0$
    \While {$(i < \ln(n)) \land (\forall \tilde{F} \subseteq F_i : \lambda(\tilde{F}) > -\gamma)$}\label{line:growwhileconditions}
        \If {$\exists R \in \mathcal{R}_{G'}\setminus \mathcal{R}_{F_i} : |V(R) \cap V(F_i)| \geq k$}\label{line:growRline}
            \State $F_{i+1} \gets F_i \cup R$\label{line:growR-F_i+1}
        \Else   
            \State $e \gets \textsc{Eligible-Edge}(F_i)$\label{line:eligible-edge}
            \State $F_{i+1} \gets \textsc{Extend-L}(F_i, e, G')$\label{line:F_i+1getsextend}
        \EndIf
        \State $i \gets i + 1$
    \EndWhile
    \If {$i \geq \ln(n)$}
        \State \Return{$F_i$}\label{line:growreturnfi}
    \Else
        \State \Return{\textsc{Minimising-Subhypergraph}($F_i$)}\label{line:growreturnminsub}
    \EndIf
    
\EndProcedure
\end{algorithmic}\smallskip\smallskip

\begin{algorithmic}[1]
\Procedure{Extend-L}{$F,e,G'$}
    \State $L \gets$ any $L \in \mathcal{L}^*_{G'}$: $e \in E(L)$\label{line:linl*}
    \State $F' \gets F \cup L$\label{line:f'fl}
        \ForAll{$e' \in E(L)\setminus E(F)$}\label{line:e'inL}
                \State $R_{e'} \gets$ any $R \in \mathcal{R}_{G'} : E(L) \cap E(R) = \{e'\}$\label{line:re'}
                \State $F' \gets F' \cup R_{e'}$\label{line:f'f're'}
        \EndFor
        \State \Return {$F'$}
\EndProcedure 

\end{algorithmic}
\caption{The implementation of algorithm \textsc{Grow}.}\label{growvfig}
\end{figure}

In each iteration $i$ of the while-loop, the growing procedure extends $F_i$ to $F_{i+1}$ in one of two ways.
The first (lines~\ref{line:growRline}-\ref{line:growR-F_i+1}) is by attaching a copy of $H_1$ in $G'$ that intersects $F_i$ in at least $k$ vertices but is not contained in $F_i$.
The second is more involved and begins with calling the function \textsc{Eligible-Edge} that we introduced in procedure \textsc{Grow-$\hat{\mathcal{B}}_{\eps}$}. Recall that \textsc{Eligible-Edge} maps $F_i$ to an open edge $e \in E(F_i)$, prioritising open edges that do not belong to fully open flowers (we will show that such an edge exists in each iteration of the while-loop of algorithm \textsc{Grow}).
Further, recall that \textsc{Eligible-Edge} selects this edge $e$ to be unique up to isomorphism of $F_i$, that is, for any two isomorphic $k$-graphs $F$ and $F'$, there exists an isomorphism $\phi$ with $\phi(F) = F'$ such that \[\phi(\textsc{Eligible-Edge}(F)) = \textsc{Eligible-Edge}(F').\] 
In particular, our choice of $e$ depends only on $F_i$ and not on the surrounding $k$-graph $G'$ or any previous $k$-graph $F_j$ with $j < i$ (indeed, there may be many ways that \textsc{Grow} could construct a $k$-graph isomorphic to $F_i$).
One could implement \textsc{Eligible-Edge} by having an enormous table of representatives for all isomorphism classes of $k$-graphs with up to $n$ vertices.
Note that \textsc{Eligible-Edge} does not itself increase the number of $k$-graphs $F$ that \textsc{Grow} can output.

Once we have selected our open edge $e \in E(F_i)$ using \textsc{Eligible-Edge}, we apply a procedure called \textsc{Extend-L} which attaches a $k$-graph $L \in \mathcal{L}^*_{G'}$ that contains $e$ to $F_i$ (line~\ref{line:F_i+1getsextend}).
We then attach to each new edge $e' \in E(L)\setminus E(F_i)$ a $k$-graph $R_{e'} \in \mathcal{R}_{G'}$ such that $E(L) \cap E(R_{e'}) = \{e'\}$ (lines~\ref{line:e'inL}-\ref{line:f'f're'} of \textsc{Extend-L}).
(We will show later that such a $k$-graph $L$ and $k$-graphs $R_{e'}$ exist and that $E(L)\setminus E(F_i)$ is non-empty.) The algorithm comes to an end when either $i \geq \ln(n)$ or $\lambda(\tilde{F}) \leq -\gamma$ for some subhypergraph $\tilde{F} \subseteq F_i$.
In the former, the algorithm returns $F_i$ (line~\ref{line:growreturnfi}); in the latter, the algorithm returns a subhypergraph $\tilde{F} \subseteq F_i$ that minimises $\lambda(\tilde{F})$ (line~\ref{line:growreturnminsub}).
For each $k$-graph $F$, the function \textsc{Minimising-Subhypergraph($F$)} returns such a minimising subhypergraph that is \textit{unique up to isomorphism}.
Once again, this is to ensure that \textsc{Minimising-Subhypergraph($F$)} does not itself artificially increase the number of $k$-graphs that \textsc{Grow} can output.
As with function \textsc{Eligible-Edge}, one could implement \textsc{Minimising-Subhypergraph} using an enormous look-up table.

We will now argue that \textsc{Grow} terminates without error, that is, \textsc{Eligible-Edge} always finds an open edge in \textsc{Grow} and all `any'-assignments in \textsc{Grow} and \textsc{Extend-L} are always successful. This is the first place where the proof deviates slightly from the corresponding result in \cite{h}, notably when showing that the call to \textsc{Eligible-Edge} in line~\ref{line:eligible-edge} is always successful.

\begin{claim}\label{claim:growv}
For any $\eps \geq 0$, algorithm \textsc{Grow($G', \eps, n$)} terminates without error on any input $k$-graph $G' \in \mathcal{C}^*$ that is not a $(H_1, H_2)$-sparse $\hat{\mathcal{B}}_{\eps}$-graph,
for which \textsc{Special($G',\eps,n)$} did not output a subhypergraph.
\COMMENT{RH: Needed to add no output from special here! JH: Indeed, good spot!}
\footnote{See Definitions~\ref{def:avgraph} and \ref{def:h1h2sparse}.} Moreover, for every iteration $i$ of the while-loop, we have $e(F_{i+1}) > e(F_i)$.
\end{claim}

\begin{proof}
One can easily see that the assignments in lines~\ref{line:growanye} and \ref{line:growanyr} are successful. 
Indeed, algorithm \textsc{Special} did not output a subhypergraph, hence we must have an edge $e \in E$ that is not contained in any $S \in \mathcal{S}^B_{G'}$. 
Also, there must exist a member of $\mathcal{R}_{G'}$ that contains $e$ because $G'$ is a member of $\mathcal{C}^* \subseteq \mathcal{C}$.

Next, we show that the call to \textsc{Eligible-Edge} in line~\ref{line:eligible-edge} is always successful.
Indeed, suppose for a contradiction that no edge is open in $F_i$ for some $i \geq 0$. 
Then every edge $e \in E(F_i)$ is in some $L \in \mathcal{L}^*_{F_i}$, by definition.
Hence $F \in \mathcal{C}^*$. Furthermore, observe that $F$ could possibly be constructed by algorithm \textsc{Grow-$\hat{\mathcal{B}}_{\eps}$}. 
Indeed, each step of the while loop of \textsc{Grow} can be mirrored by a step in \textsc{Grow-$\hat{\mathcal{B}}_{\eps}$}, regardless of the input $k$-graph $G' \in \mathcal{C}^*$ to \textsc{Grow}; in particular, \textsc{Eligible-Edge} acts in the same way in both algorithms, as \textsc{Eligible-Edge} depends purely on the structure of its input and nothing else.
However, our choice of $F_0$ in line~\ref{line:growanyr} guarantees that $F_i$ is not in $\hat{\mathcal{B}}_{\eps}$.
Indeed, the edge $e$ selected in line~\ref{line:growanye} satisfying $|\mathcal{S}^B_{G'}(e)| = 0$ is an edge of $F_0$ and $F_0 \subseteq F_i \subseteq G'$. 
Thus, $F$ is not an output of \textsc{Grow-$\hat{\mathcal{B}}_{\eps}$}. 
Together with the fact that $F$ could possibly be constructed by algorithm \textsc{Grow-$\hat{\mathcal{B}}_{\eps}$}, it must be the case that \textsc{Grow-$\hat{\mathcal{B}}_{\eps}$} would terminate before it is able to construct $F$.
Grow-$\hat{\mathcal{B}}_{\eps}$ only terminates in some step $j$ when there exists some subhypergraph $\tilde{F} \subseteq F_j$ such that $\lambda(\tilde{F}) \leq -\gamma$. 
Hence there exists $\tilde{F} \subseteq F_i$ such that $\lambda(\tilde{F}) \leq -\gamma$ thus \textsc{Grow} terminates in line~\ref{line:growwhileconditions} without calling \textsc{Eligible-Edge}. Thus every call to \textsc{Eligible-Edge} is successful and returns an edge $e$.
Since $G' \in \mathcal{C^*}$, the call to \textsc{Extend-L}$(F_i, e, G')$ is also successful and thus there exist suitable $k$-graphs $L \in \mathcal{L}^*_{G'}$ with $e \in E(L)$ and $R_{e'}$ for each $e' \in E(L) \setminus E(F_i)$. 

It remains to show that for every iteration $i$ of the while-loop, we have $e(F_{i+1}) > e(F_i)$. 
Since a copy $R$ of $H_1$ found in line~\ref{line:growRline} is a copy of $H_1$ in $G'$ but not in $F_i$ (and $H_1$ is connected), we must have that $F_{i+1} = F_i \cup R$ contains at least one more edge than $F_i$.

So assume lines~\ref{line:eligible-edge} and \ref{line:F_i+1getsextend} are called in iteration $i$ and let $e$ be the edge chosen in line~\ref{line:eligible-edge} and $L$ the subhypergraph selected in line~\ref{line:linl*} of \textsc{Extend-L}$(F_i, e, G')$.
By the definition of $\mathcal{L}^*_{G'}$, for each $e' \in E(L)$ there exists $R_{e'} \in \mathcal{R}_{G'}$ such that $E(L) \cap E(R_{e'}) = \{e'\}$. If $|E(L) \setminus E(F_i)| > 0$, then $e(F_{i+1}) \geq e(F_i \cup L) > e(F_i)$.
Otherwise, $L \subseteq F_i$. But since $e$ is open in $F_i$, we must have $L \notin \mathcal{L}^*_{F_i}$.
Thus there exists $e' \in L$ such that $R_{e'} \in \mathcal{R}_{G'}\setminus \mathcal{R}_{F_i}$ and $|V(R_{e'}) \cap V(F_i)| \geq k$, contradicting that lines~\ref{line:eligible-edge} and \ref{line:F_i+1getsextend} are called in iteration $i$.\end{proof}

\subsection{Iterations of the \textsc{Grow} procedure}

In this section we describe crucial properties of \textsc{Grow} which are required for the proof of Lemma~\ref{lemma:noerror}. Later we will see that these properties transfer easily to \textsc{Grow-$\hat{\mathcal{B}}_{\eps}$} as well.

We consider the evolution of $F_i$ in more detail. 
We call iteration $i$ of the while-loop in algorithm \textsc{Grow} \emph{non-degenerate} if all of the following hold:

\begin{itemize}
    \item The condition in line~\ref{line:growRline} evaluates to false (and \textsc{Extend-L} is called); 
    \item In line~\ref{line:f'fl} of \textsc{Extend-L}, we have $V(F) \cap V(L) = e$;
    \item In every execution of line~\ref{line:f'f're'} of \textsc{Extend-L}, we have $V(F') \cap V(R_{e'}) = e'$.
\end{itemize}
Otherwise, we call iteration $i$ \emph{degenerate}.
Note that, in non-degenerate iterations, there are only a \textit{constant} number of $k$-graphs $F_{i+1}$ that can result from any given $F_i$ since \textsc{Eligible-Edge} determines the exact position where to attach the copy $L$ of $H_2$, $V(F_i) \cap V(L) = e$ and for every execution of line~\ref{line:f'f're'} of \textsc{Extend-L} we have $V(F') \cap V(R_{e'}) = e'$ (recall that the edge $e$ found by \textsc{Eligible-Edge}($F_i$) is \emph{unique up to isomorphism of $F_i$}). 

\begin{claim}\label{claim:non-degen}
For any $\eps \geq 0$, if iteration $i$ of the while-loop in procedure \textsc{Grow($G', \eps, n$)} is non-degenerate, we have \[\lambda(F_{i+1}) = \lambda(F_i).\]
\end{claim}

\begin{proof}
In a non-degenerate iteration we add $v_2 - k$ new vertices and $e_2 - 1$ new  edges for the copy of $H_2$ and then $(e_2 - 1)(v_1 - k)$ new vertices and $(e_2 - 1)(e_1 - 1)$ new edges to complete the copies of $H_1$. This gives \begin{align*}
  \lambda(F_{i+1}) - \lambda(F_i)   & = v_2 - k + (e_2  - 1)(v_1 - k) - \frac{(e_2 - 1)e_1}{m_k(H_1,H_2)} \\
                                    & = v_2 - k + (e_2  - 1)(v_1 - k) - (e_2 - 1)\left(v_1 - k + \frac{1}{m_k(H_2)}\right) \\
                                    & = 0,
\end{align*}
where we have used in the penultimate equality that $H_1$ is (strictly) balanced with respect to $d_k(\cdot, H_2)$ and in the final inequality that $H_2$ is (strictly) $k$-balanced.\end{proof}

When we have a degenerate iteration $i$, the structure of $F_{i+1}$ may vary considerably and also depend on the structure of $G'$.
Indeed, if $F_i$ is extended by a copy $R$ of $H_1$ in line~\ref{line:growR-F_i+1}, then $R$ could intersect $F_i$ in a multitude of ways.
Moreover, there may be several copies of $H_1$ that satisfy the condition in line~\ref{line:growRline}.
The same is true for $k$-graphs added in lines~\ref{line:f'fl} and \ref{line:f'f're'} of \textsc{Extend-$L$}. 
Thus, degenerate iterations cause us difficulties since they enlarge the family of $k$-graphs that algorithm \textsc{Grow} can return.
However, we will show that at most a constant number of degenerate iterations can happen before algorithm \textsc{Grow} terminates, allowing us to bound from above sufficiently well the number of non-isomorphic $k$-graphs \textsc{Grow} can return.
Pivotal in proving this is the following claim.  

\begin{claim}\label{claim:degenfull}
There exists a constant $\kappa = \kappa(H_1,H_2) > 0$ such that,
for any $\eps \geq 0$, 
if iteration $i$ of the while-loop in procedure \textsc{Grow($G', \eps, n$)} is degenerate then we have \[\lambda(F_{i+1}) \leq \lambda(F_i) - \kappa.\]
\end{claim}

The proof of Claim~\ref{claim:degenfull} can be found in \cite[Appendix~\ref{app:degenfull}]{bhh_kk_arxiv}. Its proof closely resembles that of \cite[Claim~6.3]{h}.


\subsection{Proof of Lemma~\ref{lemma:noerror}: Case 1}\label{subsec:lemnoerror}

\begin{define}\label{def:special}
    Let $\mathcal{F}_0 = \mathcal{F}_0(H_1,H_2, n, \eps)$ denote the class of $k$-graphs that can be outputted by \textsc{Special($G', \eps, n$)}.
\end{define}

Let us show that $G^k_{n,p}$ contains no graph from $\mathcal{F}_0(H_1,H_2,n,\eps^*)$ a.a.s.

\begin{claim}\label{claim:conclusion1}
There exists a constant $b = b(H_1,H_2) > 0$ such that for all $p \leq bn^{-1/m_k(H_1,H_2)}$, $G^k_{n,p}$ does not contain any $k$-graph from $\mathcal{F}_0(H_1,H_2,n,\eps^*)$ a.a.s.
\end{claim}

\begin{proof}
We can see that any $F \in \mathcal{F}_0$ is either of the form \[F = \bigcup_{\substack{e \in E(T)}}\mathcal{S}^B_{G'}(e)\] for some $k$-graph $T \in \mathcal{T}^B_{G'}$, or of the form \[F = S_1 \cup S_2\] for some edge-intersecting $S_1, S_2 \in \mathcal{S}^B_{G'}$.
Whichever of these forms $F$ has, since every element of $\mathcal{S}^B_{G'}$ is in $\hat{\mathcal{B}}_{\eps^*}$, and $T$ is such that $T \cong H_1$ or $T \cong H_2$, 
we have that $F \subseteq G'$ is not in $\mathcal{S}^B_{G'}$, and thus not isomorphic to a $k$-graph in $\hat{\mathcal{B}}_{\eps^*}$.
Indeed, otherwise the $k$-graphs $S$ forming $F$ would not be in $\mathcal{S}^B_{G'}$ due to the maximality condition in the definition of $\mathcal{S}^B_{G'}$.
Observe that, in either form, $F$ is a $k$-graph which could be constructed by \textsc{Grow-$\hat{\mathcal{B}}_{\eps^*}$}. But since $F$ is not isomorphic to a $k$-graph in $\hat{\mathcal{B}}_{\eps^*}$, we have that $F$ is not an output of \textsc{Grow-$\hat{\mathcal{B}}_{\eps^*}$}.
Similarly as in the proof of Claim ~\ref{claim:growv}, it follows that there exists a subgraph $\tilde{F}\subseteq F$ such that $\lambda(\tilde{F}) \leq -\gamma^*:=\gamma(H_1,H_2,\eps^*)$. 
Since we assumed the conditions in Theorem~\ref{thm:colourb} hold (and by the definition of $\eps^*$) we have that the family $\mathcal{F}_0$ is finite.
Using $\gamma^*>0$,\COMMENT{RH: Could add a footnote here saying `Indeed, this is why the definition of $\eps^*$ was necessary.' or something similar? Just to highlight it again! JH: I think enough has been said. In particular, $\eps*$ popping up instead of $\eps$ should immediately signal the reader or at least make them go back to read again what having the star is all about. Its hard to miss.} by Markov's inequality we have that $G^k_{n,p}$ contains no $k$-graph from $\mathcal{F}_0$ a.a.s.
\end{proof}

\begin{define}\label{def:growoutputs}
Let $\tilde{\mathcal{F}} = \tilde{\mathcal{F}}(H_1, H_2, n, \eps)$ denote a family of representatives for the isomorphism classes of all $k$-graphs that can be the output of \textsc{Grow($G', \eps, n$)} on any input instance $G'$ for which it enters the while-loop.
\end{define}

\begin{claim}\label{claim:conclusion2}
There exists a constant $b = b(H_1,H_2) > 0$ such that for all $p \leq bn^{-1/m_k(H_1,H_2)}$, $G^k_{n,p}$ does not contain any $k$-graph from $\tilde{\mathcal{F}}(H_1,H_2,n,\eps^*)$ a.a.s.
\end{claim}

The proof of this claim is given in \cite[Appendix~\ref{sec:omittedproofs}]{bhh_kk_arxiv}.
Using Claims~\ref{claim:conclusion1} and~\ref{claim:conclusion2}, we can now easily finish the proof.
Suppose that the call to \textsc{Asym-Edge-Col-$\hat{\mathcal{B}}_{\eps^*}$}$(G)$ gets stuck for some $k$-graph $G$, and consider the subhypergraph $G' \subseteq G$ that \textsc{Asym-Edge-Col-$\hat{\mathcal{B}}_{\eps^*}$} gets stuck on at this moment. 
Then \textsc{Special}$(G', \eps^*, n)$ or \textsc{Grow}$(G', \eps^*, n)$ returns a copy of a $k$-graph $F$ in $\mathcal{F}_0(H_1, H_2, n, \eps^*)$ or $\tilde{\mathcal{F}}(H_1, H_2, n, \eps^*)$, respectively, that is contained in $G' \subseteq G$. 
By Claims~\ref{claim:conclusion1} and~\ref{claim:conclusion2}, this event a.a.s.~does not occur in $G = G^k_{n,p}$ with $p$ as claimed.
Thus \textsc{Asym-Edge-Col-$\hat{\mathcal{B}}_{\eps^*}$} does not get stuck a.a.s.~and, by Lemma~\ref{lemma:errorvalid}, finds a valid colouring for $H_1$ and $H_2$ of $G^k_{n,p}$ with $p \leq bn^{-1/m_k(H_1,H_2)}$ a.a.s.\qed

\section{A connectivity result for hypergraphs}\label{sec:connectivity}

Nenadov and Steger~\cite{ns} proved that strictly $2$-balanced graphs $H$ with $m_2(H) > 1$ are always $2$-connected (see \cite[Lemma 3.3]{ns}). 
For $k$-graphs $H_1$ and $H_2$ with $m_k(H_1) \geq m_k(H_2)$, we now prove a generalisation of this result when $H_2$ is strictly $k$-balanced and $H_1$ is strictly balanced with respect to $d_k(\cdot, H_2)$ for the $m_k(H_1) > m_k(H_2)$ case and when $H_1$ and $H_2$ are both strictly $k$-balanced for the $m_k(H_1) = m_k(H_2)$ case. This will be crucial in proving Lemmas~\ref{lem:finitegreater} and \ref{lem:finiteequal},
in particular, Lemmas~\ref{lem:newfinitenesshyper1} and \ref{lem:newfiniteness1a}. The proof is similar to that of \cite[Lemma 3.3]{ns}.

\begin{lem}\label{lem:newconnectivity}
    Let $H_1$ and $H_2$ be $k$-graphs such that $m_k(H_1) \geq m_k(H_2) \geq 1/(k-1)$, $H_2$ is strictly $k$-balanced and $H_1$ is strictly balanced with respect to $d_k(\cdot, H_2)$ if $m_k(H_1) > m_k(H_2)$ and $H_1$ and $H_2$ are both strictly $k$-balanced if $m_k(H_1) = m_k(H_2)$.
    Then $H_1$ and $H_2$ are connected. Moreover, for $i \in \{1,2\}$, if $m_k(H_2) > 1$,\footnote{Note that we could have $m_k(H_2)\geq 1$ here instead, but with the strictly $k$-balanced condition this forces $H_2$ to be two edges intersecting in $k-1$ vertices. Thus, the first property for $J_1$ and $J_2$ vacuously does not hold. Moreover, $H_2$ clearly contains vertices $x$ with $\deg(x) = 1$ and we wish to avoid $H_2$ having such a property in this paper (see Corollary~\ref{col:newconcor} and its applications). Hence, since we assume $m_k(H_2) > 1$ throughout this paper, we make this assumption here also.} then the following holds: 
    $V(H_i)$ does not contain a cut set $C$ of size at most $k-1$ such that there exist two $k$-graphs $J_1, J_2 \subsetneq H_i$ with all the following properties.
    \begin{itemize}
        \item $e_{J_1} \geq 1$, $e_{J_2} \geq 2$,
        \item $E(J_1) \cap E(J_2) = \emptyset$,
        \item $E(J_1) \cup E(J_2) = E(H_1)$, and
        \item $V(J_1) \cap V(J_2) \subseteq C$.
    \end{itemize}
\end{lem}

\begin{proof}
    We will first show that $H_1$ is connected in the case when $m_k(H_1) > m_k(H_2)$. 
    Assume for a contradiction that $H_1$ has $\ell_1 \geq 2$ components and denote the number of vertices and edges in each component by $u_1, \ldots, u_{\ell_1}$ and $c_1, \ldots, c_{\ell_1}$, respectively. 
    Then, since $H_1$ is strictly balanced with respect to $d_k(\cdot, H_2)$, we must have that \[\frac{\sum_{i=1}^{\ell_1} c_i}{\sum_{i=1}^{\ell_1} u_i - k + \frac{1}{m_k(H_2)}} > \frac{c_1}{u_1 - k + \frac{1}{m_k(H_2)}}\] and \[\frac{\sum_{i=1}^{\ell_1} c_i}{\sum_{i=1}^{\ell_1} u_i - k + \frac{1}{m_k(H_2)}} > \frac{\sum_{i=2}^{\ell_1} c_i}{\sum_{i=2}^{\ell_1} u_i - k + \frac{1}{m_k(H_2)}}.\] Since $m_k(H_2) \geq 1/(k-1) > 1/k$, 
    by Fact~\ref{fact:ineq}(i) we get that\footnote{Note that if $a/b < C$ or $c/d < C$ in Fact~\ref{fact:ineq}(i), as is the case here, then $\frac{a+c}{b+d} < C$.}
    \[\frac{\sum_{i=1}^{\ell_1} c_i}{\sum_{i=1}^{\ell_1} u_i - k + \frac{1}{m_k(H_2)}} > \frac{\sum_{i=1}^{\ell_1} c_i}{\sum_{i=1}^{\ell_1} u_i - 2k + \frac{2}{m_k(H_2)}} > \frac{\sum_{i=1}^{\ell_1} c_i}{\sum_{i=1}^{\ell_1} u_i - k + \frac{1}{m_k(H_2)}},\] a contradiction.

    Now let us show similarly that $H_2$ is connected when $m_k(H_1) > m_k(H_2)$. Note that this proof also covers the $m_k(H_1) = m_k(H_2)$ case for both $H_1$ and $H_2$. 
    Assume for a contradiction that $H_2$ has $\ell_2 \geq 2$ components and denote the number of vertices and edges in each component by $v_1, \ldots, v_{\ell_2}$ and $d_1, \ldots, d_{\ell_2}$, respectively.
    Then, since $H_2$ is strictly $k$-balanced we get that
    \[\frac{\sum_{i=1}^{\ell_2}d_i - 1}{\sum_{i=1}^{\ell_2}v_i - k} > \frac{d_1 - 1}{v_1 - k}\] and 
    \[\frac{\sum_{i=1}^{\ell_2}d_i - 1}{\sum_{i=1}^{\ell_2}v_i - k} > \frac{\sum_{i=2}^{\ell_2}d_i - 1}{\sum_{i=2}^{\ell_2}v_i - k}.\]  By Fact~\ref{fact:ineq}(i) we get that \[\frac{\sum_{i=1}^{\ell_2}d_i - 1}{\sum_{i=1}^{\ell_2}v_i - k} \geq \frac{\sum_{i=1}^{\ell_2}d_i - 2}{\sum_{i=1}^{\ell_2}v_i - 2k}.\] Together with \[\frac{\sum_{i=1}^{\ell_2}d_i - 1}{\sum_{i=1}^{\ell_2}v_i - k} = m_k(H_2) \geq \frac{1}{k-1},\] by Fact~\ref{fact:ineq}(i) we see that \[\frac{\sum_{i=1}^{\ell_2}d_i - 1}{\sum_{i=1}^{\ell_2}v_i - k} \geq \frac{\sum_{i=1}^{\ell_2}d_i - 1}{\sum_{i=1}^{\ell_2}v_i - k - 1} >  \frac{\sum_{i=1}^{\ell_2}d_i - 1}{\sum_{i=1}^{\ell_2}v_i - k},\] a contradiction.

    Let us now assume $m_k(H_1) > m_k(H_2) > 1$ and, for a contradiction, that there exists a cut set $C$ and non-empty $k$-graphs $J_1, J_2 \subsetneq H_1$ with the properties given in the statement of Lemma~\ref{lem:newconnectivity}. 
    Since $H_1$ is strictly balanced with respect to $d_k(\cdot, H_2)$, these properties together with Fact~\ref{fact:ineq}(i) imply that \begin{align} e_1 = e_{J_1} + e_{J_2} & < m_k(H_1, H_2)\left(v_{J_1} - k + \frac{1}{m_k(H_2)} + v_{J_2} - k + \frac{1}{m_k(H_2)}\right)\nonumber\\ & = m_k(H_1, H_2)\left(v_1 + |C| - 2k + \frac{2}{m_k(H_2)}\right).\nonumber\end{align}
    Since $H_1$ is (strictly) balanced with respect to $d_k(\cdot, H_2)$ and $H_2$ is strictly $k$-balanced, we may substitute in $e_1/(v_1 - k + 1/m_k(H_2))$ for $m_k(H_1, H_2)$ into the above and  rearrange to get \[0 < |C| - k + \frac{1}{m_k(H_2)},\] a contradiction since $m_k(H_2) \geq 1$ and $|C| \leq k-1$. 

    (The following covers $H_1$ and $H_2$ in the $m_k(H_1) = m_k(H_2)$ case as well.) Since $m_k(H_2) > 1$ and $H_2$ is (strictly) $k$-balanced, we have $e_2 \geq 3$. 
    Moreover, we have \begin{equation}\label{eq:nodegree1vertex}
    \frac{e_2 - 2}{v_2-k-1} > \frac{e_2 - 1}{v_2-k}
    \end{equation} and so $\delta(H_2) \geq 2$
    since $H_2$ is \emph{strictly} $k$-balanced. Let us assume for a contradiction that there exists a cut set $C$ and $J_1, J_2 \subsetneq H_2$ with the properties in Lemma~\ref{lem:newconnectivity}.
    Observe that since $\delta(H_2) \geq 2$, we must also have that $e_{J_1}, e_{J_2} \geq 2$, and thus $v_{J_1}, v_{J_2} \geq k+1$. 
    Thus, using that $H_2$ is strictly $k$-balanced together with Fact~\ref{fact:ineq}(i) we have that \[e_2 - 2 = e_{J_1} + e_{J_2} - 2 < m_k(H_2)(v_{J_1} - k + v_{J_2} - k) = m_k(H_2)(v_2 + |C| - 2k).\] Thus \[\frac{e_2 - 2}{v_2 + |C| - 2k} < \frac{e_2 - 1}{v_2 - k} \stackrel{\eqref{eq:nodegree1vertex}}{<} \frac{e_2 - 2}{v_2-k-1},\] a contradiction since $|C|\leq k-1$. \COMMENT{JH: All these proofs are slightly different, so I think we should keep them all in.}
\end{proof}

We will need the following corollary in the proof of Lemma~\ref{lem:newfinitenesshyper1}.

\begin{col}\label{col:newconcor}
    Let $H_1$ and $H_2$ be $k$-graphs such that $m_k(H_1) \geq m_k(H_2) > 1$, 
    $H_2$ is strictly $k$-balanced and $H_1$ is strictly balanced with respect to $d_k(\cdot, H_2)$ if $m_k(H_1) > m_k(H_2)$ and $H_1$ and $H_2$ are both strictly $k$-balanced if $m_k(H_1) = m_k(H_2)$. Then $\delta(H_1), \delta(H_2) \geq 2$.
\end{col}

\begin{proof}
    The result for $H_1$ and $H_2$ in the $m_k(H_1) = m_k(H_2)$ case and $H_2$ in the $m_k(H_1) > m_k(H_2)$ case was already shown in the proof of Lemma~\ref{lem:newconnectivity}. It remains to prove the result for $H_1$ in the case when $m_k(H_1) > m_k(H_2)$. For a contradiction, assume there exists a vertex $v$ of $1$-degree $1$ in $H_1$ and let $e$ be the edge containing $v$.
    By the first part of Lemma~\ref{lem:newconnectivity} we know that $H_1$ is connected.
    Further, since $m_k(H_1) > m_k(H_1, H_2) > m_k(H_2) > 1$, we have that $e_1\geq 3$. Indeed, if $e_1 = 1$ then $m_k(H_1) = 1/k < 1$, and if $e_1 = 2$ then, using that $H_1$ is (strictly) balanced with respect to $d_k(\cdot, H_2)$, we get that $k+2 > v_1 + 1/m_k(H_2)$ which with $m_k(H_2) > 1$ implies that $v_1 = k+1$. But then $m_k(H_1) = 1$.
    Hence we can choose $J_1 = e$, $J_2 = H_1[E(H_1) \setminus e]$ and $C = V(J_1) \cap V(J_2)$, observing that then $|C| \leq k-1$, $E(J_1) \cap E(J_2) = \emptyset$, $E(J_1) \cup E(J_2) = E(H_1)$, $e_{J_1} \geq 1$ and $e_{J_2} \geq 2$. But this contradicts the second part of Lemma~\ref{lem:newconnectivity}. 
\end{proof}

\section{Proof of finiteness of \texorpdfstring{$\hat{\mathcal{B}}_{\eps}$}{B family}}\label{sec:bfinite}

In this section we prove Theorem~\ref{thm:bfinite}. Naturally, we split the proof into two cases: (1) $m_k(H_1) > m_k(H_2)$ and (2) $m_k(H_1) = m_k(H_2)$.

\begin{lem}\label{lem:finitegreater}
    For all $\eps \geq 0$ and all pairs of $k$-graphs $H_1, H_2$ with $m_k(H_1) > m_k(H_2)$ that satisfy the conditions in Conjecture~\ref{conj:kk}, 
    we have that $\hat{\mathcal{B}}(H_1, H_2, \eps)$ is finite.
\end{lem}

\begin{lem}\label{lem:finiteequal}
    For all $\eps \geq 0$ and all pairs of $k$-graphs $H_1, H_2$ with $m_k(H_1) = m_k(H_2)$ that satisfy the conditions in Conjecture~\ref{conj:kk},
    we have that $\hat{\mathcal{B}}(H_1, H_2, \eps)$ is finite.
\end{lem}

Clearly, Lemmas~\ref{lem:finitegreater} and \ref{lem:finiteequal} together imply Theorem~\ref{thm:bfinite}. We prove Lemma~\ref{lem:finitegreater} in Section~\ref{sec:finitegreater}. The proof of Lemma~\ref{lem:finiteequal} is found in \cite[Appendix~\ref{sec:finiteequal}]{bhh_kk_arxiv} and is very similar to that of Lemma~\ref{lem:finitegreater}.

\subsection{Proof of Lemma~\ref{lem:finitegreater}}\label{sec:finitegreater}

We call iteration $i$ of the while loop of algorithm \textsc{Grow-$\hat{\mathcal{B}}_{\eps}$} \emph{non-degenerate} if all of the following hold:

\begin{itemize}
    \item \textsc{Close-$e$} is called (in line~\ref{line:F_i+1getsextendV} of \textsc{Grow-$\hat{\mathcal{B}}_{\eps}$} or line~\ref{line:closededgeclosee} of \textsc{Close-$e$-Closededge}); 
    \item In line~\ref{line:copyofh2V} of \textsc{Close-$e$}, we have $V(F) \cap V(L) = e$;
    \item In every execution of line~\ref{line:copyofh1V} of \textsc{Close-$e$}, we have $V(F') \cap V(R_{e'}) = e'$.
\end{itemize}

Otherwise, we call iteration $i$ \emph{degenerate}. 
One can observe that a non-degenerate iteration in algorithm \textsc{Grow-$\hat{\mathcal{B}}_{\eps}$} is essentially equivalent to a non-degenerate iteration in algorithm \textsc{Grow} in Section~\ref{sec:grow}, specifically in terms of the exact structure of the attached $k$-graph which, importantly,
means the same number of edges and vertices are added. A degenerate iteration in algorithm \textsc{Grow-$\hat{\mathcal{B}}_{\eps}$} is similarly equivalent to a degenerate iteration of \textsc{Grow}. 

In order to conclude that \textsc{Grow-$\hat{\mathcal{B}}_{\eps}$} outputs only a finite number of $k$-graphs, 
    it is sufficient to prove that there exists a constant $K$ such that there are no outputs of \textsc{Grow-$\hat{\mathcal{B}}_{\eps}$} after $K$ iterations of the while-loop.
In order to prove this, we employ Claims~\ref{claim:non-degen} and \ref{claim:degenfull} (see Section~\ref{subsec:lemnoerror}). Since degenerate and non-degenerate iterations operate in the same way (as described above) in algorithms \textsc{Grow} and \textsc{Grow-$\hat{\mathcal{B}}_{\eps}$}, we may write Claims~\ref{claim:non-degen} and \ref{claim:degenfull} with reference to \textsc{Grow-$\hat{\mathcal{B}}_{\eps}$}.

\begin{claim:non-degen}
For all $\eps \geq 0$, if iteration $i$ of the while-loop in procedure \textsc{Grow-$\hat{\mathcal{B}}_{\eps}$} is non-degenerate, we have \[\lambda(F_{i+1}) = \lambda(F_i).\]
\end{claim:non-degen}

\begin{claim:degenfull}
There exists a constant $\kappa_1 = \kappa_1(H_1,H_2) > 0$ such that, for all $\eps \geq 0$, if iteration $i$ of the while-loop in procedure \textsc{Grow-$\hat{\mathcal{B}}_{\eps}$} is degenerate then we have \[\lambda(F_{i+1}) \leq \lambda(F_i) - \kappa_1.\]
\end{claim:degenfull}

For any iteration $i$, if  \textsc{Close-$e(F,e)$} is used, then let $J_i:=L \cup_{e' \in E(L) \setminus E(F)} R_{e'}$, where $L$ is chosen in line~\ref{line:copyofh2V} and the $R_{e'}$ are chosen in line~\ref{line:copyofh1V}.
If \textsc{Attach-$H_1(F,v,w)$} is used, then let $J_i:=R$, where $R$ is chosen in  line~\ref{line:copyofh2V}.

If $i$ is non-degenerate, then we call $J_i$ a {\it flower}, and we call the edge $e$ for which \textsc{Close-$e(F,e)$} was run on the {\it attachment edge} of $J_i$.
For a flower $J_i$ with attachment edge $e=\{x_1, \ldots, x_k\}$, call $V(J_i) \setminus \{x_1, \ldots, x_k\}$ the {\it internal vertices}
and $E(J_i) \setminus E(L)$ the {\it petal edges}. 

Call a flower $J_{i}$ {\it fully open at time $\ell$}
if it was produced in a non-degenerate iteration at time $i$ with $i \leq \ell$, and no internal vertex of $J_i$ intersects with an edge added after time $i$.
\COMMENT{RH: Maybe need to make more clear: $F_i$'s are the full $k$-graphs at step $i$. JH: Its such a fundamental thing in the algorithm, it would be weird for a reader to miss that. Its the same in Grow and similar algorithms.}

\begin{lem}\label{lem:newfinitenesshyper1}
Let $G$ be any $k$-graph and let $e \in E(G)$ be any edge. 
Suppose $J_i$ is a flower attached at $e$.
Then every petal edge of $J_i$ is open at time $i$. 
\end{lem}

\begin{proof} 
Denote the flower by $$J_i := L \cup \bigcup_{j=1}^{e_2-1} R_i.$$
Suppose for a contradiction that there exists an edge $f \in E(R_i) \setminus E(L)$ for some $j \in [e_2-1]$ which is closed.
Without loss of generality suppose $j=1$.
Let $e' = L \cap R_1$.
Since we assumed $f$ is closed, there exist copies $L_1$ of $H_2$ and $P$ of $H_1$ such that $E(L_1) \cap E(P) = \{f\}$.

\begin{claim}\label{claim:PnotR1}
    $P \not = R_1$.
\end{claim}  

\begin{proof} 
Suppose not.
Then as $f \not= e'$, we may define $w$ to be a vertex of $f$ which is not in $e'$.
Since all other edges of $L_1$ aside from $f$ must be outside of $E(R_1)$, and no vertex outside of $V(R_1)$ is contained in an edge with $w$, 
we have $\deg_{L_1}(w)=1$, which is a contradiction as $\delta(H_2)\geq 2$ by Corollary~\ref{col:newconcor}.
\end{proof}

Let $R_{e_2}=G$, and let $P_j:=R_j \cap P$ for all $j \in [e_2]$.
We may now write $P=\cup_{j \in [e_2]} P_j$.
Let $I:=\{j: |E(R_j) \cap E(P)| \geq 1\}$, that is, the set of indices of $R_j$ which $P$ has non-empty edge-intersection with.
Let $M$ be the subhypergraph of $L$ with edge set $\{e \in E(L): {e}=E(L) \cap E(R_j), j \in I\}$. 
We have $P=\cup_{j \in I} P_i$ and $|I|=e_M$. 
We also have $e_1=e_P=\sum_{j \in I} e_{P_i}$. 

\begin{claim}\label{claim:allkvertices}
    For each $j \in I$, we have $V(P_j) \cap V(M) = V(R_j) \cap V(L)  $.
\end{claim} 
\begin{proof} For any $j\in I$, if $V(P_j) = e$ for some $e \in M$ then the claim holds for that $j$. So assume otherwise for some $j'\in I$. 
Then we must have that $|E(P_{j'})| \geq 2$ since $\delta(H_1)\geq 2$ by Corollary~\ref{col:newconcor}; then choosing $C := V(P_{j'}) \cap V(M)$, $J_1 := P[E(P)\setminus E(P_{j'})]$ and $J_2 := P_{j'}$, 
we have that $C$, $J_1$, $J_2$ have the properties in the second part of Lemma~\ref{lem:newconnectivity}, a contradiction to $H_1$ being strictly balanced with respect to $d_k(\cdot, H_2)$.\end{proof}
It follows from Claim~\ref{claim:allkvertices} that $v_1 = v_P = \sum_{j \in I} \left(v_{P_j}-k\right) +v_M$.

Further, each $P_j$ is a strict subhypergraph of $R_j$ since $P \not=R_1$ and $|E(P) \cap E(R_1)|\geq 1$.
Therefore, since $H_1$ is strictly balanced with respect to $d_k(\cdot, H_2)$, for each $j \in I$ we have
\begin{align}\label{eq:pstrict}
\frac{e_{P_j}}{v_{P_j}-k+\frac{1}{m_k(H_2)}} < \frac{e_1}{v_1-k+\frac{1}{m_k(H_2)}}=m_k(H_1,H_2).
\end{align}
Using Claim~\ref{claim:allkvertices} and \eqref{eq:pstrict}, we obtain
\begin{align}\label{eq:openedgecalc}
m_k(H_1,H_2) = m_k(P,H_2) 
& = \frac{e_P}{v_P-k+\frac{1}{m_k(H_2)}} \nonumber \\
& = \frac{\sum_{j \in I} e_{P_j}}{\sum_{j \in I} \left(v_{P_j}-k\right) +v_M -k + \frac{1}{m_k(H_2)}} \nonumber \\
& = \frac{\sum_{j \in I} e_{P_j}}{\sum_{j \in I} \left(v_{P_j}-k+\frac{1}{m_k(H_2)}\right) +v_M -k + \frac{1-e_M}{m_k(H_2)}} \nonumber \\
& \leq \frac{\sum_{j \in I} e_{P_j}}{\sum_{j \in I} \left(v_{P_j}-k+\frac{1}{m_k(H_2)}\right)} \nonumber \\
& < m_k(H_1,H_2),
\end{align}
a contradiction. 
The first inequality uses that $M$ is a subhypergraph of $H_2$ and $H_2$ is (strictly) $k$-balanced,
and the second inequality follows from Fact~\ref{fact:ineq}(i) and \eqref{eq:pstrict}.
\end{proof}

\begin{claim}\label{claim:newfinitenessc1} 
Suppose $J_i$ is a fully open flower at time $\ell-1$ and $J_{\ell}$ is such that $V(J_{\ell})$ has empty intersection with the internal vertices of $J_i$. 
Then $J_i$ is fully open at time $\ell$.
\end{claim}

\begin{proof}
Let $e$ be the attachment edge of $J_i$. 
Since none of the $k$-graphs $J_{i+1}, \dots, J_{\ell}$ touch any of the internal vertices of $J_i$, 
the $k$-graph $G$ with edge set $E(G):=E(F_{\ell}) \setminus (E(J_i) \setminus \{e\})$ fully contains each of $J_{i+1},\dots,J_{\ell}$. 
Therefore, we may apply Lemma~\ref{lem:newfinitenesshyper1} to $G$ and $J_i$, and deduce that the petal edges of $J_i$ are still open. 
It follows that $J_i$ is fully open at time $\ell$.
\end{proof}

\begin{claim}\label{claim:newfinitenessc2}
Suppose $J_{\ell}$ is a flower. 
Then at most one fully open flower is no longer fully open.
\end{claim}

\begin{proof}
Let $e$ be the attachment edge of $J_{\ell}$. 
Note that no two fully open flowers share any internal vertices, and therefore $e$ contains the internal vertices of at most one fully open flower. 
Let $J_i$ be any of the other fully open flowers (clearly $i<\ell$). 
Now we can apply Claim~\ref{claim:newfinitenessc1}, and deduce that $J_i$ is still fully open, as required.
\end{proof}

\begin{claim}\label{claim:newfinitenessc3}
Suppose $J_{\ell}$ and $J_{\ell+1}$ are each flowers, and further after $J_{\ell}$ is added, one fully open flower is no longer fully open.
Then after $J_{\ell+1}$ is added, the total number of fully open flowers does not decrease.
\end{claim}

\begin{proof}
Suppose $J_i$, $i<\ell$ is the flower which is no longer fully open after $J_{\ell}$ is added. 

We claim that $J_i$ still contains an open edge in $F_{\ell}$. 

Denote $J_i$ by $L \cup_{j=1}^{e_2-1} R_j$.
Let $e$ be the attachment edge of $J_{\ell}$. 
We note that $e \in R_j$ for some $j \in [e_2-1]$ -- suppose without loss of generality that $j=1$ --
and further, since $e$ was an open edge previously, it must be a petal edge contained in $R_1$ (i.e., it is not an edge of $L$). Indeed, by construction the attachment edge for $J_i$ is in some copy of $R$ in $F_{i-1}$. Hence $L \in \mathcal{L}^{*}_F$ and every edge of $L$ is closed.
Now let $f$ be a petal edge contained in $R_2$ and suppose for a contradiction that it is closed.
This implies that there exist copies $L_1$ of $H_2$ and $P$ of $H_1$ such that $E(L_1) \cap E(P) = \{f\}$.
Exactly as in Claim~\ref{claim:PnotR1}, we can deduce that $P \not = R_2$. 
Let $g$ be the attachment edge of $J_i$.
Let $P_{e_2} = P \cap (F_{\ell} \setminus (E(J_i \cup J_{\ell}) \setminus \{g\}))$.     
Let $P_1 = P \cap (R_1 \cup F_{\ell})$.
Let $P_i = P \cap R_i$ for each $i \in \{2,\dots,e_2-1\}$.
Now defining $I$ and $M$ as in the proof of Lemma~\ref{lem:newfinitenesshyper1} and doing the same calculation (see~(\ref{eq:openedgecalc})) yields a contradiction.

Since \textsc{Eligible-Edge} prioritises open edges which do not belong to fully open flowers, the next iteration will select an edge, possibly of $J_i$, to close which does not destroy\footnote{That is, make \emph{not} fully open.} any fully open copies. 
That is, the attachment edge of $J_{\ell+1}$ must be either an open edge from $J_i$ or some other open edge not belonging to a fully open flower at time $\ell$.
Since $J_i$ has already been destroyed, we have that $J_{\ell+1}$ does not destroy any fully open flowers using the argument from Claim~\ref{claim:newfinitenessc2}.\COMMENT{JH: Possibly a nicer way of writing this. I don't think it is possible to stipulate that Eligible-Edge selects an open edge exactly in $J_i$. One could think up very symmetrical looking $k$-graphs and Eligible-Edge won't be able to distinguish which `$J_i$' we're talking about here.}
\end{proof}

We can now prove Lemma~\ref{lem:finitegreater}.

\begin{proof}
By Claims~\ref{claim:non-degen} and~\ref{claim:degenfull} 
it easily follows that the number of degenerate iterations is finite, specifically at most $X:=\lambda(H_1) + \gamma/ \kappa_1$ (see the proof of Claim~\ref{claim:q_1}) to maintain that $\lambda \geq -\gamma$.

Let $f(\ell)$ be the number of fully open flowers at time $\ell$. 
For the algorithm to output a $k$-graph $G$ at time $\ell$, we must have $f(\ell)=0$. 
By Claim~\ref{claim:newfinitenessc1}, if iteration $i$ is degenerate, then $f(i) \geq f(i-1)-Y$, where $Y:=v_2+(e_2 - 1)(v_1 - k)$ (note $v(J_i) \leq Y$ for all $i$). 
By Lemma~\ref{lem:newfinitenesshyper1} and Claims~\ref{claim:newfinitenessc2} and \ref{claim:newfinitenessc3},
if $J_i$ and $J_{i+1}$ are both
non-degenerate iterations,
then $f(i+1) \geq f(i-1)+1$.

Now assume for a contradiction that there are $2XY+2$ consecutive non-degenerate iterations.
It follows that after all $X$ degenerate iterations are performed, we will still have
$f(\ell) \geq (2XY+2)/2 -XY>0$. 
Since any further iteration cannot decrease $f(\ell)$ while still maintaining that $\lambda \geq -\gamma$, the algorithm will not output any $k$-graph $G$ at any time $\ell' \geq \ell$.
Thus \textsc{Grow-$\hat{\mathcal{B}}_{\eps}$} outputs only a finite number of $k$-graphs.
\COMMENT{RH: This looks quite a bit different to the Nenadov--Steger proof, but its doing exactly the same job and I'm not sure it is necessary to elaborate any further.}
\end{proof}

\section{Proof of the colouring results Lemmas~\ref{lem:maincolouring} and~\ref{lem:hypergraphcolouring}}\label{sec:colour}

In Section~\ref{sec:usefultoolsforcolouring} we gather together colouring results which will be used to prove Lemmas~\ref{lem:maincolouring} and \ref{lem:hypergraphcolouring} in Section~\ref{sec:proofofcolouring}.


\subsection{Useful tools and colouring results for graphs and hypergraphs}\label{sec:usefultoolsforcolouring}

We begin with a collection of colouring results which can be viewed as an asymmetric $k$-graph version of~\cite[Lemma A.1]{ns},
which was the key tool for proving the symmetric graph version (i.e. where $H_1=H_2$) of Lemma~\ref{lem:maincolouring} (for all cases).
A symmetric version of Lemma~\ref{lem:hypergraphcolouring} already appeared in~\cite{t}.
These results will be used to prove Lemma~\ref{lem:maincolouring}(i)--(v), (viii) and (ix),
and Lemma~\ref{lem:hypergraphcolouring}.
The proof of these statements follow by reading through the proofs of the symmetric graph versions and appropriately generalising them.

Recall that for $k$-graphs $H_1,\dots,H_r$, we define $\delta_i:=\delta(H_i)$.
For a $k$-graph $G$, define $$\delta_{\max}(G):= \max\{\delta(G'): G' \subseteq G\}.$$ 

We will use the following hypergraph generalisation of the Nash-Williams arboricity theorem~\cite{n} due to Frank, Kir\'{a}ly and Kriesell~\cite{fkk}. Here we define a $k$-forest with vertex set $V$ and edge set $E$ to be a $k$-graph $H=(V,E)$ such that 
for all $X \subseteq V$ we have $e(H[X]) \leq |X|-1$.
\COMMENT{RH: changed it as otherwise it is just a $k$-tree, not a $k$-forest? JH: Agreed, particularly for $k=2$, but the other $k$'s it didn't make sense for either.}

\begin{thm}[\cite{fkk}, Theorem~2.10]\label{thm:nashwilliamshyper}
    The edge set of a $k$-graph $H$ can be decomposed into $\ell$ many $k$-forests if and only if \[ar(H) \leq \ell.\]
\end{thm}

Note that we will not utilise (i) and (iii) below with $k \geq 3$ in the proofs of our colouring results, but we still state Lemma~\ref{lem:asymmcol} completely in terms of $k$-graphs since (given Theorem~\ref{thm:nashwilliamshyper} for (i)) the proofs of the graph cases extend easily to the $k$-graph case.
\COMMENT{RH: Need comment such as this to acknowledge we are not using them. Is this comment okay? JH: I think its fine}

\begin{lem}\label{lem:asymmcol}
For any $k$-graphs $G,H_1,\dots,H_r$, we have $G \not \to (H_1,\dots,H_r)$ if any of the following hold:
\begin{enumerate}
\item[(i)] There exists $\eps>0$ such that $ar(G) \leq \sum_{i=1}^r \lfloor ar(H_i)-\eps \rfloor$;
\item[(ii)] There exists $\eps>0$ such that $m(G) \leq \sum_{i=1}^r \lfloor m(H_i)-\eps \rfloor$;
\item[(iii)] We have $\delta_{\max}(G) \leq \sum_{i=1}^r (\delta_i-1)$;
\item[(iv)] We have $m(G) < \delta_{\max}(H_1)$ and $\chi(H_2) \geq k+1$.
\end{enumerate}
\end{lem}

\begin{proof}
\begin{enumerate}
\item[(i)] 
By Theorem~\ref{thm:nashwilliamshyper} every $k$-graph $G$ has a partition of its edge set into $\lceil ar(G) \rceil$ parts such that each part is a $k$-forest. 
Now colour $\lfloor ar(H_i)-\eps \rfloor$ parts with the colour $i$, for each $i \in [r]$. 
By Theorem~\ref{thm:nashwilliamshyper} applied to each $H_i$, we thus avoid a monochromatic copy of $H_i$ in colour $i$ for each $i \in [r]$.
\medskip
    
\item[(ii)]
Let $A$ be $\lceil m(G) \rceil$ copies $v_1, \ldots, v_{\lceil m(G) \rceil}$ of each vertex 
$v \in V(G)$, and let $B:=E(G)$.
Join $a \in A$ to $b \in B$ if $a \in b$.
The definition of $m$-density tells us that for any
$X \subseteq B$, we have $\lceil m(G) \rceil \geq |X|/v(H[X])$.
Therefore $|N_A(X)|= \lceil m(G) \rceil \cdot v(H[X]) \geq |X|$, so by Hall's marriage theorem we can find a matching $M$ which covers $B$. 
Split the matching into $\lceil m(G) \rceil$ parts corresponding to which index of copy of $v$ was taken.
It follows that for all $k$-graphs there exists a partition of its edges into $\lceil m(G) \rceil$ parts such that each component in each part contains no more edges than vertices, in other words each part $P$ has $m(P) \leq 1$.
Similarly to the last case, colour $\lfloor m(H_i)-\eps \rfloor$ parts with the colour $i$, for each $i \in [r]$. 
We thus avoid a monochromatic copy of $H_i$ in colour $i$ for each $i \in [r]$.
\medskip

\item[(iii)] 
Construct a sequence $v_1,v_2,\dots,v_n$ of the vertices of $G$ as follows: let $v_i$ be a vertex of minimum degree in $G-\{v_1,\dots,v_{i-1}\}$. 
Then every vertex $v_i$ has degree at most $\delta_{\max}(G)$ into $G[\{v_{i+1},\dots,v_n\}]$, the $k$-graph induced by vertices `to the right'. 
Colour the vertices of $G$ backwards, i.e. starting with $v_n$. 
As every vertex $v_i$ has degree at most $\delta_{\max}(G) \leq (\delta_1-1)+\dots+(\delta_r-1)$ into the part that is already coloured, 
we can colour $(\delta_j-1)$ of these edges with colour $j$ for each $j \in [r]$. 
Clearly, the coloured part cannot contain a monochromatic copy of $H_j$ in colour $j$ for each $j \in [r]$ that contains $v_i$. 
By repeating this procedure for every vertex $v_i$, we obtain the desired colouring.
\medskip

\item[(iv)] 
Suppose there exists a $k$-colouring of the \emph{vertices} of $G$ without a monochromatic copy of $H_1$ (in any colour), and suppose that $X_1,\dots,X_k$ are the vertex classes of this colouring. 
Then colour all edges contained in $G[X_i]$ for some $i \in [k]$ with colour $1$ and all edges containing vertices from more than one $X_i$ colour $2$. 
Clearly there is no red $H_1$. Further, since the blue $k$-graph has weak chromatic number at most $k$ and $\chi(H_2) \geq k+1$, there is no blue $H_2$. 
    
Hence it suffices to find the vertex-colouring without  a monochromatic $H_1$. 
Choose $H' \subseteq H_1$ such that $\delta(H')=\delta_{\max}(H_1)$. 
We need to show that for every $k$-graph $G$ with $m(G)<\delta(H')$ we can find a vertex-colouring of $G$ without a monochromatic $H'$. 
Assume this is not true, and so there exists a minimal counterexample $G_0$. 
As $G_0$ is minimal, we know for every $v \in V(G_0)$, the $k$-graph $G_0-v$ does have a vertex-colouring without a monochromatic $H'$. 
Clearly if $\deg(v)<k\delta(H')$ then such a colouring can be extended to $v$. 
So we know that in $G_0$ every vertex has degree at least $k\delta(H')$, that is 
\begin{align*}
m(G_0) \geq \sum_{v\in G_0} \frac{\deg(v)}{k|G_0|}\geq \delta(H'),
\end{align*}
a contradiction.\end{enumerate}\end{proof}

It will be useful to note the following, which can be easily shown from the definitions: for any graph $G$ we have
\begin{align}\label{eq:ineqs}
m_2(G) \leq ar(G) + \frac{1}{2} \leq m(G)+1;
\end{align}
\begin{align}\label{eq:ineqs2}
m_2(G) > 1 \implies ar(G)>1.
\end{align}

For any $k$-graph $H$ we have
\begin{align}\label{eq:mkHbound}
m_k(H) & \leq m(H) + \binom{v_H}{k-2}.
\end{align}
Indeed, for $k=2$, 
this is the same as the bound in~(\ref{eq:ineqs}). 
For $k \geq 3$, let $J$ be a $k$-graph which maximises the expression in $m_k(H)$, 
and then choose $t>0$ so that $m_k(H) = \frac{e_J-1}{v_J-k} = \frac{e_J}{v_J}+t$ and note that $m_k(H) \leq m(H)+t$.
We have $t=\frac{e_J-1}{v_J-k} - \frac{e_J}{v_J}$, which is maximised by taking $e_J=\binom{v_J}{k}$.
A simple calculation shows $t \leq \binom{v_J}{k-2}$ 
\COMMENT{RH: I think we can probably omit this calculation, but here it is for future reference (e.g. in case referee says we should include it):
Using $v=v_J$ and $e=e_J$ to save space. 
First note $\frac{(v-1)(v-k+1)}{(k-1)(v-k)} \leq v$, since this rearranges to say $v \geq k+\frac{k-2}{v(k-2)}$ which is true since $v \geq k+1$.
We have $t \leq \frac{\binom{v}{k}-1}{v-k} - \frac{\binom{v}{k}}{v} = \frac{(v-1)(v-2)\cdots(v-k+1)}{(k-1)! (v-k)}-\frac{v}{v(v-k)}$ which if $k\geq 4$ is less than $\frac{v(v-2)\dots(v-k+2)}{(k-2)!} \leq \binom{v}{k-2}$. If $k=3$ then it is less than $v=\binom{v}{k-2}$ as well.
} 
and then we have $t \leq \binom{v_H}{k-2}$, as required.

(The rest of this subsection refers only to graphs.) 

For the proof of Lemma~\ref{lem:maincolouring}(vi), we will use an easy modification of a result on the minimum degree of Ramsey-minimal graphs
by Bishnoi et al.
~\cite[Proposition 3.4]{bbcgll}.
Their result states that any graph $G$ which satisfies $G \to (K_a, C_b)$ for $a \geq 3$ and $b \geq 4$ such that $G$ is minimal with this property (i.e. Ramsey-minimal), satisfies $\delta(G) \geq 2(a-1)$. 
The same result holds when replacing $C_b$ by any graph with minimum degree $2$.
This follows easily since having minimum degree $2$ is the only property of cycles used within their proof.
We include a proof here for completeness.

\begin{lem}\label{lem:cliqueminimal}
Suppose $G$ is Ramsey-minimal for $(K_a, H_2)$, where $a \geq 3$ and $H_2$ has minimum degree at least $2$,
i.e. $G \to (K_a,H_2)$ but $G' \not \to (K_a,H_2)$ for any proper subgraph $G' \subsetneq G$.
Then $\delta(G) \geq 2(a-1)$.
\end{lem}

\begin{proof}
Suppose for a contradiction that $v$ is a vertex of degree at most $2(a-1)-1$ in $G$. 
By the minimality of $G$, there exists a red/blue-colouring $c$ of $G-v$ with no red copy of $K_a$ and no blue copy of $H_2$.
If $G[N(v)]$ contains no red copy of $K_{a-1}$, then we can extend the colouring $c$ to $G$ by colouring all edges incident to $v$ red,
to obtain a colouring with no red $K_{a}$ or blue $H_2$, a contradiction.

Therefore, assume that there is at least one red copy of $K_{a-1}$ in $G[N(v)]$.
By~\cite[Lemma 3.3]{bbcgll}, any graph on $n<2(a-1)$ vertices containing $K_{a-1}$, such that $$\bigcap_{\stackrel{H \cong K_{a-1}}{H \subseteq G}} V(H) = \emptyset$$ contains $K_a$.
Since $G[N(v)]$ does not contain a red copy of $K_a$, there must be at least one vertex $u$ which lies in the intersection of all red copies of $K_{a-1}$ in $G[N(v)]$.
Extend the colouring $c$ by colouring $uv$ blue and all other edges from $v$ to $N(v) \setminus \{u\}$ red.
The blue edge ensures that there are no red copies of $K_a$, and further $v$ only lies in one blue edge, and so cannot lie in a blue copy of $H_2$ since $H_2$ has minimum degree $2$.
The obtained colouring again yields a contradiction. 
\end{proof}
\COMMENT{RH: feels slightly strange to quote Lemma 3.3 here in a ``complete proof'' but it also doesn't make sense to copy out its proof...}

For the proof of Lemma~\ref{lem:maincolouring}(vii), we wish to exploit the fact $K_{a,b}$ does not have edges between vertices of low degree; we design a graph parameter to quantify this property.
Define $z(G):=0$ if $G$ contains two vertices of minimum degree which are adjacent; 
otherwise, define $z(G)$ to be the maximum $t$ such that $$G[\{ v \in V(G):  \deg(v) \leq \delta(G)-1+t \}]$$ forms an independent set. 
In particular, we have $z(K_{a,b})=b-a$ (assuming $b \geq a$). 
Recall that \[\mathcal{C} = \mathcal{C}(H_1, H_2) := \{G = (V,E) : \forall e \in E\ \exists(L,R) \in \mathcal{L}_G \times \mathcal{R}_G\ \mbox{s.t.}\ E(L) \cap E(R) = \{e\}\}.\]

\begin{lem}\label{lem:z}
Let $A$ be a graph such that $A \in \mathcal{C}(H_1, H_2)$.
Let $x=\delta_1+\delta_2-1$ and $z=\max\{z(H_1),z(H_2)\}$. Then $$d(A) \geq \frac{x(x+z)}{2x+z}.$$
\end{lem} 

\begin{proof} We begin with the following key claim.
\begin{claim}\label{claim:independentset} In $A$, the set $B$ of vertices with degrees in $\{x, \dots, x+z-1\}$ form an independent set. \end{claim} \begin{proof} Assume not. Then there are two vertices $u,v$ in $A$ of degree at most $x + z - 1$ which are adjacent. Let $L$ and $R$ be the copies of $H_2$ and $H_1$, respectively, such that $E(L) \cap E(R) = \{uv\}$. Without loss of generality, assume $z = z(H_1)$. Then since $uv\in E(R)$ we can also assume without loss of generality that $u \not\in \{w\in V(R): \deg_R(w)\leq \delta_1+z-1\}$. Hence $\deg_A(u) \geq \delta_1 + z + \delta_2 - 1$. But at the start we said $\deg_A(u) \leq x-1+z = \delta_1 + z + \delta_2 - 2$, so this is a contradiction.\end{proof}
Let $C := V(A)\setminus B$.
Suppose for each $i \in \{0, \dots, z-1\}$ there are $b_i$ vertices of degree $x+i$ in $B$, and $c$ vertices in $C$.

We have $e(A)=b_0 x + b_1 (x+1) + \dots + b_{z-1} (x+z-1) + r$, where $r=e(A[C])$.
We also have $\sum_{v \in C} \deg(v) = e(A)+r \geq (x+z)c$.

From this we have $r \geq \max \{ 0,\frac{1}{2} ((x+z)c - (b_0 x + \dots + b_{z-1} (x+z-1))) \}$.

{\bf Case 1: $b_0 x + \dots + b_{z-1}(x+z-1) \geq (x+z)c$.}

We have $r \geq 0$, hence
\begin{align*}
d(A) & = \frac{e(A)}{v(A)} \geq \frac{b_0 x + b_1 (x+1) + \dots + b_{z-1} (x+z-1)}{\sum b_i+c} \\
& \geq \frac{b_0 x + \dots + b_{z-1} (x+z-1)}{\sum b_i+ \frac{b_0 x + \dots + b_{z-1} (x+z-1)}{x+z}} \\
& \geq \frac{b_0 x + \dots + b_{z-1} (x+z-1)}{\frac{b_0 x + \dots + b_{z-1}(x+z-1)}{x} + \frac{b_0 x + \dots + b_{z-1} (x+z-1)}{x+z}} \\
& = \frac{1}{\frac{1}{x} + \frac{1}{x+z}} = \frac{x(x+z)}{2x+z}.
\end{align*}

{\bf Case 2: $(x+z)c > b_0 x + \dots + b_{z-1}(x+z-1)$.}

Clearly, we have $r \geq 1/2((x+z)c - (b_0 x + \dots + b_{z-1}(x+z-1)))$, hence
\begin{align*}
d(A) & = \frac{e(A)}{v(A)} = \frac{b_0 x + b_1 (x+1) + \dots + b_{z-1} (x+z-1) + r}{\sum b_i+c} \\
& \geq  \frac{b_0 x + \dots + b_{z-1} (x+z-1) + (x+z)c}{2(\sum b_i+c)} \\
& \geq \frac{ x \sum b_i + (x+z)c}{2(\sum b_i +c)} = \frac{x}{2} + \frac{zc}{2(\sum b_i+c)}\\
& \geq \frac{x}{2} + \frac{zc}{2 ( \frac{(x+z)c}{x} + c)} = \frac{x}{2} + \frac{z}{2 (\frac{2x+z}{x})} = \frac{x(x+z)}{2x+z}.
\end{align*}
\end{proof}

\subsection{Proof of colouring results}\label{sec:proofofcolouring}
\begin{proofofmaincolouring}
First, in order be clear that the following contains no circular arguments, note that case (i) relies on cases (ii), (iii) and (vii), and cases (v) and (vi) rely on case (i).

%

{\bf Case (i):} We have $m_2(H_1)=m_2(H_2)$.

First, using Lemma~\ref{lem:maincolouring}(ii) and (iii), we may assume that $H_2$ is bipartite and $m(H_2) \leq 2$.
Further, since $m_2(H_1)=m_2(H_2)$, we may swap the roles of $H_1$ and $H_2$, 
so we may also assume that $H_1$ is bipartite and has $m(H_1) \leq 2$ as well.
Each $H_i$ being bipartite means they satisfy
\begin{align}\label{eq:ineq3}
e_i \leq \frac{1}{4} v_i^2 \quad \text{ and } \quad m_2(H_i) \leq m(H_i) + \frac{1}{2},
\end{align}
and further we have equality in the first inequality above if and only if we have equality in the second.
Write $m_2(H_i)=t+x$ where $t \in \mathbb{N}$ and $x \in [0,1)$. 

{\bf Subcase 1}: 
We have $x \in [1/2,1)$, and also either $x>1/2$ or $e_i < \frac{1}{4} v_i^2$ for each $i \in [2]$.

Then by~(\ref{eq:ineq3}), there exists $\eps$ such that $\lfloor m(H_i) - \eps \rfloor \geq t$ for each $i \in [2]$.
Therefore we have 
\begin{align*}
m(G) \leq m_2(H_1, H_2) = m_2(H_i) \leq t+1 \leq 2t \leq \lfloor m(H_1)-\eps \rfloor + \lfloor m(H_2) - \eps \rfloor,
\end{align*}
and so we can apply Lemma~\ref{lem:asymmcol}(ii).

{\bf Subcase 2}: We have $x=\frac{1}{2}$ and $e_i=\frac{1}{4} v_i^2$ for some $i \in [2]$.

Without loss of generality assume it is $i=1$. 
Then $H_1$ is a complete balanced bipartite graph and we must have 
$v_1=2\ell$ for some $\ell \in \{2,4\}$ since $m(H_1) \leq 2$. 
(To be precise, $H_1$ is either $K_{2,2}$ or $K_{4,4}$. Indeed, $m_2(K_{3,3}) = 2$, but $m_2(K_{2,2}) = 1.5$ and $m_2(K_{4,4}) = 2.5$.)
Therefore $m_2(H_1)=\frac{1}{2}(\ell+1)$ and so $t=\frac{1}{2} \ell$, which implies $ar(H_1)\geq e_1/(v_1-1)>\frac{1}{4}v_1=t$. 
We therefore see that Lemma~\ref{lem:maincolouring}(vii) covers this case.

{\bf Subcase 3}: We have $x \in [0,1/2)$.

Let $G'$ be a subgraph of $G$ with minimum degree $\delta_{\max}(G)$. 
Then $e(G') \geq \frac{1}{2} v(G') \cdot \delta_{\max}(G)$ and so $2m(G) \geq \delta_{\max}(G)$.
Thus we have 
\begin{align*}
\delta_{\max}(G) \leq 2m(G) \leq 2m_2(H_1) < 2t+1.
\end{align*} 
Since $\delta_{\max}(G)$ is integral, it must be at most $2t$.  
For each $i \in [2]$, we have that $H_i$ is strictly $2$-balanced, and so we have $\delta_i > m_2(H_i)$.
Further, since $m_2(H_i) \geq t$ and $\delta_i$ is integral, we must have $\delta_i \geq t+1$.
Overall, we have 
\begin{align*}
\delta_{\max}(G) \leq 2t \leq \delta_1+ \delta_2-2,
\end{align*} 
and so we can apply Lemma~\ref{lem:asymmcol}(iii). 
\medskip

{\bf Case (ii):} We have $\chi(H_2) \geq 3$.

See the proof of case (i) of Lemma~\ref{lem:hypergraphcolouring}.
\medskip

{\bf Case (iii):} We have $m(H_2) >2$.

See the proof of case (ii) of Lemma~\ref{lem:hypergraphcolouring}.
\medskip



{\bf Case (iv):} We have $ar(H_2)>2$.

Since $ar(H_2)>2$, there exists $\eps>0$ such that $\lfloor ar(H_2)-\eps \rfloor \geq 2$.
Further we can choose $\eps$ so that $\lfloor ar(H_1)-\eps \rfloor \geq ar(H_1)-1$.  
Therefore we have 
\begin{align*}
ar(G) & \stackrel{(\ref{eq:ineqs})}{\leq} m(G)+\frac{1}{2}
\leq m_2(H_1,H_2) + \frac{1}{2}
\stackrel{(P. \ref{prop:m2h1h2})}{\leq} m_2(H_1) + \frac{1}{2} \\
& \stackrel{(\ref{eq:ineqs})}{\leq} ar(H_1) + 1
\leq \lfloor ar(H_1)-\eps \rfloor + \lfloor ar(H_2)-\eps \rfloor,
\end{align*}
and so we can apply Lemma~\ref{lem:asymmcol}(i).
\medskip

{\bf Case (v):} We have $H_1=K_a$ for some integer $a \geq 3$.

Suppose for a contradiction that $G \to (K_a,H_2)$, and let $G'$ be a Ramsey-minimal subgraph of $G$.
By Lemma~\ref{lem:cliqueminimal}, we have $\delta(G') \geq 2(a-1)$.
It follows that $m(G) \geq \frac{e(G')}{v(G')} \geq \frac{\delta(G')}{2} \geq (a-1)$.
Either, we have $m_2(H_1)=m_2(H_2)$ and so Lemma~\ref{lem:maincolouring}(i) covers this case,
or we have $m(G) \leq m_2(H_1,H_2) < m_2(K_a) = \frac{a+1}{2}$, which is a contradiction as $a \geq 3$.
\medskip

{\bf Case (vi):} We have $H_1=K_{a,b}$ for some integers $b \geq a \geq 2$.
If $a=b=2$ and $m_2(H_2)=3/2$, then we have $m_2(H_1)=m_2(H_2)$ and so this case is covered by Lemma~\ref{lem:maincolouring}(i).

So assume otherwise. 
Assume for a contradiction that $G \to (K_{a,b},H_2)$.
Let $G'$ be a Ramsey-minimal subgraph of $G$.
It is not hard to see that $G' \in \mathcal{C}$.
We will use Lemma~\ref{lem:z} to show that $d(G')>m_2(K_{a,b},H_2)$, which is a contradiction since then $m(G) \geq d(G') >m_2(K_{a,b},H_2)$.
Note that since all graphs in $\hat{\mathcal{B}}(K_{a,b},H_2)$ are in $\mathcal{C}$, this argument also shows that $\hat{\mathcal{B}}(K_{a,b},H_2)=\emptyset$.

By Lemma~\ref{lem:z}, it suffices to prove that $m_2(K_{a,b},H_2) < x(x+z)/(2x+z)$, where $x=a+\delta_2-1$ and $z=b-a$. 
Note that $(K_{a,b}, H_2)$ is a heart by assumption, hence it suffices to prove that
\begin{align*}
\frac{ab}{a+b-2+\frac{1}{m_2(H_2)}} < \frac{(a+\delta_2-1)(b+\delta_2-1)}{a+b+2\delta_2-2} = \frac{ab+\delta_2(a+b+\delta_2-2)-a-b+1}{a+b-2+\frac{1}{m_2(H_2)}+2\delta_2-\frac{1}{m_2(H_2)}}.
\end{align*}
Now using Fact~\ref{fact:ineq}(ii),
\COMMENT{RH: the $(a+b)/(c+d) \geq C$ one.}
it suffices to show that
\begin{align*}
\frac{ab}{a+b-2+\frac{1}{m_2(H_2)}} < \frac{\delta_2(a+b+\delta_2-2)-a-b+1}{2\delta_2-\frac{1}{m_2(H_2)}}.
\end{align*}

{\bf Subcase 1:} $a,b \geq 2$ and $\delta_2 \geq 3$.

Since the right hand side is an increasing function of $\delta_2$, we use $\delta_2 \geq 3$ and obtain that it is at least $(2a+2b+4)/(6-1/m_2(H_2))$.
So it suffices to show that 
\begin{align*}
& \left ( 6-\frac{1}{m_2(H_2)} \right ) ab  < 2(a+b)^2+4(a+b)-\left ( 2-\frac{1}{m_2(H_2)} \right ) (2a+2b+4) \\
\Longleftarrow \ & 0  < \left ( 2+\frac{1}{m_2(H_2)} \right ) ab + 4(a+b) - \left (2-\frac{1}{m_2(H_2)} \right ) (2a+2b+4) \\
\Longleftarrow \ & 0  < \left ( 2+\frac{1}{m_2(H_2)} \right ) ab -8,
\end{align*}
where we used $(a+b)^2 \geq 4ab$ going from the first to second line. Now since $a,b \geq 2$, the last line is true.

{\bf Subcase 2:} $b \geq 3, a \geq 2$ and $\delta_2=2$.

We still use that the right hand side is an increasing function of $\delta_2$, and so here is at least $(a+b+1)/(4-1/m_2(H_2))$.
So it suffices to show that
\begin{align*}
& \left ( 4-\frac{1}{m_2(H_2)} \right ) ab  < (a+b)^2+(a+b)-\left ( 2-\frac{1}{m_2(H_2)} \right ) (a+b+1) \\
\Longleftarrow \ & 0  \leq \frac{1}{2} ab + (a+b) - \left (2-\frac{1}{2} \right ) (a+b+1) + (b-a)^2\\
 &   = \frac{1}{2} (ab-a-b-3+2(b-a)^2),
\end{align*}
where we used $(a+b)^2=4ab+(b-a)^2$ and $m_2(H_2)<\delta_2=2$ (since $H_2$ is strictly $2$-balanced as $(H_1, H_2)$ is a heart) going from the first to second line. If $a,b \geq 3$, the last line is true.
Otherwise $a=2$, so we require $2b^2-7b+3 \geq 0$ which is true for $b \geq 3$.

{\bf Subcase 3:} $a=b=2$, $\delta_2=2$ and $m_2(H_2) < \frac{3}{2}$. 

Here note that $m_2(K_{a,b})=\frac{3}{2}$.
Then plugging values straight into both sides, we require
\begin{align*}
\frac{2 \times 2}{2+2-2+\frac{1}{m_2(H_2)}} \leq \frac{2(4)-2-2+1}{2(2)-\frac{1}{m_2(H_2)}},
\end{align*}
which rearranges to say $m_2(H_2) < \frac{3}{2}$, which we assumed. 
%
\medskip

{\bf Case (vii):} There exists $\ell \in \mathbb{N}$ such that $\ell<ar(H_1)$ and $m_2(H_1,H_2) \leq \ell+\frac{1}{2}$.

We have $m_2(H_2)>1$ , so by~(\ref{eq:ineqs2}), 
there exists $\eps>0$ such that $\lfloor ar(H_2) - \eps \rfloor \geq 1$.
Further we can choose $\eps$ so that also 
$\lfloor ar(H_1) - \eps \rfloor = \ell$.
Therefore we have
\begin{align*}
ar(G) \stackrel{(\ref{eq:ineqs})}{\leq} m(G) + \frac{1}{2} \leq m_2(H_1,H_2) + \frac{1}{2} \leq \ell+1 \leq \lfloor ar(H_1) - \eps \rfloor +  \lfloor ar(H_2) - \eps \rfloor,
\end{align*}
and so we can apply Lemma~\ref{lem:asymmcol}(i).
\medskip

{\bf Case (viii):} There exists $\ell \in \mathbb{N}$ such that $\ell<m(H_1)$ and $m_2(H_1,H_2) \leq \ell+1$, and $H_2$ is not a cycle.

Since $H_2$ is connected, contains a cycle, but is not itself a cycle, we have $m(H_2)>1$.
Thus there exists $\eps>0$ such that $ \lfloor m(H_2) - \eps \rfloor \geq 1$.
Further we can choose $\eps$ so that also $\lfloor m(H_1) - \eps \rfloor = \ell$.
Therefore we have
\begin{align*} 
m(G) \leq m_2(H_1,H_2) \leq \ell+1 \leq \lfloor m(H_1) - \eps \rfloor  + \lfloor m(H_2) - \eps \rfloor,
\end{align*}
and so we can apply Lemma~\ref{lem:asymmcol}(ii).
\endproof
\end{proofofmaincolouring}

\begin{proofofhypergraphcolouring}
{\bf Case (i)}: We have $\chi(H_2) \geq k+1$.

\COMMENT{RH: Could add the previously observed `since for any $k$-graph $H$, a colouring avoiding a monochromatic subhypergraph of $H$ also avoids a monochromatic copy of $H$ itself'. JH: This was actually somewhere where not being a heart was a problem, as $\chi(H_2) \geq k+1$ may not carry at all to the `heart part' of $H_2'\subseteq H_2$. So the `we can assume $(H_1, H_2)$ is a heart' that was here before I think wasn't correct.}
Since $(H_1, H_2)$ is a heart, we have that $H_1$ is strictly $k$-balanced with respect to $m_k(\cdot, H_2)$ if $m_k(H_1) > m_k(H_2)$ or $H_1$ is strictly $k$-balanced if $m_k(H_1) = m_k(H_2)$. Thus we have $\delta_1 > m_k(H_1, H_2)$.
Therefore we have
\begin{align*}
m(G) \leq m_k(H_1,H_2) < \delta_1 \leq \delta_{\max}(H_1),
\end{align*} 
and so we can apply Lemma~\ref{lem:asymmcol}(iv).
\medskip

{\bf Case (ii)}: We have $m(H_2) > \binom{v_1}{k-2}+1$.

Since $m(H_2)>\binom{v_1}{k-2}+1$, there exists $\varepsilon>0$ such that $\lfloor m(H_2)-\varepsilon \rfloor \geq \binom{v_1}{k-2}+1$. 
Further we can choose $\varepsilon$ so that also $\lfloor m(H_1)-\varepsilon \rfloor \geq m(H_1)-1$.
Therefore we have
\begin{align*}
m(G) \leq m_k(H_1,H_2) \leq m_k(H_1) \stackrel{(\text{\ref{eq:mkHbound}})}{\leq} m(H_1) + \binom{v_1}{k-2} \leq \lfloor m(H_1)-\varepsilon \rfloor + \lfloor m(H_2)-\varepsilon \rfloor,
\end{align*}
and so we can apply Lemma~\ref{lem:asymmcol}(ii).
\endproof
\end{proofofhypergraphcolouring}

\section{Concluding remarks}\label{sec:conclude}

Recall that we assumed $m_2(H_2)>1$ in Conjecture~\ref{conj:kk} in order to avoid possible complications arising from $H_1$ or $H_2$ being certain forests, in particular those failing to satisfy the statement `for all graphs $G$ we have $m(G) \leq m_2(H)$ implies $G \not\to (H,H)$'.
One may have noticed earlier that we did not explicitly state a version of the Kohayakawa--Kreuter conjecture for hypergraphs. This is because although the reduction of the natural $0$-statement to the statement `for all $k$-graphs $G$ we have $m(G) \leq m_k(H_1, H_2)$ implies $G \not\to (H_1, H_2)$' (see Theorem~\ref{thm:coloursuffices}) holds when $m_k(H_2) > 1$, Gugelmann et al.~\cite{gnpsst} showed that there are hypergraphs with $m_k(H) > 1$ which have different thresholds than $n^{-1/m_k(H)}$.
Indeed, for a graph $G$, denote by $G^{+k}$ the $k$-uniform hypergraph obtained from $G$ by adding the same set of $k-2$ additional vertices to every edge of $G$.
For $k \geq 4$, we have $K_6^{+k} \to (K_3^{+k},K_3^{+k})$ and $m(K_6^{+k})=15/(k+4)<2=m_k(K_3^{+k})$, so the statement fails when taking $H_1=H_2=K_3^{+k}$. 
On the other hand, Thomas~\cite{t} proved that if $H$ is a $k$-graph containing a strictly $k$-balanced subhypergraph $H'$ with $m_k(H')=m_k(H)$ such that $m(H') \geq 2+\frac{2k-1}{v_{H'}-2k}$, then the threshold is $n^{-1/m_k(H)}$.
It is therefore tempting to speculate that $m_k(H)>2$ might be sufficient to prove $n^{-1/m_k(H)}$ is a threshold for $G^k_{n,p} \to (H,H)$.
To suggest a sufficient lower bound on $m_k(H_1,H_2)$, or just on $m_k(H_2)$, would be even more ambitious. Indeed, the complexity of the hypergraph case is even more evident given the following observation. Gugelmann et al.~\cite{gnpsst} and Thomas~\cite[Section 6.5]{t} showed that there exist $k$-graphs $H$ for the symmetric problem -- that is, $H_1 = \cdots = H_r = H$ -- such that a threshold for $G^{k}_{n,p} \to (H, \ldots, H)$ is neither determined by $m_{k}(H)$ nor by a density $m(G) \leq m_k(H_1, H_2)$ of some $k$-graph $G$ (see \cite[Theorem~9]{gnpsst}). 

In \cite{npss}, Nenadov, Person, \v{S}kori\'{c} and Steger prove a `meta' version of Theorem~\ref{thm:coloursuffices}, 
up to a deterministic condition of a particular subgraph of $G^k_{n,p}$ being so-called `$H$-closed' (see Definition~1.10 in~\cite{npss}).
They prove this condition holds in the various Ramsey properties they consider. 
For $r \in \mathbb{N}$ and graphs $G, H_1, \ldots, H_r$, we say that \emph{$G$ has the anti-Ramsey property for $(H_1, \ldots, H_r)$} and write $G \overset{ar}{\to}$ ($H_1, \ldots, H_r$) if for every \emph{proper} colouring of the edges of $G$, there does not exist a rainbow copy of $H_i$ for any $i\in [r]$.
In a paper of the second and third authors together with Behague, Letzter and Morrison~\cite{bhhlm},
we prove the $H$-closed property holds for any strictly $2$-balanced $H$ with at least $5$ edges, thereby proving a version of Theorem~\ref{thm:coloursuffices} with the anti-Ramsey property replacing the usual Ramsey property and $H_1 = \cdots = H_r = H$.
It would also be very interesting to see if a similar approach to~\cite{npss} could be taken for general asymmetric Ramsey-like properties to obtain a version of Theorem~1.8 with the Ramsey property there replaced with some general Ramsey property $P$.

The reader will observe that we technically prove a slightly stronger result than Theorem~\ref{thm:coloursuffices}: we can replace the condition `we have that $m(G) \leq m_k(H_1,H_2)$ implies $G \not \to (H_1,H_2)$' with `for all $G \in \hat{\mathcal{B}}$, we have $G \not \to (H_1,H_2)$'. The `triforce' graph $G$ in Figure~\ref{fig:k3k12badgraph} satisfies $G \in \hat{\mathcal{B}}$ and $G \to (K_3, K_{1,2})$, hence any threshold for the $0$-statement to hold for $(H_1, H_2) = (K_3, K_{1,2})$ is indeed still asymptotically below $n^{-1/m_2(K_3, K_{1,2})}$.\COMMENT{JH: Not entirely sure where I'm going with this. I want to say something about $\hat{\mathcal{B}}$ not containing the `cycle' like graphs discussed in Appendix~\ref{app:hydeconj}, but I'm not sure this is really that interesting. I almost want to say we're characterising with $\hat{\mathcal{B}}$ those graphs $G$ with $m(G) \leq m_k(H_1, H_2)$ which are `interesting to colour' though I guess that might be blown out of the water by Gugelmann et also discovery of the disconnected $k$-graph with a different threshold not governed by the appearance of a particular $k$-graph with low density. CB: Not sure what is best here, but think if not expanded upon the sentence should be deleted, but then perhaps, too, the paragraph seems odd. So could be worth commenting whole paragraph out for sending draft to KSW, even if want to say more for the arxiv version and beyond. JH: This observation is actually relevant for Lemma~\ref{lem:maincolouring}(vi), so I think its good to put here. Added a touch to the last sentence to get across the meaning. Slightly future JH: I also think my thoughts on `characterising with $\hat{\mathcal{B}}$ those graphs $G$ with $m(G) \leq m_k(H_1, H_2)$ which are `interesting to colour' ' is likely wrong, given that the correct thing is probably the $(H_1, H_2)$-cores of Kuperwasser, Samotij and Wigdersons' paper.} 

Let us speak to the strangeness of some of the conditions in Lemma~\ref{lem:maincolouring},
many involving using floor functions and/or proof methods that only give the existence of a colouring and do not construct it. 
Indeed, this seems particularly curious when considering how Algorithms \textsc{Grow-$\hat{\mathcal{B}}_{\eps}$}$(H_1, H_2,\eps)$ and \textsc{Grow-$\hat{\mathcal{B}}_{\eps}$-Alt}$(H_1, H_2,\eps)$
are readily programmable and one can construct $\hat{\mathcal{B}}(H_1, H_2,\eps)$ using the appropriate algorithm in polynomial time with respect to $e_1$, $e_2$, $v_1$ and $v_2$.
Moreover, knowing that $\hat{\mathcal{B}}(H_1, H_2)$ is finite with sizes again in terms of $e_1$, $e_2$, $v_1$ and $v_2$ we could hope to check all red/blue colourings of graphs in $\hat{\mathcal{B}}(H_1, H_2)$ as well.
This is achieved in the paper of Marciniszyn, Skokan, Sp\"{o}hel and Steger~\cite{msss} which deals with the case when $H_1$ and $H_2$ are both cliques (and not both $K_3$), albeit with different methods. 
In \cite{kk} and \cite{lmms}, which deal with the cases when $H_1$ and $H_2$ are both cycles and $H_1$ and $H_2$ are a clique and a cycle,
respectively, the authors of both papers show implicitly that $\hat{\mathcal{B}}(H_1, H_2) = \emptyset$.
Let us also remark that in \cite{msss} 
it is implicitly shown that $\hat{\mathcal{B}}(H_1, H_2)$, with each $H_i$ a clique, is non-empty only when $H_2 = K_3$ or $H_1 = H_2 = K_4$, 
and moreover, $|\hat{\mathcal{B}}(H_1, H_2)| \leq 3$ in each case. A reader can also observe that in the proof of Lemma~\ref{lem:maincolouring}(vi) we show that $\hat{\mathcal{B}}(H_1,H_2)$ is empty. \COMMENT{JH: Add in anything relevant from our paper and Samotij and Kuperwasser's paper here.}
We raise studying whether $\hat{\mathcal{B}}(H_1,H_2)=\emptyset$ as an intriguing question in its own right. 
\begin{ques}
For which pairs of hypergraphs $(H_1,H_2)$ do we have $\hat{B}(H_1,H_2)=\emptyset$?
\end{ques}
In the graph case this would provide an alternate route to that taken by Christoph et al.~\cite{cmsw} in completing the proof of the Kohayakawa--Kreuter conjecture, whereas in the hypergraph case it may provide a possible pathway to obtaining new thresholds.


Let us remark that easy adaptions to our methods would allow us to prove Conjecture~\ref{conj:kk} in the case where $r \geq 3$ and $H_2=H_3$. In particular, 
an easy adaption of the proof of Lemma~\ref{lem:maincolouring}(iv) yields a $3$-colouring of $G$ which avoids monochromatic $H_i$ in colour $i$ for each $i \in [3]$.\footnote{Note that in~\cite{bhh} we claimed we could prove Conjecture~\ref{conj:kk} in the case where $r \geq 3$. However for this, we require $G \not \to (H_1,\dots,H_r)$ in place of $G \not \to (H_1,H_2)$ in Theorem~\ref{thm:coloursuffices}, which we cannot prove yet.}
\COMMENT{RH: new paragraph here!! JH: I'm not sure talking about the sparseness is necessary here. To accurately communicate to a reader why having $H_3$ sparseness is needed, and why $H_2 = H_3$ is (trivially) good for this but $H_2 \neq H_3$ possibly creates problems would take a good bit of time to explain (and a reviewer may well want to know), so I think its best that we just mention the colouring argument and leave it there. We can come back to this though if we want to add in more later.}

Finally, note that in the symmetric hypergraph case, where we have $H=H_1=\dots=H_r$ for $r\geq 3$, our methods show that to prove $n^{-1/m_k(H)}$ is the threshold, it suffices to prove a deterministic colouring result with the appropriate number of colours $r$, that is, for all $G$ with $m(G) \leq m_2(H)$, we have $G \not \to (\underbrace{H, \ldots, H}_{r \ times})$.
Since the validity of this statement could depend on the number of colours, this raises the following fascinating question.

\begin{ques}
Does there exist a hypergraph $H$ such that $n^{-1/m_k(H)}$ is not the threshold for the symmetric Ramsey property for $H$ with $2$ colours, but is the threshold when using $r$ colours for some $r \geq 3$?
\end{ques}
In the graph case, when $H$ is the path of three edges, while always having a threshold at $n^{-1/m_2(H)}$ for any $r\geq 2$, the type of threshold changes, going from a coarse threshold\footnote{A property $P$ has a coarse threshold at $f(n)$ if, for $p(n)$ such that $\lim_{n \to \infty} p(n)/f(n)=0$, $P$ does not occur w.h.p., while for $p(n)$ such that $\lim_{n \to \infty} p(n)/f(n)=\infty$, $P$ occurs w.h.p.}
 at $r=2$ to a threshold as defined in Section~\ref{sec:intro} for $r\geq 3$.

\section{Guide to appendices}\label{sec:guidetoapp}

Here we lay out a brief guide to what is contained in the various appendices. In Appendix~\ref{app:hydeconj}, we consider the family of graphs $\hat{\mathcal{A}}$ introduced in \cite{h} and discuss why in this paper we choose to work with $\hat{\mathcal{B}}$ instead of $\hat{\mathcal{A}}$ for the interested reader. Appendix~\ref{sec:omittedproofs} contains several proofs which were sufficiently similar to their counterparts in \cite{h} to warrant removing from the main text for brevity and so to not trouble the reader with calculations already performed in \cite{h}. Appendix~\ref{app:degenfull} contains almost verbatim the proof of \cite[Claim~6.3]{h}, except here it is generalised to hypergraphs. Since the hypergraph version differs very little from the graph version presented in \cite{h}, we have added the proof to the appendices for completeness. Finally, in Appendix~\ref{app:mh1mh2equal}, we prove the $m_k(H_1) = m_k(H_2)$ case of Lemma~\ref{lemma:noerror}, as well as proving Lemma~\ref{lem:finiteequal}. The proof of the $m_k(H_1) = m_k(H_2)$ case of Lemma~\ref{lemma:noerror} is both sufficiently similar in parts to the proof of the $m_k(H_1) > m_k(H_2)$ case and similar to work done in \cite{h} that we defer it to this appendix. The analysis of $\hat{\mathcal{B}}_{\eps}$ in this case, and proving it is finite, is very similar to the proof of Lemma~\ref{lem:finitegreater}. Moreover, Kuperwasser and Samotij~\cite{ks} already proved this case for graphs in a stronger form (see Section~\ref{sec:intro}), further justifying the move of this work to the appendices.

\section*{Acknowledgements}
We are very grateful to Eden Kuperwasser, Wojciech Samotij and Yuval Wigderson for sharing an early draft of~\cite{ksw}, and for many interesting discussions.

\bibliographystyle{plain}
\bibliography{bibliography}

\begin{thebibliography}{10}

\bibitem{bms}
J.~Balogh, R.~Morris, and W.~Samotij.
\newblock Independent sets in hypergraphs.
\newblock {\em J. Amer. Math. Soc.}, 28(3):669--709, 2015.

\bibitem{bhhlm}
N.~Behague, R.~Hancock, J.~Hyde, S.~Letzter, and N.~Morrison.
\newblock Thresholds for constrained {R}amsey and anti-{R}amsey problems.
\newblock {\em European J. Combin.}, 129:104159, 2025.

\bibitem{bbcgll}
A.~Bishnoi, S.~Boyadzhiyska, D.~Clemens, P.~Gupta, T.~Lesgourgues, and
  A.~Liebenau.
\newblock On the minimum degree of minimal {R}amsey graphs for cliques versus
  cycles.
\newblock {\em SIAM J. Discrete Math.}, 37(1):25--50, 2023.

\bibitem{bt}
B.~Bollob\'{a}s and A.~Thomason.
\newblock Threshold functions.
\newblock {\em Combinatorica}, 7(1):35--38, 1987.

\bibitem{bhh_kk_arxiv}
C.~Bowtell, R.~Hancock, and J.~Hyde.
\newblock Proof of the {K}ohayakawa--{K}reuter conjecture for the majority of
  cases, 2023.
\newblock arXiv:2307.16760.

\bibitem{bhh}
C.~Bowtell, R.~Hancock, and J.~Hyde.
\newblock A resolution of the {K}ohayakawa-{K}reuter conjecture for the
  majority of cases.
\newblock {\em European Conference on Combinatorics, Graph Theory and
  Applications}, 12, 2023.

\bibitem{cmsw}
M.~Christoph, A.~Martinsson, R.~Steiner, and Y.~Wigderson.
\newblock Resolution of the {K}ohayakawa--{K}reuter conjecture.
\newblock {\em Proc. Lond. Math. Soc.}, 130(1):e70013, 2025.

\bibitem{cg}
F.~Chung and R.~Graham.
\newblock {\em Erd\H{o}s on graphs}.
\newblock A K Peters, Ltd., Wellesley, MA, 1998.
\newblock His legacy of unsolved problems.

\bibitem{f}
J.~Folkman.
\newblock Graphs with monochromatic complete subgraphs in every edge coloring.
\newblock {\em SIAM J. Appl. Math.}, 18:19--24, 1970.

\bibitem{fkk}
A.~Frank, T.~Kir\'{a}ly, and M.~Kriesell.
\newblock On decomposing a hypergraph into {$k$} connected sub-hypergraphs.
\newblock {\em Discrete Appl. Math.}, 131(2):373--383, 2003.

\bibitem{fr}
P.~Frankl and V.~R\"{o}dl.
\newblock Large triangle-free subgraphs in graphs without {$K_4$}.
\newblock {\em Graphs Combin.}, 2(2):135--144, 1986.

\bibitem{fk}
E.~Friedgut and M.~Krivelevich.
\newblock Sharp thresholds for certain {R}amsey properties of random graphs.
\newblock {\em Random Structures Algorithms}, 17(1):1--19, 2000.

\bibitem{fkss}
E.~Friedgut, E.~Kuperwasser, W.~Samotij, and M.~Schacht.
\newblock Sharp thresholds for {R}amsey properties, 2022.
\newblock arXiv:2207.13982.

\bibitem{frrt}
E.~Friedgut, V.~R\"{o}dl, A.~Ruci\'{n}ski, and P.~Tetali.
\newblock A sharp threshold for random graphs with a monochromatic triangle in
  every edge coloring.
\newblock {\em Mem. Amer. Math. Soc.}, 179(845):vi+66, 2006.

\bibitem{frs}
E.~Friedgut, V.~R\"{o}dl, and M.~Schacht.
\newblock {R}amsey properties of random discrete structures.
\newblock {\em Random Structures Algorithms}, 37(4):407--436, 2010.

\bibitem{gnpsst}
L.~Gugelmann, R.~Nenadov, Y.~Person, N.~\v{S}kori\'{c}, A.~Steger, and
  H.~Thomas.
\newblock Symmetric and asymmetric {R}amsey properties in random hypergraphs.
\newblock {\em Forum Math. Sigma}, 5:Paper No. e28, 47, 2017.

\bibitem{hst}
R.~Hancock, K.~Staden, and A.~Treglown.
\newblock Independent sets in hypergraphs and {R}amsey properties of graphs and
  the integers.
\newblock {\em SIAM J. Discrete Math.}, 33(1):153--188, 2019.

\bibitem{h}
J.~Hyde.
\newblock Towards the 0-statement of the {K}ohayakawa-{K}reuter conjecture.
\newblock {\em Combin. Probab. Comput.}, 32(2):225--268, 2023.

\bibitem{jlr}
S.~Janson, T.~\L uczak, and A.~Ruci\'{n}ski.
\newblock {\em Random graphs}.
\newblock Wiley-Interscience Series in Discrete Mathematics and Optimization.
  Wiley-Interscience, New York, 2000.

\bibitem{kk}
Y.~Kohayakawa and B.~Kreuter.
\newblock Threshold functions for asymmetric {R}amsey properties involving
  cycles.
\newblock {\em Random Structures Algorithms}, 11(3):245--276, 1997.

\bibitem{kss}
Y.~Kohayakawa, M.~Schacht, and R.~Sp\"{o}hel.
\newblock Upper bounds on probability thresholds for asymmetric {R}amsey
  properties.
\newblock {\em Random Structures Algorithms}, 44(1):1--28, 2014.

\bibitem{ks}
E.~Kuperwasser and W.~Samotij.
\newblock The list-{R}amsey threshold for families of graphs.
\newblock {\em Combin. Probab. Comput.}, 33(6):829--851, 2024.

\bibitem{ksw}
E.~Kuperwasser, W.~Samotij, and Y.~Wigderson.
\newblock On the {K}ohayakawa--{K}reuter conjecture, 2023.
\newblock arXiv:2307.16611.

\bibitem{lmms}
A.~Liebenau, L.~Mattos, W.~Mendon\c{c}a, and J.~Skokan.
\newblock Asymmetric {R}amsey properties of random graphs involving cliques and
  cycles.
\newblock {\em Random Structures Algorithms}, 62(4):1035--1055, 2023.

\bibitem{msss}
M.~Marciniszyn, J.~Skokan, R.~Sp\"{o}hel, and A.~Steger.
\newblock Asymmetric {R}amsey properties of random graphs involving cliques.
\newblock {\em Random Structures Algorithms}, 34(4):419--453, 2009.

\bibitem{mns}
F.~Mousset, R.~Nenadov, and W.~Samotij.
\newblock Towards the {K}ohayakawa-{K}reuter conjecture on asymmetric {R}amsey
  properties.
\newblock {\em Combin. Probab. Comput.}, 29(6):943--955, 2020.

\bibitem{n}
C.~St. J.~A. Nash-Williams.
\newblock Decomposition of finite graphs into forests.
\newblock {\em J. London Math. Soc.}, 39:12, 1964.

\bibitem{npss}
R.~Nenadov, Y.~Person, N.~\v{S}kori\'{c}, and A.~Steger.
\newblock An algorithmic framework for obtaining lower bounds for random
  {R}amsey problems.
\newblock {\em J. Combin. Theory Ser. B}, 124:1--38, 2017.

\bibitem{ns}
R.~Nenadov and A.~Steger.
\newblock A short proof of the random {R}amsey theorem.
\newblock {\em Combin. Probab. Comput.}, 25(1):130--144, 2016.

\bibitem{nr}
J.~Ne\u{s}et\u{r}il and V.~R\"{o}dl.
\newblock The {R}amsey property for graphs with forbidden complete subgraphs.
\newblock {\em J. Combin. Theory Ser. B}, 20(3):243--249, 1976.

\bibitem{r}
F.~P. Ramsey.
\newblock On a {P}roblem of {F}ormal {L}ogic.
\newblock {\em Proc. London Math. Soc. (2)}, 30(4):264--286, 1929.

\bibitem{rr1}
V.~R\"{o}dl and A.~Ruci\'{n}ski.
\newblock Lower bounds on probability thresholds for {R}amsey properties.
\newblock In {\em Combinatorics, {P}aul {E}rd\H{o}s is eighty, {V}ol. 1},
  Bolyai Soc. Math. Stud., pages 317--346. J\'{a}nos Bolyai Math. Soc.,
  Budapest, 1993.

\bibitem{rrrandom}
V.~R\"{o}dl and A.~Ruci\'{n}ski.
\newblock Random graphs with monochromatic triangles in every edge coloring.
\newblock {\em Random Structures Algorithms}, 5(2):253--270, 1994.

\bibitem{rr2}
V.~R\"{o}dl and A.~Ruci\'{n}ski.
\newblock Threshold functions for {R}amsey properties.
\newblock {\em J. Amer. Math. Soc.}, 8(4):917--942, 1995.

\bibitem{rrhyper}
V.~R\"{o}dl and A.~Ruci\'{n}ski.
\newblock {R}amsey properties of random hypergraphs.
\newblock {\em J. Combin. Theory Ser. A}, 81(1):1--33, 1998.

\bibitem{rrs}
V.~R\"{o}dl, A.~Ruci\'{n}ski, and M.~Schacht.
\newblock {R}amsey properties of random {$k$}-partite, {$k$}-uniform
  hypergraphs.
\newblock {\em SIAM J. Discrete Math.}, 21(2):442--460, 2007.

\bibitem{st}
D.~Saxton and A.~Thomason.
\newblock Hypergraph containers.
\newblock {\em Invent. Math.}, 201(3):925--992, 2015.

\bibitem{ss}
M.~Schacht and F.~Schulenburg.
\newblock Sharp thresholds for {R}amsey properties of strictly balanced nearly
  bipartite graphs.
\newblock {\em Random Structures Algorithms}, 52(1):3--40, 2018.

\bibitem{t}
H.~Thomas.
\newblock {\em Aspects of Games on Random Graphs}.
\newblock PhD thesis, ETH Zurich, 2013.

\bibitem{lrv}
T.~\L uczak, A.~Ruci\'{n}ski, and B.~Voigt.
\newblock {R}amsey properties of random graphs.
\newblock {\em J. Combin. Theory Ser. B}, 56(1):55--68, 1992.

\end{thebibliography}

\appendix{}

\section{The conjecture of the third author}\label{app:hydeconj}

In order to state \cite[Conjecture~1.8]{h}, the conjecture of the third author~\cite{h} similar to Theorem~\ref{thm:colourb} in flavour, recall the definitions of 
$\mathcal{R}_G$, 
$\mathcal{L}_G$,
$\mathcal{L}^*_G$,
$\mathcal{C}$,
$\mathcal{C}^*$,
from previous sections.

The following sets out the central set of graphs that \cite[Conjecture~1.8]{h} is concerned with.

\begin{define}
\textnormal{Let $H_1$ and $H_2$ be non-empty graphs such that $m_2(H_1) \geq m_2(H_2) > 1$. Let $\eps := \eps(H_1, H_2) > 0$ be a constant. Define $\hat{\mathcal{A}} = \hat{\mathcal{A}}(H_1,H_2, \eps)$ to be \begin{small}\[\hat{\mathcal{A}} := \begin{cases}
    \{A \in \mathcal{C}^*(H_1, H_2): m(A) \leq m_2(H_1, H_2) + \eps \land A \ \mbox{is $2$-connected}\} \ \mbox{if} \ m_2(H_1) > m_2(H_2),\\
    \{A \in \mathcal{C}(H_1, H_2): m(A) \leq m_2(H_1, H_2) + \eps \land A \ \mbox{is $2$-connected}\} \ \mbox{if} \ m_2(H_1) = m_2(H_2).
\end{cases}\]\end{small}}
\end{define} 

\begin{conj}{\cite[Conjecture~1.8]{h}}\label{conj:aconj}
Let $H_1$ and $H_2$ be non-empty graphs such that $H_1 \neq H_2$ and $m_2(H_1) \geq m_2(H_2)>1$.\COMMENT{JH: I actually didn't write `$> 1$' in my old conjecture, but its implicit since we are trying to prove Conjecture~\ref{conj:kk} ultimately.} Assume $H_2$ is strictly $2$-balanced. Moreover, assume $H_1$ is strictly balanced with respect to $d_2(\cdot, H_2)$ if $m_2(H_1) > m_2(H_2)$ and strictly $2$-balanced if $m_2(H_1) = m_2(H_2)$.
Then there exists a constant $\eps := \eps(H_1, H_2) > 0$ such that the set $\hat{\mathcal{A}}$ is finite and every graph in $\hat{\mathcal{A}}$ has a valid edge-colouring for $H_1$ and $H_2$.
\end{conj}

The third author proved the following in \cite{h}.

\begin{thm}\label{thm:olda}
If Conjecture~\ref{conj:aconj} is true then the $0$-statement of Conjecture~\ref{conj:kk} is true.
\end{thm}

Notice that we can assume $H_1 \neq H_2$ as the $H_1 = H_2$ case of Conjecture~\ref{conj:kk} is handled by Theorem~\ref{thm:rr}. Note that the statement of Theorem~\ref{thm:colourb} is the same as the statement of Conjecture~\ref{conj:aconj} combined with Theorem~\ref{thm:olda}, except we replace $\hat{\mathcal{A}}$ with $\hat{\mathcal{B}}$. 

\subsection{Why use $\hat{\mathcal{B}}$ instead of $\hat{\mathcal{A}}$?}

The central concern with using $\hat{\mathcal{A}}$ instead of $\hat{\mathcal{B}}$ in order to ultimately prove the $0$-statement of Conjecture~\ref{conj:kk} is that $\hat{\mathcal{A}}$ could theoretically contain an infinite number of a certain type of structure which would render the finiteness part of Conjecture~\ref{conj:aconj} false. To illustrate this, let $H_1 = H_2 = K_3$. 
For $\ell \geq 4$, define $C^{K_4}_{\ell}$ to be the graph constructed by taking $\ell$ copies $A_1, \ldots, A_{\ell}$ of $K_4$ and concatenating them into a `cyclic' structure where $|V(A_i) \cap V(A_j)| = 1$ if and only if $j = i+1$ (mod $\ell$) and $|E(A_i) \cap E(A_j)| = 0$ for all $i,j \in [\ell]$ (see Figure~\ref{fig:badcyclegraph}).

\begin{figure}[!ht]
\begin{center}
\begin{tikzpicture}
\draw [line width=2pt] (7.75,3.5)-- (7.75,2.25);
\draw [line width=2pt] (7.75,3.5)-- (8.75,4.25);
\draw [line width=2pt] (8.75,4.25)-- (9.75,3.5);
\draw [line width=2pt] (7.75,2.25)-- (8.75,1.5);
\draw [line width=2pt] (8.75,1.5)-- (9.75,2.25);
\draw [line width=2pt] (9.75,2.25)-- (9.75,3.5);
\draw [line width=2pt] (7.75,3.5)-- (6.5,3.5);
\draw [line width=2pt] (6.5,3.5)-- (6.5,2.25);
\draw [line width=2pt] (6.5,2.25)-- (7.75,2.25);
\draw [line width=2pt] (7.75,2.25)-- (6.5,3.5);
\draw [line width=2pt] (6.5,2.25)-- (7.75,3.5);
\draw [line width=2pt] (9.75,3.5)-- (11,3.5);
\draw [line width=2pt] (11,3.5)-- (11,2.25);
\draw [line width=2pt] (11,2.25)-- (9.75,2.25);
\draw [line width=2pt] (9.75,2.25)-- (11,3.5);
\draw [line width=2pt] (9.75,3.5)-- (11,2.25);
\draw [line width=2pt] (7.75,3.5)-- (7,4.5);
\draw [line width=2pt] (8.75,4.25)-- (8,5.25);
\draw [line width=2pt] (8,5.25)-- (7,4.5);
\draw [line width=2pt] (7,4.5)-- (8.75,4.25);
\draw [line width=2pt] (8,5.25)-- (7.75,3.5);
\draw [line width=2pt] (8.75,4.25)-- (9.5,5.25);
\draw [line width=2pt] (9.75,3.5)-- (10.5,4.5);
\draw [line width=2pt] (10.5,4.5)-- (9.5,5.25);
\draw [line width=2pt] (9.5,5.25)-- (9.75,3.5);
\draw [line width=2pt] (8.75,4.25)-- (10.5,4.5);
\draw [line width=2pt] (7.75,2.25)-- (7,1.25);
\draw [line width=2pt] (8.75,1.5)-- (8,0.5);
\draw [line width=2pt] (8,0.5)-- (7,1.25);
\draw [line width=2pt] (7.75,2.25)-- (8,0.5);
\draw [line width=2pt] (7,1.25)-- (8.75,1.5);
\draw [line width=2pt] (9.75,2.25)-- (10.5,1.25);
\draw [line width=2pt] (8.75,1.5)-- (9.5,0.5);
\draw [line width=2pt] (9.5,0.5)-- (10.5,1.25);
\draw [line width=2pt] (9.75,2.25)-- (9.5,0.5);
\draw [line width=2pt] (8.75,1.5)-- (10.5,1.25);
\begin{scriptsize}
\draw [fill=black] (8,0.5) circle (2.5pt);
\draw [fill=black] (7.75,3.5) circle (2.5pt);
\draw [fill=black] (7.75,2.25) circle (2.5pt);
\draw [fill=black] (8.75,4.25) circle (2.5pt);
\draw [fill=black] (9.75,3.5) circle (2.5pt);
\draw [fill=black] (8.75,1.5) circle (2.5pt);
\draw [fill=black] (9.75,2.25) circle (2.5pt);
\draw [fill=black] (6.5,3.5) circle (2.5pt);
\draw [fill=black] (6.5,2.25) circle (2.5pt);
\draw [fill=black] (11,3.5) circle (2.5pt);
\draw [fill=black] (11,2.25) circle (2.5pt);
\draw [fill=black] (7,4.5) circle (2.5pt);
\draw [fill=black] (8,5.25) circle (2.5pt);
\draw [fill=black] (9.5,5.25) circle (2.5pt);
\draw [fill=black] (10.5,4.5) circle (2.5pt);
\draw [fill=black] (7,1.25) circle (2.5pt);
\draw [fill=black] (10.5,1.25) circle (2.5pt);
\draw [fill=black] (9.5,0.5) circle (2.5pt);
\end{scriptsize}
\end{tikzpicture}
\caption{The graph $C_{6}^{K_4}$.}\label{fig:badcyclegraph}
\end{center}
\end{figure}

One can observe that for all $\ell \geq 4$ we have that \[\lambda(C^{K_4}_{\ell}) = 3\ell - \frac{6\ell}{m_2(K_3, K_3)} = 0.\]
That is $e(C^{K_4}_{\ell})/v(C^{K_4}_{\ell}) = m_2(K_3, K_3)$, and so $d(C^{K_4}_{\ell}) = e(C^{K_4}_{\ell})/v(C^{K_4}_{\ell}) \leq m_2(K_3, K_3) + \eps$ for all $\eps \geq 0$.
Further, we can observe that $C^{K_4}_{\ell} \in \mathcal{C}^*(K_3, K_3)$ and $C^{K_4}_{\ell}$ is $2$-connected and balanced\COMMENT{JH: Haven't explicitly checked this. RH: Checked, its true.}, that is, $d(C^{K_4}_{\ell}) = m(C^{K_4}_{\ell})$.
Thus $C^{K_4}_{\ell} \in \hat{\mathcal{A}}(K_3, K_3, \eps)$ for all $\eps \geq 0$ and all $\ell \geq 4$, and so $\hat{\mathcal{A}}(K_3, K_3, \eps)$ contains infinitely many graphs. 

Now, in Conjecture~\ref{conj:aconj} we assume that $H_1 \neq H_2$, so the above is not a counterexample to Conjecture~\ref{conj:aconj}.
However, it may well be possible that a similar structure does provide a counterexample.
Indeed, say for a pair of graphs $H_1$ and $H_2$ with properties as in Conjecture~\ref{conj:aconj} that there exists a graph $A \in \hat{\mathcal{A}}(H_1, H_2, \eps)$ with $\lambda(A) \geq 1$.
Generalising our structure from earlier, for $\ell \geq 2$ define $C^{A}_{\ell}$ to be the graph constructed by taking $\ell$ copies $A_1, \ldots, A_{\ell}$ of $A$ and concatenating them into a `cyclic' structure where $|V(A_i) \cap V(A_j)| = 1$ if and only if $i = j+1$ (mod $\ell$) and $|E(A_i) \cap E(A_j)| = 0$ for all $i,j \in [\ell]$.
Then $$\lambda(C^{A}_{\ell}) = \sum_{i=1}^{\ell}(|V(A_i)| - 1) - \frac{\sum_{i=1}^{\ell} |E(A_i)|}{m_2(H_1, H_2)} = \ell\lambda(A) - \ell \geq 0,$$ which, as before, implies that $d(C^{A}_{\ell}) \leq m_2(H_1, H_2) + \eps$ for all $\eps \geq 0$.
If $C^{A}_{\ell}$ is also balanced, then we would now have a contradiction to Conjecture~\ref{conj:aconj}.

The authors do not know if such a structure exists for any particular pair of graphs $H_1$ and $H_2$ (with properties as in Conjecture~\ref{conj:aconj}).
Hence it may well be the case that the finiteness part of Conjecture~\ref{conj:aconj} holds for all suitable pairs of graphs $H_1$ and $H_2$. Note that the graph in Figure~\ref{fig:k3k12badgraph} is in fact $C^{K_3}_{3}$ and that $m(C^{K_3}_{\ell}) = 3/2 = m_2(K_3, K_{1,2})$ for all $\ell \geq 4$, so the condition $m_2(H_2) > 1$ in Conjecture~\ref{conj:aconj} must be included.

So why did we consider $\hat{\mathcal{B}}$ instead of $\hat{\mathcal{A}}$? Observe that for $\ell > \max\{E(H_1), E(H_2)\}$, there do not exist $i, j \in \{1, \ldots, \ell\}$ and a copy $H$ of $H_1$ or $H_2$ in $C^{A}_{\ell}$ such that $|E(H) \cap A_i|, |E(H) \cap A_j| \geq 1$. Indeed, $H_1$ and $H_2$ are both $2$-connected by Lemma~~\ref{lem:newconnectivity} applied with $k = 2$.\footnote{Note that one could also use \cite[Lemma~4.2]{h}, which is equivalent to Lemma~~\ref{lem:newconnectivity} when $k = 2$.}
That is, for every copy $H$ of $H_1$ and $H_2$ in $C^{A}_{\ell}$ we have $H \subseteq A_i$. Thus algorithm \textsc{Grow-$\hat{\mathcal{B}}_{\eps}$} -- or \textsc{Grow-$\hat{\mathcal{B}}_{\eps}$-Alt}, but for brevity we'll assume $m_2(H_1) > m_2(H_2)$, and so \textsc{Grow-$\hat{\mathcal{B}}_{\eps}$} is the relevant algorithm -- cannot construct $C^{A}_{\ell}$ for any $\ell \geq 2$ or $\eps \geq 0$.
Indeed, \textsc{Grow-$\hat{\mathcal{B}}_{\eps}$} can construct $A$, but there is no process in \textsc{Grow-$\hat{\mathcal{B}}_{\eps}$} that can attach a copy of $H_1$ or $H_2$ at a single vertex of $A$. 
Thus $C^{A}_{\ell} \notin \hat{\mathcal{B}}$ for all $\ell > \max\{E(H_1), E(H_2)\}$, and so there are not infinitely many of this type of structure in $\hat{\mathcal{B}}$.\COMMENT{JH: I do not think there is anything from the explanatory stuff from the old KK paper to add into this section. Section $1.4.3$ from that paper is maybe worth adding in, but I think it won't make sense here since we relegate the $m_2(H_1) = m_2(H_2)$ case to the appendices for the most part.}

\section{Proofs of results in Sections~\ref{sec:asymedgecolb} and \ref{subsec:lemnoerror}}\label{sec:omittedproofs}

In this appendix we provide the proofs omitted from Sections~\ref{sec:asymedgecolb} and~\ref{subsec:lemnoerror}. We repeat the results for context. Define $\gamma^*:=\gamma(H_1,H_2,\eps^*)$.

\begin{lemma:errorvalid}
Algorithm \textsc{Asym-Edge-Col-$\hat{\mathcal{B}}_{\eps^*}$} either terminates with an error in line~\ref{line:error} or finds a valid edge-colouring for $H_1$ and $H_2$ of $G$.
\end{lemma:errorvalid}

\begin{proof}
    Let $G^*$ denote the argument in the call to \textsc{B-Colour} in line~\ref{line:acolourcall}.
By Lemma~\ref{lemma:avcolour}, there is a valid edge-colouring for $H_1$ and $H_2$ of $G^*$. 
It remains to show that no forbidden monochromatic copies of $H_1$ or $H_2$ are created when this colouring is extended to a colouring of $G$ in lines~\ref{line:while2start}-\ref{line:while2end}. 

Firstly, we argue that the algorithm never creates a blue copy of $H_2$.
Observe that \emph{every} copy of $H_2$ that does not lie entirely in $G^*$ is pushed on the stack in the first while-loop (lines~\ref{line:while1start}-\ref{line:while1end}). 
Therefore, in the execution of the second loop, the algorithm checks the colouring of every such copy.
By the order of the elements on the stack, each such test is performed only after all edges of the corresponding copy of $H_2$ were inserted and coloured. 
For every blue copy of $H_2$, one particular edge $f$ (see line~\ref{line:getf}) is recoloured to red. Since red edges are never flipped back to blue, no blue copy of $H_2$ can occur.

We need to show that the edge $f$ in line~\ref{line:getf} always exists. 
Since the second loop inserts edges into $G'$ in the reverse order in which they were deleted during the first loop, when we select $f$ in line~\ref{line:getf}, $G'$ has the same structure as at the time when $L$ was pushed on the stack. 
This happened either in line~\ref{line:Lpush1} when there exists no copy of $H_1$ in $G'$ whose edge set intersects with $L$ on some particular edge $e \in E(L)$, or in line~\ref{line:Lpush2} when $L$ is not in $\mathcal{L}^*_{G'}$ due to the if-clause in line~\ref{line:L*check}. 
In both cases we have $L \notin \mathcal{L}^*_{G'}$, and hence there exists an edge $e \in E(L)$ such that the edge sets of all copies of $H_1$ in $G'$ do not intersect with $L$ exactly in $e$.

It remains to prove that changing the colour of some edges from blue to red by the algorithm never creates an entirely red copy of $H_1$.
By the condition on $f$ in line~\ref{line:getf} of the algorithm, at the moment $f$ is recoloured there exists no copy of $H_1$ in $G'$ whose edge set intersects $L$ exactly in $f$.
So there is either no copy of $H_1$ containing $f$ at all, or every such copy contains also another edge from $L$.
In the latter case, those copies cannot become entirely red since $L$ is entirely blue. \end{proof}

In order to prove Claim~\ref{claim:conclusion2}, we need the following two claims.

\begin{claim}\label{claim:q_1}
There exists a constant $q_1 = q_1(H_1,H_2)$ such that, for any input $G'$, algorithm \textsc{Grow($G', \eps^*, n$)} performs at most $q_1$ degenerate iterations before it terminates.
\end{claim}


\begin{proof}
By Claim~\ref{claim:non-degen}, the value of the function $\lambda$ remains the same in every non-degenerate iteration of the while-loop of algorithm \textsc{Grow}. 
However, Claim~\ref{claim:degenfull} yields a constant $\kappa$, which depends solely on $H_1$ and $H_2$, such that \[\lambda(F_{i+1}) \leq \lambda(F_i) - \kappa\] for every degenerate iteration $i$.

Hence, after at most \[q_1 := \frac{\lambda(F_0) + \gamma^*}{\kappa}\] degenerate iterations, we have $\lambda(F_i) \leq -\gamma^*$, and algorithm \textsc{Grow} terminates.\end{proof}

For $0 \leq d \leq t < \lceil \ln(n) \rceil$, let $\mathcal{F}(t,d)$ denote a family of representatives for the isomorphism classes of all $k$-graphs $F_t$ that algorithm \textsc{Grow} can possibly generate after exactly $t$ iterations of the while-loop with exactly $d$ of those $t$ iterations being degenerate. Let $f(t,d) := |\mathcal{F}(t,d)|$. 

\begin{claim}\label{claim:polylog}
There exist constants $C_0 = C_0(H_1, H_2)$ and $A = A(H_1,H_2)$ such that,
for any input $G'$ and any $\eps \geq 0$,
\[f(t,d)~\leq\lceil\ln(n)\rceil^{(C_0 +1)d}\cdot~A^{t-d}\] for $n$ sufficiently large.
\end{claim}


\begin{proof}
By Claim~\ref{claim:growv}, in every iteration $i$ of the while-loop of \textsc{Grow}, we add new edges onto $F_i$. These new edges span a $k$-graph on at most \[K := v_2 + (e_2 - 1)(v_1 - k)\] vertices.  
Thus $v(F_t) \leq v_1 + Kt$. 
Let $\mathcal{G}_K$ denote the set of all $k$-graphs on at most $K$ vertices. 
In iteration $i$ of the while-loop, $F_{i+1}$ is uniquely defined if one specifies the $k$-graph $G \in \mathcal{G}_K$ with edges $E(F_{i+1})\setminus E(F_{i})$, the number $y$ of vertices in which $G$ intersects $F_i$,
and two ordered lists of vertices from $G$ and $F_i$ respectively of length $y$, which specify the mapping of the intersection vertices from $G$ onto $F_i$. 
Thus, the number of ways that $F_i$ can be extended to $F_{i+1}$ is bounded from above by \[\sum_{G \in \mathcal{G}_K}\sum_{y = 2}^{v(G)}v(G)^yv(F_i)^y \leq |\mathcal{G}_K|\cdot K \cdot K^K(v_1 + Kt)^K \leq \lceil \ln(n) \rceil^{C_0},\] where $C_0$ depends only on $v_1$, $v_2$ and $e_2$, and $n$ is sufficiently large.
The last inequality follows from the fact that $t < \ln(n)$ as otherwise the while-loop would have already ended. 

Recall that, since \textsc{Eligible-Edge} determines the exact position where to attach the copy of $H_2$, in non-degenerate iterations $i$ there are at most \[k!e_2(k!e_1)^{e_2-1} =: A\] ways to extend $F_i$ to $F_{i+1}$, 
where the coefficients of $k!$ correspond with the orientations of the edge of the copy of $H_2$ we attach to $F_i$ and the edges of the copies of $H_1$ we attach to said copy of $H_2$. 
Hence, for $0 \leq d \leq t < \lceil \ln(n) \rceil$, \[f(t,d) \leq \binom{t}{d}(\lceil \ln(n) \rceil^{C_0})^d \cdot A^{t-d} \leq \lceil \ln(n) \rceil^{(C_0 +1)d} \cdot A^{t-d},\] 
where the binomial coefficient corresponds to the choice of when in the $t$ iterations the $d$ degenerate iterations happen.\end{proof}

\begin{claim:conclusion2}
There exists a constant $b = b(H_1,H_2) > 0$ such that for all $p \leq bn^{-1/m_k(H_1,H_2)}$, $G^k_{n,p}$ does not contain any $k$-graph from $\tilde{\mathcal{F}}(H_1,H_2,n,\eps^*)$ a.a.s.
\end{claim:conclusion2}

\begin{proof}
Let $\mathcal{F}_1$ and $\mathcal{F}_2$ denote the classes of $k$-graphs that algorithm \textsc{Grow} can output in lines~\ref{line:growreturnfi} and \ref{line:growreturnminsub}, respectively. 
For each $F \in \mathcal{F}_1$, we have that $e(F) \geq \ln(n)$, as $F$ was generated in $\lceil \ln(n) \rceil$ iterations, each of which introduces at least one new edge by Claim~\ref{claim:growv}.
Moreover, Claims~\ref{claim:non-degen} and \ref{claim:degenfull} imply that $\lambda(F_i)$ is non-increasing. Thus, we have that $\lambda(F) \leq \lambda(F_0)$ for all $F \in \mathcal{F}_1$. 
For all $F \in \mathcal{F}_2$, we have that $\lambda(F) \leq -\gamma^* < 0$ due to the condition in line~\ref{line:growwhileconditions} of \textsc{Grow}.
Let $A := A(H_1, H_2)$ be the constant found in the proof of Claim~\ref{claim:polylog}.
Since we have chosen $F_0 \cong H_1$ as the seed of the growing procedure, it follows that for \[b := (Ae)^{-\lambda(F_0)-1} \leq 1,\] the expected number of copies of $k$-graphs from $\tilde{\mathcal{F}}$ in $G^k_{n,p}$ with $p \leq bn^{-1/m_k(H_1,H_2)}$ is bounded by 

\begin{align}
\sum_{F \in \tilde{\mathcal{F}}}n^{v(F)}p^{e(F)} & \leq  \sum_{F \in \tilde{\mathcal{F}}}b^{e(F)}n^{\lambda(F)} \label{eq:finalstart}                           \\
                                                 & \leq \sum_{F \in \mathcal{F}_1}(eA)^{(-\lambda(F_0) - 1)\ln(n)}n^{\lambda(F_0)} + \sum_{F \in \mathcal{F}_2}b^{e(F)}n^{-\gamma^*}\nonumber  \\ 
                                                 & = \sum_{F \in \mathcal{F}_1}A^{(-\lambda(F_0) - 1)\ln(n)}n^{-1} + \sum_{F \in \mathcal{F}_2}b^{e(F)}n^{-\gamma^*}. \nonumber                     
\end{align}
Observe that, \begin{equation}\label{eq:lambdaF_0}\lambda(F_0) = v_1 - \frac{e_1}{m_k(F_1,F_2)} = k - \frac{1}{m_k(F_2)} \geq 1.\end{equation}

By Claims~\ref{claim:growv}, \ref{claim:q_1} and \ref{claim:polylog}, and \eqref{eq:lambdaF_0}, we have that
\begin{align}
     \sum_{F \in \mathcal{F}_1}A^{(-\lambda(F_0) - 1)\ln(n)}n^{-1} & \leq \sum_{d = 0}^{\min\{t, q_1\}}f(\lceil \ln(n) \rceil,d)A^{(-\lambda(F_0) - 1)\ln(n)}n^{-1}\label{eq:final1} \\ 
     & \leq (q_1 + 1)\lceil \ln(n) \rceil^{(C_0 + 1)q_1} \cdot A^{\lceil \ln(n) \rceil}A^{(-\lambda(F_0) - 1)\ln(n)}n^{-1}\nonumber \\ 
     & \leq (\ln(n))^{2(C_0 + 1)q_1}n^{-1}.\nonumber
\end{align}
Observe that, by Claim~\ref{claim:growv}, if some $k$-graph $F \in \mathcal{F}_2$ is the output of \textsc{Grow} after precisely $t$ iterations of the while-loop then $e(F) \geq t$. Since $b < 1$, this implies 
\begin{equation}\label{eq:befbt}
    b^{e(F)} \leq b^t
\end{equation}
for such a $k$-graph $F$.
Using \eqref{eq:befbt} and Claims~\ref{claim:growv}, \ref{claim:q_1} and \ref{claim:polylog}, we have that 
\begin{align}
    \sum_{F \in \mathcal{F}_2}b^{e(F)}n^{-\gamma^*} & \leq \sum_{t = 0}^{\lceil \ln(n)  \rceil}\sum_{d = 0}^{\min\{t, q_1\}}f(t,d)b^tn^{-\gamma^*}\label{eq:final2} \\ 
    & \leq \sum_{t = 0}^{\lceil \ln(n)  \rceil}\sum_{d = 0}^{\min\{t, q_1\}}\lceil \ln(n) \rceil^{(C_0 +1)d} \cdot A^{t-d}(Ae)^{(-\lambda(F_0)-1)t}n^{-\gamma^*}\nonumber \\
    & \leq (\lceil \ln(n) \rceil + 1)(q_1 + 1)\lceil \ln(n) \rceil^{(C_0 +1)q_1} n^{-\gamma^*}\nonumber \\
    & \leq (\ln(n))^{2(C_0 + 1)q_1}n^{-\gamma^*}. \nonumber
\end{align}
Thus, by \eqref{eq:finalstart}, \eqref{eq:final1} and \eqref{eq:final2}, we have that $\sum_{F \in \tilde{\mathcal{F}}}n^{v(F)}p^{e(F)} = o(1)$.
\COMMENT{RH: Again we strictly need def of $\eps^*$ here so that we can use $\gamma^*>0$.}
Consequently, Markov's inequality implies that $G^k_{n,p}$ a.a.s.~contains no $k$-graph from $\tilde{\mathcal{F}}$.
\end{proof}

\section{Proof of Claim~\ref{claim:degenfull}}\label{app:degenfull}
Our strategy for proving Claim~\ref{claim:degenfull} revolves around comparing our degenerate iteration $i$ of the while-loop of algorithm \textsc{Grow} with any non-degenerate iteration which could have occurred instead.
In accordance with this strategy, we have the following technical lemma which will be crucial in proving Claim~\ref{claim:degenfull}.\footnote{More specifically, in proving Claim~\ref{claim:degen2}, stated later.} 
In order to state our technical lemma, we define the following families of $k$-graphs.

\begin{define}
\textnormal{Let $F$, $H_1$ and $H_2$ be $k$-graphs and $\hat{e} \in E(F)$. We define $\mathcal{H}(F, \hat{e}, H_1, H_2)$ to be the family of $k$-graphs constructed from $F$ in the following way: 
Attach a copy $H_{\hat{e}}$ of $H_2$ to $F$ such that $E(H_{\hat{e}}) \cap E(F) = \{\hat{e}\}$ and $V(H_{\hat{e}}) \cap V(F) = \hat{e}$. 
Then, for each edge $f \in E(H_{\hat{e}})\setminus\{\hat{e}\}$, attach a copy $H_f$ of $H_1$ to $F \cup H_{\hat{e}}$ such that $E(F \cup H_{\hat{e}}) \cap E(H_f) = \{f\}$ and $(V(F)\setminus \hat{e}) \cap V(H_f) = \emptyset$. }
\end{define}

\begin{figure}[!ht]
\begin{center}
\definecolor{ffqqqq}{rgb}{1,0,0}
\definecolor{qqqqff}{rgb}{0,0,1}
\definecolor{qqzzqq}{rgb}{0,0.6,0}
\definecolor{ffxfqq}{rgb}{1,0.4980392156862745,0}
\definecolor{yqqqyq}{rgb}{0.5019607843137255,0,0.5019607843137255}
\begin{tikzpicture}[line cap=round,line join=round,>=triangle 45,x=1cm,y=1cm]
\draw [line width=2pt] (6,5)-- (5,2);
\draw [line width=2pt] (6,5)-- (6,2);
\draw [line width=2pt] (6,5)-- (7,2);
\draw [line width=2pt] (9,5)-- (8,2);
\draw [line width=2pt] (9,5)-- (9,2);
\draw [line width=2pt] (9,5)-- (10,2);
\draw (7.25,1.75) node[anchor=north west] {\Large $F$};
\draw [line width=2pt] (6,5)-- (9,5);
\draw [line width=2pt,color=yqqqyq] (6,5)-- (5,7);
\draw [line width=2pt,color=ffxfqq] (5,7)-- (6,9);
\draw [line width=2pt,color=qqzzqq] (6,9)-- (9,9);
\draw [line width=2pt,color=qqqqff] (9,9)-- (10,7);
\draw [line width=2pt,color=ffqqqq] (10,7)-- (9,5);
\draw [line width=2pt,color=ffqqqq, dash pattern= on 8pt off 8pt] (10,7)-- (13,7);
\draw [line width=2pt,color=qqqqff, dash pattern= on 8pt off 8pt,dash phase=8pt] (10,7)-- (13,7);
\draw [line width=2pt,color=ffqqqq] (9,5)-- (11,4);
\draw [line width=2pt,color=ffqqqq] (11,4)-- (13,5);
\draw [line width=2pt,color=ffqqqq] (13,7)-- (13,5);
\draw [line width=2pt,color=qqqqff] (13,7)-- (14,10);
\draw [line width=2pt,color=qqqqff] (14,10)-- (11,11);
\draw [line width=2pt,color=qqqqff] (11,11)-- (9,9);
\draw [line width=2pt,color=qqzzqq] (6,9)-- (6,5);
\draw [line width=2pt,color=qqzzqq] (6,5)-- (3,8);
\draw [line width=2pt,color=qqzzqq] (9,9)-- (5,11);
\draw [line width=2pt,color=qqzzqq] (3,8)-- (5,11);
\draw [line width=2pt,color=ffxfqq, dash pattern= on 8pt off 8pt] (5,7)-- (3,8);
\draw [line width=2pt,color=yqqqyq, dash pattern= on 8pt off 8pt,dash phase=8pt] (5,7)-- (3,8);
\draw [line width=2pt,color=ffxfqq] (3,8)-- (2,11);
\draw [line width=2pt,color=ffxfqq] (2,11)-- (5,11);
\draw [line width=2pt,color=ffxfqq] (5,11)-- (6,9);
\draw [line width=2pt,color=yqqqyq] (3,8)-- (2,6);
\draw [line width=2pt,color=yqqqyq] (3,4)-- (2,6);
\draw [line width=2pt,color=yqqqyq] (3,4)-- (6,5);
\draw [color=yqqqyq](3.25,6.25) node[anchor=north west] {\Large $H_{f_5}$};
\draw [color=qqqqff](11.08,9.25) node[anchor=north west] {\Large $H_{f_2}$};
\draw [color=ffxfqq](2.9,10.5) node[anchor=north west] {\Large $H_{f_4}$};
\draw (7.1,7.4) node[anchor=north west] {\Large $H_{\hat{e}}$};
\draw [color=ffqqqq](10.75,6) node[anchor=north west] {\Large $H_{f_1}$};
\draw [color=qqzzqq](4.35,9.15) node[anchor=north west] {\Large $H_{f_3}$};
\draw (7.25,5) node[anchor=north west] {$\hat{e}$};
\begin{scriptsize}
\draw [fill=black] (6,5) circle (2.5pt);
\draw [fill=black] (9,5) circle (2.5pt);
\draw [fill=black] (5,7) circle (2.5pt);
\draw[color=yqqqyq] (5.7,6.25) node {$f_5$};
\draw [fill=black] (6,9) circle (2.5pt);
\draw[color=ffxfqq] (5.7,7.9) node {$f_4$};
\draw [fill=black] (9,9) circle (2.5pt);
\draw[color=qqzzqq] (7.62,8.67) node {$f_3$};
\draw [fill=black] (10,7) circle (2.5pt);
\draw[color=qqqqff] (9.2,7.9) node {$f_2$};
\draw[color=ffqqqq] (9.29,6.25) node {$f_1$};
\draw [fill=black] (13,7) circle (2.5pt);
\draw [fill=black] (11,4) circle (2.5pt);
\draw [fill=black] (13,5) circle (2.5pt);
\draw [fill=black] (14,10) circle (2.5pt);
\draw [fill=black] (11,11) circle (2.5pt);
\draw [fill=black] (3,8) circle (2.5pt);
\draw [fill=black] (5,11) circle (2.5pt);
\draw [fill=black] (2,11) circle (2.5pt);
\draw [fill=black] (2,6) circle (2.5pt);
\draw [fill=black] (3,4) circle (2.5pt);
\end{scriptsize}
\end{tikzpicture}
\caption{A $2$-graph $J \in \mathcal{H}(F, \hat{e}, C_5, C_6)\setminus \mathcal{H}^*(F, \hat{e}, C_5, C_6)$.}\label{fig:jexample}
\end{center}
\end{figure}

Notice that, during construction of a $k$-graph $J \in \mathcal{H}(F, \hat{e}, H_1, H_2)$, the edge of $H_{\hat{e}}$ intersecting at $\hat{e}$ and the edge of each copy $H_f$ of $H_1$ intersecting at an edge $f \in E(H_{\hat{e}})\setminus\{\hat{e}\}$
are not stipulated. 
That is, we may end up with different $k$-graphs after the construction process if we choose different edges of $H_{\hat{e}}$ to intersect $F$ at $\hat{e}$ and different edges of the copies $H_f$ of $H_1$ to intersect the edges in $E(H_{\hat{e}})\setminus\{\hat{e}\}$.
Observe that although $E(F \cup H_{\hat{e}}) \cap E(H_f) = \{f\}$ and $(V(F)\setminus \hat{e}) \cap V(H_f) = \emptyset$ for each $f \in E(H_{\hat{e}}) \setminus \{\hat{e}\}$, 
the construction may result in one or more $k$-graphs $H_f$ intersecting $H_{\hat{e}}$ in more than $k$ vertices, including possibly in vertices of $\hat{e}$ (e.g. $H_{f_3}$ in Figure~\ref{fig:jexample}).
Also, the $k$-graphs $H_f$ may intersect with each other in vertices and/or edges (e.g.~$H_{f_1}$ and $H_{f_2}$ in Figure~\ref{fig:jexample}).

Borrowing notation and language from \cite{msss}, for any $J \in \mathcal{H}(F, \hat{e}, H_1, H_2)$ we call $V_J := V(H_{\hat{e}})\setminus \hat{e}$ the \emph{inner vertices of $J$} and $E_J := E(H_{\hat{e}})\setminus\{\hat{e}\}$ the \emph{inner edges of $J$}. 
Let $H_{\hat{e}}^{J}$ be the \emph{inner $k$-graph} on vertex set $V_J \ \dot{\cup} \ \hat{e}$ and edge set $E_J$, and observe that this $k$-graph $H_{\hat{e}}^{J}$ is isomorphic to a copy of $H_2$ minus some edge.
Further, for each copy $H_f$ of $H_1$, we define $U_J(f) := V(H_f)\setminus f$ and $D_J(f) := E(H_f) \setminus\{f\}$ and call \[U_J := \bigcup_{\substack{{f \in E_J}}} U_J(f)\] the set of \emph{outer vertices} of $J$ and \[D_J := \bigcup_{\substack{{f \in E_J}}} D_J(f)\] the set of \emph{outer edges} of $J$.
Observe that the sets $U_J(f)$ may overlap with each other and, as noted earlier, with $V(H_{\hat{e}}^{J})$. However, the sets $D_J(f)$ may overlap only with each other. 
Further, define $\mathcal{H}^*(F, \hat{e}, H_1, H_2) \subseteq \mathcal{H}(F, \hat{e}, H_1, H_2)$ such that for any $J^* \in \mathcal{H}^*(F, \hat{e}, H_1, H_2)$ we have $U_{J^*}(f_1) \cap U_{J^*}(f_2) = \emptyset$ and $D_{J^*}(f_1) \cap D_{J^*}(f_2) = \emptyset$ for all $f_1, f_2 \in E_{J^*}$, $f_1 \neq f_2$, and $U_{J^*}(f) \cap V(H_{\hat{e}}^{J^*}) = \emptyset$ for all $f \in E_{J^*}$; that is, the copies of $H_1$ are, in some sense, pairwise disjoint.
Note that each $J^* \in \mathcal{H}^*(F, \hat{e}, H_1, H_2)$ corresponds with a non-degenerate iteration $i$ of the while loop of algorithm \textsc{Grow} when $F = F_i$, $J^* = F_{i+1}$ and $\hat{e}$ is the edge chosen by \textsc{Eligible-Edge}($F_i$). This observation will be very helpful several times later.
For any $J \in \mathcal{H}(F, \hat{e}, H_1, H_2)$, define \[v^{+}(J) := |V(J)\setminus V(F)| = v(J) - v(F)\] and \[e^{+}(J) := |E(J)\setminus E(F)| = e(J) - e(F),\] and call $\frac{e^{+}(J)}{v^{+}(J)}$ the \emph{$F$-external density of $J$}.
The following lemma relates the $F$-external density of any $J^* \in \mathcal{H}^*(F, \hat{e}, H_1, H_2)$ to that of any $J \in \mathcal{H}(F, \hat{e}, H_1, H_2)\setminus \mathcal{H}^*(F, \hat{e}, H_1, H_2)$.

\begin{lem}\label{lemma:21}
Let $F$ be a $k$-graph and $\hat{e} \in E(F)$. 
Then for any $J \in \mathcal{H}(F, \hat{e}, H_1, H_2)\setminus \mathcal{H}^*(F, \hat{e}, H_1, H_2)$ and any $J^* \in \mathcal{H}^*(F, \hat{e}, H_1, H_2)$, we have \[\frac{e^{+}(J)}{v^{+}(J)} > \frac{e^{+}(J^*)}{v^{+}(J^*)}.\]
\end{lem}  We prove Lemma~\ref{lemma:21} in Section~\ref{sec:lemma21}.\smallskip

Claim~\ref{claim:degenfull} will follow from the next two claims. We say that algorithm \textsc{Grow} encounters a \emph{degeneracy of type 1} in iteration $i$ of the while-loop if line~\ref{line:growRline} returns true, that is, $\exists R \in \mathcal{R}_{G'}\setminus \mathcal{R}_{F_i} : |V(R) \cap V(F_i)| \geq k$.
Note that the following claim requires that $m_k(H_1) > m_k(H_2)$.

\begin{claim}\label{claim:degen1}
There exists a constant $\kappa_1 = \kappa_1(H_1, H_2) > 0$ such that if procedure \textsc{Grow} encounters a degeneracy of type $1$ in iteration $i$ of the while-loop, we have \[\lambda(F_{i+1}) \leq \lambda(F_i) - \kappa_1.\]
\end{claim}

\begin{proof}
Let $F := F_i$ be the $k$-graph before the operation in line~\ref{line:growR-F_i+1} is carried out (that is, before $F_{i+1} \gets F_i \cup R$), 
let $R$ be the copy of $H_1$ merged with $F$ in line~\ref{line:growR-F_i+1} and let $F' := F_{i+1}$ be the output from line~\ref{line:growR-F_i+1}.
We aim to show there exists a constant $\kappa_1 = \kappa_1(H_1, H_2) > 0$ such that
\begin{equation*}
    \lambda(F) - \lambda(F') = v(F) - v(F') - \frac{e(F) - e(F')}{m_k(H_1, H_2)} \geq \kappa_1.
\end{equation*}
Choose any edge $\hat{e} \in E(F)$ (the edge $\hat{e}$ need not be in the intersection of $R$ and $F$).
Choose any $F^* \in \mathcal{H}^*(F, \hat{e}, H_1, H_2)$.
Our strategy is to compare our degenerate outcome $F'$ with $F^*$. 
As noted earlier, $F^*$ corresponds to a non-degenerate iteration of the while loop of algorithm \textsc{Grow} (if $\hat{e}$ was the edge chosen by \textsc{Eligible-Edge}). 
Then Claim~\ref{claim:non-degen} gives us that $\lambda(F) = \lambda(F^*)$. Then 
\begin{eqnarray*}
      \lambda(F) - \lambda(F') & = \lambda(F^*) - \lambda(F')
     = v(F^*) - v(F') - \frac{e(F^*) - e(F')}{m_k(H_1, H_2)}.
\end{eqnarray*}
Hence we aim to show that there exists $\kappa_1 = \kappa_1(H_1, H_2) > 0$ such that
\begin{equation}\label{eq:kappa1bound}
    v(F^*) - v(F') - \frac{e(F^*) - e(F')}{m_k(H_1, H_2)} \geq \kappa_1.
\end{equation}
Define $R'$ to be the $k$-graph with vertex set $V' := V(R) \cap V(F)$ and edge set $E' :=  E(R) \cap E(F)$, and let $v' := |V'|$ and $e' := |E'|$. Observe that $R' \subsetneq R$.
Since $F^*$ corresponds with a non-degenerate iteration of the while-loop of algorithm \textsc{Grow}, $H_2$ is (strictly) $k$-balanced and $H_1$ is (strictly) balanced with respect to $d_k(\cdot, H_2)$, we have
\begin{eqnarray}
  v(F^*) - v(F') - \frac{e(F^*) - e(F')}{m_k(H_1, H_2)} & = &  (e_2 - 1)(v_1 - k) + (v_2 - k) - (v_1 - v')\nonumber \\ & &  - \frac{(e_2 - 1)e_1 - (e_1 - e')}{m_k(H_1, H_2)}\nonumber \\
  & = & (e_2 - 1)(v_1 - k) + (v_2 - k) \nonumber \\
  & & - (e_2 - 1)\left(v_1 - k + \frac{1}{m_k(H_2)}\right) \nonumber \\
  & & + \frac{e_1 - e'}{m_k(H_1, H_2)} - (v_1 - v')\nonumber \\
  & = &  \frac{e_1 - e'}{m_k(H_1, H_2)} - (v_1 - v')\nonumber \\ 
  & = &  v' - k + \frac{1}{m_k(H_2)} - \frac{e'}{m_k(H_1, H_2)}. \label{eq:f1end}
\end{eqnarray}
Also, since \textsc{Grow} encountered a degeneracy of type 1, we must have $v' \geq k$. 
Hence, if $e' = 0$ then \[v' - k + \frac{1}{m_k(H_2)} - \frac{e'}{m_k(H_1, H_2)} \geq \frac{1}{m_k(H_2)} > 0.\] 
If $e' \geq 1$, then since $R$ is a copy of $H_1$, $H_1$ is \emph{strictly} balanced with respect to $d_k(\cdot, H_2)$ and $R' \subsetneq R$ with $|E(R')| = e' \geq 1$, we have that $0 < d_k(R', H_2) < m_k(H_1, H_2)$, and so 
\begin{equation}\label{eq:m2d2}
-\frac{1}{m_k(H_1, H_2)} > -\frac{1}{d_k(R', H_2)}.
\end{equation}
Then by \eqref{eq:f1end} and \eqref{eq:m2d2}, we have that 
\begin{equation*}
    v(F^*) - v(F') - \frac{e(F^*) - e(F')}{m_k(H_1, H_2)} = v' - k + \frac{1}{m_k(H_2)} - \frac{e'}{m_k(H_1, H_2)} > v' - k + \frac{1}{m_k(H_2)} - \frac{e'}{d_k(R', H_2)} = 0,
\end{equation*}
using the definition of $d_k(R', H_2)$. Thus \eqref{eq:kappa1bound} holds for \[\kappa_1 = \min_{\stackrel{R' \subsetneq R:}{e(R')\geq 1}}
\left\{ \frac{1}{m_k(H_2)}, \ v' - k + \frac{1}{m_k(H_2)} - \frac{e'}{m_k(H_1, H_2)}\right\}.\]\end{proof} 
We say that algorithm \textsc{Grow} encounters a \emph{degeneracy of type 2} in iteration $i$ of the while-loop if, when we call \textsc{Extend-L}$(F_i, e, G')$, the $k$-graph $L$ found in line~\ref{line:linl*} overlaps with $F_i$ in more than $k$ vertices, 
or if there exists some edge $e' \in E(L)\setminus E(F_i)$ such that the $k$-graph $R_{e'}$ found in line~\ref{line:re'} overlaps in more than $k$ vertices with $F'$.
In the following result, we transform $F'$ into the output of a non-degenerate iteration $F^*$ in three steps.
In the last step we require Lemma~\ref{lemma:21}.

\begin{claim}\label{claim:degen2}
    There exists a constant $\kappa_2 = \kappa_2(H_1, H_2) > 0$ such that if procedure \textsc{Grow} encounters a degeneracy of type 2 in iteration $i$ of the while-loop, we have \[\lambda(F_{i+1}) \leq \lambda(F_i) - \kappa_2.\]
\end{claim}

\begin{proof}
Let $F := F_i$ be the $k$-graph passed to \textsc{Extend-L} and let $F' := F_{i+1}$ be its output.
We aim to show that there exists a constant $\kappa_2 = \kappa_2(H_1, H_2) > 0$ such that 
\begin{equation}\label{eq:kappa2full}
    \lambda(F) - \lambda(F') = v(F) - v(F') - \frac{e(F) - e(F')}{m_k(H_1, H_2)} \geq \kappa_2.
\end{equation}
Recall that $F'$ would be one of a constant number of $k$-graphs if iteration $i$ was non-degenerate.
Our strategy is to transform $F'$ into the output of such a non-degenerate iteration $F^*$ in three steps
\[F' =: F^0 \overset{(i)}{\to} F^1 \overset{(ii)}{\to} F^2 \overset{(iii)}{\to} F^3 := F^*,\] with each step carefully resolving a different facet of a degeneracy of type 2. By Claim~\ref{claim:non-degen}, we have $\lambda(F) = \lambda(F^*)$, hence we have that 
\begin{eqnarray*}
      \lambda(F) - \lambda(F') & = \lambda(F^*) - \lambda(F') = \sum_{j=1}^3\left(\lambda(F^j) - \lambda(F^{j-1})\right) \\
     & = \sum_{j=1}^3\left(v(F^j) - v(F^{j-1}) - \frac{e(F^j) - e(F^{j-1})}{m_k(H_1, H_2)}\right).
\end{eqnarray*}
We shall show that there exists $\kappa_2 = \kappa_2(H_1, H_2) > 0$ such that 
\begin{equation}\label{eq:kappa2partial}
    \left(v(F^j) - v(F^{j-1}) - \frac{e(F^j) - e(F^{j-1})}{m_k(H_1, H_2)}\right) \geq \kappa_2
\end{equation}
for each $j \in \{1,2,3\}$, whenever $F^j$ and $F^{j-1}$ are not isomorphic. 
In each step we will look at a different structural property of $F'$ that may result from a degeneracy of type 2. 
We do not know the exact structure of $F'$, and so, for each $j$, step $j$ may not necessarily modify $F^{j-1}$. 
However, since $F'$ is not isomorphic to $F^*$, as $F'$ resulted from a degeneracy of type 2, we know for at least one $j$ that $F^j$ is not isomorphic to $F^{j-1}$. This will allow us to conclude \eqref{eq:kappa2full} from \eqref{eq:kappa2partial}.

We will now analyse the $k$-graph that \textsc{Extend-L} attaches to $F$ when a degeneracy of type 2 occurs. 
First of all, \textsc{Extend-L} attaches a $k$-graph $L \cong H_2$ to $F$ such that $L \in \mathcal{L}^*_{G'}$. 
Let $x$ be the number of new vertices that are added onto $F$ when $L$ is attached, that is, $x = |V(L)\setminus V(F)|$.
Since $L$ overlaps with the edge $e$ determined by \textsc{Eligible-Edge} in line~\ref{line:eligible-edge} of \textsc{Grow}, we must have that $x \leq v_2 - k$. 
Further, as $L \in \mathcal{L}_{G'}^*$, every edge of $L$ is covered by a copy of $H_1$.
Thus, since the condition in line~\ref{line:growRline} of \textsc{Grow} came out as false in iteration $i$, we must have that 

\begin{equation}\label{eq:uv}
    \mbox{for all}\ u_1, \ldots, u_k \in V(F) \cap V(L),\ \mbox{if}\ \{u_1, \ldots, u_k\} \in E(L)\ \mbox{then}\ \{u_1, \ldots, u_k\} \in E(F).
\end{equation} 
(By Claim~\ref{claim:growv}, \eqref{eq:uv} implies that $x \geq 1$ since $F$ must be extended by at least one edge.)

Let $L' \subseteq L$ denote the subhypergraph of $L$ obtained by removing every edge in $E(F) \cap E(L)$. Observe that $|V(L')| = |V(L)| = v_2$ and $|E(L')| \geq 1$ (see the remark above). 
\textsc{Extend-L} attaches to each edge $e' \in E(L')$ a copy $R_{e'}$ of $H_1$ in line~\ref{line:f'f're'} such that $E(L') \cap E(R_{e'}) = \{e'\}$.  
As the condition in line~\ref{line:growRline} of \textsc{Grow} came out as false, each $k$-graph $R_{e'}$ intersects $F$ in at most $k-1$ vertices and, hence, zero edges. Let \[L_R' := L' \cup \bigcup_{e' \in E(L')} R_{e'}.\] 
Then $F'$ is the same as $F \cup L_R'$, and since every $k$-graph $R_{e'}$ contains at most $k-1$ vertices of $F$, we have that $E(F') = E(F)\ \dot{\cup}\ E(L_R')$. Therefore, \[e(F') - e(F) = e(L_R').\]
Observe that $|V(F) \cap V(L')| = v_2 - x$ and so
\begin{alignat*}{2}
      v(F') - v(F)  & = v(L_R') - |V(F) \cap V(L_R')| \\
                    & = v(L_R') - (v_2 - x) - |V(F) \cap (V(L_R')\setminus V(L'))|. 
\end{alignat*}
\textbf{Transformation (i): $F^0 \to F^1$.} If $|V(F) \cap (V(L_R')\setminus V(L'))| \geq 1$, then we apply transformation (i), mapping $F^0$ to $F^1$:
For each vertex $v \in V(F) \cap (V(L_R')\setminus V(L'))$, transformation (i) introduces a new vertex $v'$. 
Every edge incident to $v$ in $E(F)$ remains connected to $v$ and all those edges incident to $v$ in $E(L_R')$ are redirected to $v'$. 
In $L_R'$ we replace the vertices in $V(F) \cap (V(L_R')\setminus V(L'))$ with the new vertices.
So now we have $|V(F) \cap (V(L_R')\setminus V(L'))| = 0$.  
Since $E(F) \cap E(L_R') = \emptyset$, the output of this transformation is uniquely defined. 
Moreover, the structure of $L_R'$ is completely unchanged. Hence, since $|V(F) \cap (V(L_R')\setminus V(L'))| \geq 1$, and  $|E(F')| = |E(F)\ \dot{\cup}\ E(L_R')|$ remained the same after transformation (i), we have that
\begin{equation*}
    v(F^1) - v(F^{0}) - \frac{e(F^1) - e(F^{0})}{m_k(H_1, H_2)} = |V(F) \cap (V(L_R')\setminus V(L'))| \geq 1.
\end{equation*}
\textbf{Transformation (ii): $F^1 \to F^2$.} Recall the definition of $\mathcal{H}(F, e, H_1, H_2)$. If $x \leq v_2 - (k + 1)$, then we apply transformation (ii), mapping $F^1$ to $F^2$ by replacing $L_{R}'$ with a $k$-graph $L_{R}''$ such that $F \cup L_{R}'' \in \mathcal{H}(F, e, H_1, H_2)$.

If $x = v_2 - k$, observe that already $F \cup L_{R}' \in \mathcal{H}(F, e, H_1, H_2)$ and we continue to transformation (iii). So assume $x \leq v_2 - (k+1)$.
Consider the proper subhypergraph $L_F := L[V(F) \cap V(L)] \subsetneq L$ obtained by removing all $x$ vertices in $V(L) \setminus V(F)$ and their incident edges from $L$.
Observe that $v(L_F) = v_2 - x \geq k+1$ and also that $L_F \subseteq F$ by \eqref{eq:uv}.
Assign labels to $V(L_F)$ so that $V(L_F) = \{y_1, \ldots, y_k,w_1, \ldots, w_{v_2 - (x+k)}\}$ where $e = \{y_1, \ldots, y_k\}$ and $w_1, \ldots, w_{v_2 - (x+k)}$ are arbitrarily assigned.
At the start of transformation (ii), we create $v_2 - (x+k)$ new vertices $w_1', \ldots, w'_{v_2 - (x+k)}$ and also new edges such that $\{y_1, \ldots, y_k,w_1', \ldots, w'_{v_2 - (x+k)}\}$ induces a copy $\hat{L}_F$ of $L_F$, and the following holds. 
For $W \subsetneq \{w_1, \ldots, w_{v_2 - (x+k)}\}$, define $W' := \{w'_i:w_i \in W\}$. For all $i_1, \ldots, i_k \in \{1, \ldots, v_2 - (x+k)\}$, $i_j \neq i_{j'}$ for all $j \neq j'$, and all subsets $W \subsetneq \{w_1, \ldots, w_{v_2 - (x+k)}\}$ and $Y \subsetneq \{y_1, \ldots, y_k\}$ with $|W \cup Y| = k$, we have \begin{alignat*}{2}
     & \mbox{if} \ \{w_{i_1},\ldots, w_{i_k}\} \in E(L_F)\ \mbox{then} \ \{w'_{i_1},\ldots ,w'_{i_k}\} \in E(\hat{L}_F); \\
     & \mbox{if} \ W \cup Y \in E(L_F)\ \mbox{then} \ W' \cup Y \in E(\hat{L}_F); \\
     & \mbox{and} \ e = \{y_1, \ldots, y_k\} \in E(\hat{L}_F).
\end{alignat*} We also transform $L_R'$. For each edge in $E(L_{R}')$ containing some $W \subsetneq \{w_1, \ldots, w_{v_2 - (x+k)}\}$ in $L_F$, replace $W$ with $W'$, and remove $w_1, \ldots, w_{v_2 - (x+k)}$ from $V(L_R')$. 
Hence the structure of $L_R'$ remains the same except for the vertices $w_1, \ldots, w_{v_2 - (x+k)}$ that we removed.
Define $L'' := L_R' \cup \hat{L}_F$ and observe that $V(L'') \cap V(F) = e$.  

Continuing transformation (ii), for each $e' \in E(\hat{L}_{F})\setminus \{e\}$ , attach a copy $R_{e'}$ of $H_1$ to $L''$ such that $E(R_{e'}) \cap E(L'' \cup F) = \{e'\}$ and $V(R_{e'}) \cap V(L'' \cup F) = e'$.
That is, all these new copies $R_{e'}$ of $H_1$ are, in some sense, pairwise disjoint. Observe that $E(L'') \cap E(F) = \{e\}$ and
define $$L_R'' := L'' \cup \bigcup_{e' \in E(\hat{L}_F)\setminus \{e\}} R_{e'}.$$ Then $F \cup L_R'' \in \mathcal{H}(F, e, H_1, H_2)$. (See Figure~\ref{fig:trans2} for an example of transformation (ii) when $k = 2$.)
Let $F^2 := F \cup L_R''$. Then, \begin{alignat*}{2}
        \                            & v(F^2) - v(F^1) - \frac{e(F^2) - e(F^1)}{m_k(H_1, H_2)} \\
    = \                               & (e(\hat{L}_F) - 1)(v_1 - k) + v(\hat{L}_F) - k - \frac{(e(\hat{L}_F) - 1)e_1}{m_k(H_1, H_2)} \\
    = \                               & v(\hat{L}_F) - k - \frac{(e(\hat{L}_F) - 1)}{m_k(H_2)} \\
     =  \                              & \frac{(v(\hat{L}_F) - k)\left(m_k(H_2) - \frac{e(\hat{L}_F) - 1}{v(\hat{L}_F) - k}\right)}{m_k(H_2)} \\
    \geq \                               & \alpha_1
\end{alignat*} for some $\alpha_1 = \alpha_1(H_1, H_2) > 0$,
where the second equality follows from $H_1$ being (strictly) balanced with respect to $d_k(\cdot, H_2)$,
the third equality follows from $v(\hat{L}_F) = v(L_F) \geq k+1$ and the last inequality follows from $\hat{L}_F$ being a copy of $L_F \subsetneq L \cong H_2$ and $H_2$ being \emph{strictly} $k$-balanced.

\begin{figure}[!ht]
\begin{center}
\definecolor{ffqqqq}{rgb}{1,0,0}
\definecolor{yqqqyq}{rgb}{0.8,0,0.8}
\begin{tikzpicture}[line cap=round,line join=round,>=triangle 45,x=1cm,y=1cm]
\draw [line width=2pt,color=yqqqyq] (-8,0)-- (-7,0);
\draw [line width=2pt] (-7,0)-- (-6,0);
\draw [line width=2pt,color=yqqqyq] (-6,0)-- (-5,-1);
\draw [line width=2pt,color=yqqqyq] (-5,-1)-- (-5,-2);
\draw [line width=2pt,color=yqqqyq] (-5,-2)-- (-7,0);
\draw [line width=2pt,color=ffqqqq] (-9,1)-- (-8,0);
\draw [line width=2pt,color=ffqqqq] (-9,1)-- (-9,2);
\draw [line width=2pt,color=ffqqqq] (-9,2)-- (-8,3);
\draw [line width=2pt,color=ffqqqq] (-8,3)-- (-6,0);
\draw [line width=2pt] (-10.5,1.5)-- (-9,1);
\draw [line width=2pt] (-10.5,1.5)-- (-9,2);
\draw [line width=2pt] (-10.5,1.5)-- (-8,0);
\draw [line width=2pt] (-10.5,1.5)-- (-8,3);
\draw [line width=2pt] (-8,3)-- (-7.5,4);
\draw [line width=2pt] (-7.5,4)-- (-6,0);
\draw [line width=2pt] (-2.5,1.5)-- (-1,2);
\draw [line width=2pt] (-2.5,1.5)-- (-1,1);
\draw [line width=2pt] (-2.5,1.5)-- (0,0);
\draw [line width=2pt] (-2.5,1.5)-- (0,3);
\draw [line width=2pt,color=ffqqqq] (0,3)-- (-1,2);
\draw [line width=2pt,color=ffqqqq] (-1,2)-- (-1,1);
\draw [line width=2pt,color=ffqqqq] (-1,1)-- (0,0);
\draw [line width=2pt,color=ffqqqq] (0,3)-- (1,3);
\draw [line width=2pt,color=yqqqyq] (1,3)-- (2,2);
\draw [line width=2pt,color=yqqqyq] (2,2)-- (2,1);
\draw [line width=2pt,color=yqqqyq] (2,1)-- (1,0);
\draw [line width=2pt,color=yqqqyq] (0,0)-- (1,0);
\draw [line width=2pt] (1,0)-- (2,0);
\draw [line width=2pt] (2,0)-- (3,-1);
\draw [line width=2pt] (3,-1)-- (3,-2);
\draw [line width=2pt] (3,-2)-- (1,0);
\draw [line width=2pt] (1,0)-- (2.5,0.5);
\draw [line width=2pt] (2,1)-- (2.5,0.5);
\draw [line width=2pt] (2,1)-- (3,1.5);
\draw [line width=2pt] (2,2)-- (3,1.5);
\draw [line width=2pt] (0,3)-- (0.5,4);
\draw [line width=2pt] (1,3)-- (0.5,4);
\draw [line width=2pt] (1,3)-- (2,3);
\draw [line width=2pt] (2,2)-- (2,3);
\draw (0.5,-1.8) node[anchor=north west] {\LARGE $F$};
\draw [color=yqqqyq](0.45,2) node[anchor=north west] {\LARGE $\hat{L}_{F}$};
\draw [color=yqqqyq](-5.25,0) node[anchor=north west] {\LARGE $L_{F}$};
\draw (-7.5,-1.8) node[anchor=north west] {\LARGE $F$};
\draw [line width=2pt] (-8,0)-- (-8.5,-1.5);
\draw [line width=2pt] (-8,0)-- (-8,-1.5);
\draw [line width=2pt] (-8,0)-- (-7.5,-1.5);
\draw [line width=2pt] (-7,0)-- (-7,-1.5);
\draw [line width=2pt] (-7,0)-- (-6.5,-1.5);
\draw [line width=2pt] (-5,-2)-- (-6,-2);
\draw[color=black] (-3.85,2.25) node {\bf (ii)};
\draw [->,line width=5.2pt] (-4.5,1.5) -- (-3,1.5);
\draw [line width=2pt] (0,0)-- (-0.5,-1.5);
\draw [line width=2pt] (0,0)-- (0,-1.5);
\draw [line width=2pt] (0,0)-- (0.5,-1.5);
\draw [line width=2pt] (1,0)-- (1,-1.5);
\draw [line width=2pt] (1,0)-- (1.5,-1.5);
\draw [line width=2pt] (3,-2)-- (2,-2);
\begin{scriptsize}
\draw [fill=black] (-8,0) circle (2.5pt);
\draw[color=black] (-8,0.4) node {$y_1$};
\draw [fill=black] (-7,0) circle (2.5pt);
\draw[color=black] (-7,0.4) node {$y_2$};
\draw [fill=black] (-6,0) circle (2.5pt);
\draw[color=black] (-5.55,0) node {$w_{1}$};
\draw [fill=black] (-5,-1) circle (2.5pt);
\draw[color=black] (-4.55,-1) node {$w_{2}$};
\draw [fill=black] (-5,-2) circle (2.5pt);
\draw[color=black] (-4.55,-2) node {$w_{3}$};
\draw [fill=black] (-9,1) circle (2.5pt);
\draw [fill=black] (-9,2) circle (2.5pt);
\draw [fill=black] (-8,3) circle (2.5pt);
\draw[color=black] (-8.1,3.3) node {$a$};
\draw [fill=black] (-10.5,1.5) circle (2.5pt);
\draw [fill=black] (-7.5,4) circle (2.5pt);
\draw[color=black] (-7.65,4.3) node {$b$};
\draw [fill=black] (-2.5,1.5) circle (2.5pt);
\draw [fill=black] (-1,2) circle (2.5pt);
\draw [fill=black] (-1,1) circle (2.5pt);
\draw [fill=black] (0,0) circle (2.5pt);
\draw[color=black] (0,0.4) node {$y_1$};
\draw [fill=black] (0,3) circle (2.5pt);
\draw[color=black] (-0.1,3.3) node {$a$};
\draw [fill=black] (1,3) circle (2.5pt);
\draw[color=black] (0.35,4.3) node {$b$};
\draw[color=black] (1.25,3.25) node {$w'_{1}$};
\draw [fill=black] (2,2) circle (2.5pt);
\draw[color=black] (1.55,2) node {$w'_{2}$};
\draw [fill=black] (2,1) circle (2.5pt);
\draw[color=black] (1.55,1) node {$w'_{3}$};
\draw [fill=black] (1,0) circle (2.5pt);
\draw[color=black] (1,0.4) node {$y_2$};
\draw [fill=black] (2,0) circle (2.5pt);
\draw[color=black] (2.45,0) node {$w_1$};
\draw [fill=black] (3,-1) circle (2.5pt);
\draw[color=black] (3.45,-1) node {$w_2$};
\draw [fill=black] (3,-2) circle (2.5pt);
\draw[color=black] (3.45,-2) node {$w_3$};
\draw [fill=black] (2.5,0.5) circle (2.5pt);
\draw [fill=black] (3,1.5) circle (2.5pt);
\draw [fill=black] (0.5,4) circle (2.5pt);
\draw [fill=black] (2,3) circle (2.5pt);
\end{scriptsize}
\end{tikzpicture}
\caption{An example of transformation (ii) for $k = 2$, where $H_1 = K_3$ and $H_2 = C_8$. Observe that edges $aw_1$ and $bw_1$ are replaced by edges $aw'_1$ and $bw'_1$.}\label{fig:trans2}
\end{center}
\end{figure}

\textbf{Transformation (iii): $F^2 \to F^3$.} Recall that for any $J \in \mathcal{H}(F, \hat{e}, H_1, H_2)$, we define $v^{+}(J) := |V(J)\setminus V(F)| = v(J) - v(F)$ and $e^{+}(J) := |E(J)\setminus E(F)| = e(J) - e(F)$. 
Remove the edge $e$ from $E(L'')$ to give $E(L'') \cap E(F) = \emptyset$. 
Then \[e^+(F \cup L_R'') = e(L_R'')\] and \[v^+(F \cup L_R'') = v(L_R'') - k.\]

If $F^2 = F \cup L_R'' \in \mathcal{H}^*(F, e, H_1, H_2)$, then transformation (iii) sets $F^3 := F^2$.
Otherwise we have that $F \cup L_R'' \in \mathcal{H}(F, e, H_1, H_2)\setminus \mathcal{H}^*(F, e, H_1, H_2)$.
Let $F^3 := J^*$ where $J^*$ is any member of $\mathcal{H}^*(F, e, H_1, H_2)$ and recall that,
indeed, $J^*$ is equivalent to a non-degenerate iteration of the while-loop of \textsc{Grow}. 

Then, in transformation (iii), we replace $F \cup L_R''$ with the $k$-graph $J^*$. 
Since $H_2$ is (strictly) $k$-balanced and $H_1$ is (strictly) balanced with respect to $d_k(\cdot, H_2)$, we have that 
\begin{equation}\label{eq:m2h1h2}
    m_k(H_1, H_2) = \frac{e_1}{v_1 - k + \frac{1}{m_k(H_2)}} = \frac{e_1(e_2 - 1)}{(v_1 - k)(e_2 - 1) + v_2 - k} = \frac{e^{+}(J^*)}{v^{+}(J^*)}.
\end{equation}
Using \eqref{eq:m2h1h2} and Lemma~\ref{lemma:21}, together with $H_2$ being (strictly) $k$-balanced and $H_1$ being (strictly) balanced with respect to $d_k(\cdot, H_2)$, we have that
\begin{alignat*}{2}
      \                              & v(F^3) - v(F^2) - \frac{e(F^3) - e(F^2)}{m_k(H_1, H_2)} \\
    = \                              & v(J^*) - v(F \cup L_R'') - \frac{e(J^*) - e(F \cup L_R'')}{m_k(H_1, H_2)}  \\
    = \                              & v^+(J^*) - v^+(F \cup L_R'') - \frac{e^+(J^*) - e^+(F \cup L_R'')}{m_k(H_1, H_2)}  \\
    \overset{L.\ref{lemma:21}}{>}  \  & v^+(J^*) - v^+(F \cup L_R'') - \frac{e^+(J^*) - e^+(J^*)\left(\frac{v^+(F \cup L_R'')}{v^+(J^*)}\right)}{m_k(H_1, H_2)}  \\
     = \                              & \left(v^+(J^*) - v^+(F \cup L_R'')\right) \left(1 - \frac{e^+(J^*)}{v^+(J^*)m_k(H_1, H_2)}\right) \\
    = \                              & 0.  \end{alignat*} Since $v^+(J^*)$, $v^+(F \cup L_R'')$, $e^+(J^*)$, $e^+(F \cup L_R'')$ and $m_k(H_1, H_2)$ only rely on $H_1$ and $H_2$, there exists $\alpha_2 = \alpha_2(H_1, H_2) > 0$ such that \[v(F^3) - v(F^2) - \frac{e(F^3) - e(F^2)}{m_k(H_1, H_2)} \geq \alpha_2.\]
Taking \[\kappa_2 := \min\{1, \alpha_1, \alpha_2\}\] we see that \eqref{eq:kappa2partial} holds.\end{proof}

As stated earlier, Claim~\ref{claim:degenfull} follows from Claims~\ref{claim:degen1} and \ref{claim:degen2}. All that remains to prove Case~1 of Lemma~\ref{lemma:noerror} is to prove Lemma~\ref{lemma:21}.

\subsection{Proof of Lemma~\ref{lemma:21}}\label{sec:lemma21}

Let $J \in \mathcal{H}(F, \hat{e}, H_1, H_2)\setminus \mathcal{H}^*(F, \hat{e}, H_1, H_2)$.
We choose the $k$-graph $J^* \in \mathcal{H}^*(F, \hat{e}, H_1, H_2)$ with the following properties.

\begin{itemize}
    \item[1)] The edge of the copy $H_{\hat{e}}$ of $H_2$ in $J$ attached at $\hat{e}$, and which vertices of this edge intersect which vertices of $\hat{e}$ when attached, are the same as the edge of the copy $H^*_{\hat{e}}$ of $H_2$ in $J^*$ attached at $\hat{e}$ and which vertices of this edge intersect which vertices of $\hat{e}$ when attached; 
    \item[2)] for each $f \in E_{J}$, the edge of the copy $H_f$ of $H_1$ in $J$ attached at $f$, and which vertices of this edge intersect which vertices of $f$ when attached, are the same as the edge of the copy $H^*_f$ of $H_1$ in $J^*$ attached at $f$ and which vertices of this edge intersect which vertices of $f$ when attached.\COMMENT{JH: Maybe a little too wordy?}
\end{itemize}

Then, recalling definitions from the beginning of Section~\ref{app:degenfull}, we have that $V_J = V_{J^*}$ and $E_J = E_{J^*}$; that is, $H_{\hat{e}}^{J} = H_{\hat{e}}^{J^*}$. 
From now on, let $V := V_J$, $E := E_J$ and $H_{\hat{e}}^{-} := H_{\hat{e}}^{J}$.
Observe for \emph{all} $J' \in \mathcal{H}^*(F, \hat{e}, H_1, H_2)$, that
\begin{equation}
    \frac{e^{+}(J')}{v^{+}(J')} = \frac{e_1(e_2 - 1)}{(v_1 - k)(e_2 - 1) + v_2 - k}.
\end{equation} 
Hence, to prove Lemma~\ref{lemma:21} it suffices to show
\begin{equation*}
\frac{e^{+}(J)}{v^{+}(J)} > \frac{e^{+}(J^*)}{v^{+}(J^*)}.
\end{equation*}
The intuition behind our proof is that $J^*$ can be transformed into $J$ by successively merging the copies $H^*_f$ of $H_1$ in $J^*$ with each other and vertices in $H^{-}_{\hat{e}}$.
We do this in $e_2 - 1$ steps, fixing carefully a total ordering of the inner edges $E$.
For every edge $f \in E$, we merge the attached outer copy $H^*_f$ of $H_1$ in $J^*$ with copies of $H_1$ (attached to edges preceding $f$ in our ordering) and vertices of $H^{-}_{\hat{e}}$.
Throughout, we keep track of the number of edges $\Delta_e(f)$ and the number of vertices $\Delta_v(f)$ vanishing in this process.
One could hope that the $F$-external density of $J$ increases in every step of this process, or, even slightly stronger,
that $\Delta_e(f) / \Delta_v(f) < e^{+}(J^*) / v^{+}(J^*)$. 
This does not necessarily hold, but we will show that there exists a collection of edge-disjoint subhypergraphs $A_i$ of $H^{-}_{\hat{e}}$ such that, for each $i$, the edges of $E(A_i)$ are `collectively good' for this process and every edge not belonging to one of these $A_i$ is also `good' for this process.

Recalling definitions from the beginning of Section~\ref{app:degenfull}, let $H^{-}_f := (U_J(f)\ \dot{\cup}\ f, D_J(f))$ denote the subhypergraph obtained by removing the edge $f$ from the copy $H_f$ of $H_1$ in $J$. 

Later, we will carefully define a (total) ordering $\prec$ on the inner edges $E$.\footnote{For clarity, for any $f \in E$, $f \not\prec f$ in this ordering $\prec$.} For such an ordering $\prec$
and each $f \in E$, define \[\Delta_E(f) := D_J(f) \cap \left(\bigcup_{f' \prec f} D_J(f')\right),\] and \[\Delta_V(f) := U_J(f) \cap \left(\left(\bigcup_{f' \prec f} U_J(f')\right) \cup V(H_{\hat{e}}^{-})\right),\] 
and set $\Delta_e(f) := |\Delta_E(f)|$ and $\Delta_v(f) := |\Delta_V(f)|$. We emphasise here that the definition of $\Delta_v(f)$ takes into account how vertices of outer vertex sets can intersect with the inner $k$-graph $H_{\hat{e}}^{-}$.
One can see that $\Delta_e(f)$ ($\Delta_v(f)$) is the number of edges (vertices) vanishing from $H^*_f$ when it is merged with preceding attached copies of $H_1$ and $V(H_{\hat{e}}^{-})$.

By our choice of $J^*$, one can quickly see that
\begin{equation}\label{eq:sume}
    e^{+}(J) = e^{+}(J^*) - \sum_{f \in E} \Delta_e(f)
\end{equation}
and
\begin{equation}\label{eq:sumv}
    v^{+}(J) = v^{+}(J^*) - \sum_{f \in E} \Delta_v(f).
\end{equation}
By \eqref{eq:sume} and \eqref{eq:sumv}, we have
\begin{equation*}
    \frac{e^{+}(J)}{v^{+}(J)} = \frac{e^{+}(J^*) - \sum_{f \in E} \Delta_e(f)}{v^{+}(J^*) - \sum_{f \in E} \Delta_v(f)}.
\end{equation*}
Then, by Fact~\ref{fact:ineq}(iv), to show that 
\begin{equation*}
\frac{e^{+}(J)}{v^{+}(J)} > \frac{e^{+}(J^*)}{v^{+}(J^*)}
\end{equation*}
it suffices to prove that 
\begin{equation}\label{eq:conclusion2}
    \frac{\sum_{f \in E} \Delta_e(f)}{\sum_{f \in E} \Delta_v(f)} < \frac{e^{+}(J^*)}{v^{+}(J^*)}.\footnote{Note that $\sum_{f \in E} \Delta_v(f) \geq 1$ as otherwise $J = J^*$.}
\end{equation}
To show \eqref{eq:conclusion2}, we will now carefully order the edges of $E$ using an algorithm \textsc{Order-Edges} (Figure~\ref{orderedgesfig}).
The algorithm takes as input the $k$-graph $H_{\hat{e}}^{-} = (V \ \dot{\cup} \ \hat{e},E)$ and outputs a stack $s$ containing every edge from $E$ and a collection of edge-disjoint edge sets $E_i$ in $E$ and (not necessarily disjoint) vertex sets $V_i$ in $V \ \dot{\cup} \ \hat{e}$. 
We take our total ordering $\prec$ of $E$ to be that induced by the order in which edges of $E$ were placed onto the stack $s$ (that is, $f \prec f'$ if and only if $f$ was placed onto the stack $s$ before $f'$).
Also, for each $i$, we define $A_i := (V_i, E_i)$ to be the $k$-graph on vertex set $V_i$ and edge set $E_i$ and observe that $A_i \subsetneq H_2$.
We will utilise this ordering and our choice of $E_i$ and $V_i$ for each $i$, alongside $H_2$ being (strictly) $k$-balanced and $H_1$ being \emph{strictly} balanced with respect to $d_k(\cdot, H_2)$, in order to conclude \eqref{eq:conclusion2}. 

\begin{figure}
\begin{algorithmic}[1]
\Procedure{\sc Order-Edges}{$H_{\hat{e}}^{-} = (V\  \dot{\cup}\  \hat{e},E)$}
    \State $s\gets$ {\sc empty-stack}()\label{line:oes}
    \ForAll {$i \in [\lfloor e_2/2 \rfloor]$}
        \State $E_i \gets \emptyset$
        \State $V_i \gets \emptyset$
    \EndFor
    \State $j \gets 1$
    \State $E'\gets E$\label{line:oee'getse}
    \While{$E' \neq \emptyset$}\label{line:oee'notempty}
            \If{$\exists f, f' \in E' \ \mbox{s.t.} \ (f \neq f') \land (D_J(f) \cap D_J(f') \neq \emptyset)$}\label{line:oeaiconstructionstart}
                \State $s$.{\sc push}($f$)\label{line:oespushf}
                \State $E_j$.{\sc push}($f$)\label{line:oeeipushf}
                \State $E'$.{\sc remove}($f$)
                \State $V_j \gets f \cup (U_J(f) \cap V(H_{\hat{e}}^{-}))$\label{line:oeviupdate1}
                \While{$\exists \ u \in E' \ \mbox{s.t.} \ \left(\{v\in u\} \in V_j\right) \lor \left(D_J(u) \cap \bigcup_{f \in E_j} D_J(f) \neq \emptyset\right)$}\label{line:oeaibadedgescollectedup}
                    \State $s$.{\sc push}($u$)
                    \State $E_j$.{\sc push}($u$)\label{line:oeeipushuw}
                    \State $E'$.{\sc remove}($u$)
                    \State $V_j \gets \bigcup_{f \in E_j}\left(f \cup (U_J(f) \cap V(H_{\hat{e}}^{-})) \right)$
                \EndWhile\label{line:oeendwhile1}
                \State $j \gets j + 1$\label{line:oeigetsiplus1}    
            \Else
                \ForAll {$f \in E'$}\label{line:oeconclusion}
                \State $s$.{\sc push}($f$)\label{line:oespushlastf}
                 \State $E'$.{\sc remove}($f$)
                \EndFor \label{line:oeendfor2}
            \EndIf
    \EndWhile
    \State\Return $s$\label{line:oereturns}
    \ForAll {$i \in [\lfloor e_2/2 \rfloor]$ s.t. $E_{i} \neq \emptyset$}\label{line:oelastforalli}
    \State    \Return $E_i$
    \State    \Return $V_i$\label{line:oereturnvj}
    \EndFor
\EndProcedure
\end{algorithmic}
\caption{The implementation of algorithm \textsc{Order-Edges}.}\label{orderedgesfig}
\end{figure}

Let us describe algorithm \textsc{Order-Edges} (Figure~\ref{orderedgesfig}) in detail. 
In lines~\ref{line:oes}-\ref{line:oee'getse}, we initialise several parameters: a stack $s$, which we will place edges of $E$ on during our algorithm; sets $E_i$ and $V_i$ for each $i \in [\lfloor e_2/2 \rfloor]$,\footnote{\textsc{Order-Edges} can output at most $\lfloor e_2/2 \rfloor$ pairs of sets $E_i$ and $V_i$.} which we will add edges of $E$ and vertices of $V \ \dot{\cup} \ \hat{e}$ into, respectively;
an index $j$, which will correspond to whichever $k$-graph $A_j$ we consider constructing next; and a set $E'$, which will keep track of those edges of $E$ we have not yet placed onto the stack $s$. 
Line~\ref{line:oee'notempty} ensures the algorithm continues until $E' = \emptyset$, that is, until all the edges of $E$ have been placed onto $s$.

In line~\ref{line:oeaiconstructionstart}, we begin constructing $A_j$ by finding a pair of distinct edges in $E'$ whose outer edge sets (in $J$) intersect.
In lines~\ref{line:oespushf}-\ref{line:oeviupdate1}, we place one of these edges, $f$, onto $s$, into $E_j$ and remove it from $E'$. We also set $V_j$ to be the $k$ vertices in $f$ alongside any vertices in the outer vertex set $U_J(f)$ that intersect $V(H^{-}_{\hat{e}})$.

In lines~\ref{line:oeaibadedgescollectedup}-\ref{line:oeendwhile1}, we iteratively add onto $s$, into $E_j$ and remove from $E'$ any edge $u \in E'$ which either contains $k$ vertices previously added to $V_j$ or has an outer edge set $D_J(u)$ that intersects the collection of outer edge sets of edges previously added to $E_j$.
We also update $V_j$ in each step of this process.

In line~\ref{line:oeigetsiplus1}, we increment $j$ in preparation for the next check at line~\ref{line:oeaiconstructionstart} (if we still have $E' \neq \emptyset$). 
If the condition in line~\ref{line:oeaiconstructionstart} fails then in lines~\ref{line:oeconclusion}-\ref{line:oeendfor2} we arbitrarily place the remaining edges of $E'$ onto the stack $s$. In line~\ref{line:oereturns}, we output the stack $s$ and in lines~\ref{line:oelastforalli}-\ref{line:oereturnvj} we output each non-empty $E_i$ and $V_i$.

We will now argue that each proper subhypergraph $A_i = (V_i, E_i)$ of $H_2$ and each edge placed onto $s$ in line~\ref{line:oespushlastf} are `good', in some sense, for us to conclude \eqref{eq:conclusion2}. 

For each  $f \in E$, define the $k$-graph \[T(f) := (\Delta_V(f) \ \dot{\cup} \ f, \Delta_E(f)) \subseteq H_f^- \subsetneq H_1.\]
Observe that one or more vertices of $f$ may be isolated in $T(f)$. This observation will be very useful later.

For each $i$ and $f \in E_i$, define \[(V(H_{\hat{e}}^{-}))_f := (V(H_{\hat{e}}^{-}) \cap U_J(f))\setminus \left(\bigcup_{\substack{f' \in E_i: \\ f' \prec f}}f' \cup \left(\bigcup_{\substack{f' \in E_i: \\ f' \prec f}} \left(V(H_{\hat{e}}^{-}) \cap U_J(f')\right)\right)\right) \subseteq \Delta_V(f)\] since $(V(H_{\hat{e}}^{-}))_f$ is a subset of $V(H_{\hat{e}}^{-}) \cap U_J(f)$. 
One can see that $(V(H_{\hat{e}}^{-}))_f$ consists of those vertices of $V(H_{\hat{e}}^{-})$ which are new to $V_i$ at the point when $f$ is added to $E_i$ but are not contained in $f$. 
Importantly for our purposes, every vertex in $(V(H_{\hat{e}}^{-}))_f$ is isolated in $T(f)$. 
Indeed, otherwise there exists $\ell < i$ and $f'' \in E_{\ell}$ such that $D_J(f) \cap D_J(f'') \neq \emptyset$ and $f$ would have been previously added to $E_{\ell}$ in line~\ref{line:oeeipushuw}. 

For each $f \in E$, let $T'(f)$ be the $k$-graph obtained from $T(f)$ by removing all isolated vertices from $V(T(f))$.
Crucially for our proof, since vertices of $f$ may be isolated in $T(f)$, one or more of them may not belong to $V(T'(f))$. 
Further, no vertex of $(V(H_{\hat{e}}^{-}))_f$ is contained in $V(T'(f))$.

For all $f \in E$ with $\Delta_e(f) \geq 1$, since $T'(f) \subsetneq H_1$ and $H_1$ is \emph{strictly} balanced with respect to $d_k(\cdot, H_2)$, we have that 
\begin{equation}\label{eq:centralobservation}
    m_k(H_1, H_2) > d_k(T'(f), H_2) = \frac{|E(T'(f))|}{|V(T'(f))| - k + \frac{1}{m_k(H_2)}}.
\end{equation} 
Recall \eqref{eq:m2h1h2}, that is, \[m_k(H_1, H_2) = \frac{e^+(J^*)}{v^+(J^*)}.\]
We now make the key observation of our proof: Since vertices of $f$ may be isolated in $T(f)$, and so not contained in $V(T'(f))$, and no vertex of $(V(H_{\hat{e}}^{-}))_f$ is contained in $V(T'(f))$, we have that 
\begin{equation}\label{eq:vt'fbound}
    |V(T'(f))| \leq \Delta_v(f) + |f \cap V(T'(f))| - |(V(H_{\hat{e}}))_f|.
\end{equation} 
Hence, from \eqref{eq:centralobservation} and \eqref{eq:vt'fbound} we have that \begin{alignat}{2}
    \Delta_e(f)     =  \ &|E(T'(f))| \nonumber\\
                    < \ &m_k(H_1, H_2)\left(\Delta_v(f) - (k - |f \cap V(T'(f))|) - |(V(H_{\hat{e}}^{-}))_f| + \frac{1}{m_k(H_2)}\right).\label{eq:deltaef}
\end{alignat}
Those edges $f$ such that $|f \cap V(T'(f))| = k$ will be, in some sense, `bad' for us when trying to conclude \eqref{eq:conclusion2}.
Indeed, if $|(V(H_{\hat{e}}^{-}))_f| = 0$ then we may have that $\frac{\Delta_e(f)}{\Delta_v(f)} \geq m_k(H_1, H_2) = \frac{e^{+}(J^*)}{v^{+}(J^*)}$, by \eqref{eq:m2h1h2}.
However, edges $f$ such that $|f \cap V(T'(f))| < k$ will be, in some sense, `good' for us when trying to conclude \eqref{eq:conclusion2}. 
Indeed, since $m_k(H_2) \geq 1$, we have for such edges $f$ that $\frac{\Delta_e(f)}{\Delta_v(f)} < m_k(H_1, H_2) = \frac{e^{+}(J^*)}{v^{+}(J^*)}$.\COMMENT{JH: I don't know if we've observed this before, but this part of the proof here is also somewhere where we definitely need $m_k(H_2) \geq 1$, and so can't have just the assumption $m_k(H_2) \geq 1/(k-1)$ - before I was curious whether you only needed it when proving the finiteness of $\hat{\mathcal{B}}$.}

We show in the following claim that our choice of ordering $\prec$, our choice of each $A_i$ and the fact that $H_2$ is (strictly) $k$-balanced ensure that for each $i$ there are enough `good' edges in $A_i$ to compensate for any `bad' edges that may appear in $A_i$.
 
\begin{claim}\label{claim:inei}
For each $i$, 
\begin{equation*}
    \sum_{f \in E_i} \Delta_e(f) < m_k(H_1, H_2) \sum_{f \in E_i} \Delta_v(f).
\end{equation*}
\end{claim}

\begin{proof}
Fix $i$. Firstly, as observed before, each $A_i$ is a non-empty subgraph of $H_2$. 
Moreover, $|E_i| \geq 2$ (by the condition in line~\ref{line:oeaiconstructionstart}).
Since $H_2$ is (strictly) $k$-balanced, we have that 
\begin{equation}\label{eq:f22balancedbound}
    m_k(H_2) \geq d_k(A_i) = \frac{|E_i| - 1}{|V_i| - k}.
\end{equation}
Now let us consider $\Delta_e(f)$ for each $f \in E_i$. 
For the edge $f$ added in line~\ref{line:oeeipushf}, observe that $\Delta_e(f) = 0$. 
Indeed, otherwise $\Delta_e(f) \geq 1$, and there exists $\ell < i$ such that $D_J(f) \cap \left(\bigcup_{f' \in E_{\ell}}D_J(f')\right) \neq \emptyset$. 
That is, $f$ would have been added to $E_{\ell}$ in line~\ref{line:oeeipushuw} previously in the algorithm. 
Since $(V(H_{\hat{e}}^{-}))_f \subseteq \Delta_V(f)$, we have that 
\begin{equation}\label{eq:aif}
    \Delta_e(f) = 0 \leq m_k(H_1, H_2)\left(\Delta_v(f) - |(V(H_{\hat{e}}^{-}))_f|\right).    
\end{equation}
Now, for each edge $u = \{u_1, \ldots, u_k\}$ added in line~\ref{line:oeeipushuw}, either $u_1,\ldots, u_k \in V_i$ when $u$ was added to $E_i$, or we had $D_J(u) \cap \left(\bigcup_{\substack{f' \in E_i: \\ f' \prec u}}D_J(f')\right) \neq \emptyset$, that is,
$\Delta_e(u) \geq 1$, and at least one of $u_1,\ldots, u_k$ did not belong to $V_i$ when $u$ was added to $E_i$.
In the former case, if $\Delta_e(u) \geq 1$, then by \eqref{eq:deltaef} and that $-(k - |u \cap V(T'(u))|) \leq 0$, we have that
\begin{equation}\label{eq:aiuwbad}
    \Delta_e(u) < m_k(H_1, H_2)\left(\Delta_v(u) - |(V(H_{\hat{e}}^{-}))_{u}| + \frac{1}{m_k(H_2)}\right).
\end{equation}
Observe that \eqref{eq:aiuwbad} also holds when $\Delta_{e}u) = 0$. In the latter case, observe that for all $\ell < i$, we must have $D_{J}(u) \cap \left(\bigcup_{f \in E_{\ell}}D_J(f)\right) = \emptyset$.
Indeed, otherwise $u$ would have been added to some $E_{\ell}$ in line~\ref{line:oeeipushuw}. 
Let $W_u\subseteq \{u_1, \ldots, u_k\}$ non-empty set (by assumption) which did not belong to $V_i$ before $u$ was added to $E_i$. Note that every vertex of $W_u$ is isolated in $T(u)$.
That is, no vertex in $W_u$ belongs to $T'(u)$, and so $|u \cap V(T'(u))| \in \{0, \ldots, k-1\}$. Thus, since $\Delta_e(u) \geq 1$, by \eqref{eq:deltaef} we have that 
\begin{equation}\label{eq:aiuwgood1}
    \Delta_e(u) < m_k(H_1, H_2)\left(\Delta_v(u) - |W-u| - |(V(H_{\hat{e}}^{-}))_{uw}| + \frac{1}{m_k(H_2)}\right).
\end{equation} 


In conclusion, except for the $k$ vertices in the edge added in line~\ref{line:oeeipushf}, every time a new vertex $x$ was added to $V_i$ when some edge $f$ was added to $E_i$, either $x \in (V(H_{\hat{e}}^{-}))_{f}$, or $x \in W_f$ and \eqref{eq:aiuwgood1} held.
Indeed, $x$ was isolated in $T(f)$. 
Moreover, after the $k$ vertices in the edge added in line~\ref{line:oeeipushf} there are $|V_i| - k$ vertices added to $V_i$.

Hence, by \eqref{eq:aif}-\eqref{eq:aiuwgood1}, we have that 
\begin{equation}\label{eq:aifinal1}
    \sum_{f \in E_i} \Delta_e(f) < m_k(H_1, H_2)\left(\sum_{f \in E_i} \Delta_v(f) - (|V_i| - k) + \frac{|E_i| - 1}{m_k(H_2)}\right).
\end{equation}
By \eqref{eq:f22balancedbound}, 
\begin{equation}\label{eq:aifinal2}
    - (|V_i| - k) + \frac{|E_i| - 1}{m_k(H_2)} \leq 0.
\end{equation}
Thus, by \eqref{eq:aifinal1} and \eqref{eq:aifinal2}, we have that
\[\sum_{f \in E_i} \Delta_e(f) < m_k(H_1, H_2) \sum_{f \in E_i} \Delta_v(f)\] as desired.\end{proof}

\begin{claim}\label{claim:notinei}
For each edge $f$ placed onto $s$ in line~\ref{line:oespushlastf}, we have $\Delta_e(f) = 0$.
\end{claim}

\begin{proof}
Assume not. Then $\Delta_e(f) \geq 1$ for some edge $f$ placed onto $s$ in line~\ref{line:oespushlastf}. 
Observe that $D_J(f) \cap \left(\bigcup_{f'' \in E_i} D_J(f'')\right) = \emptyset$ for any $i$, otherwise $f$ would have been added to some $E_i$ in line~\ref{line:oeeipushuw} previously.

Thus we must have that $D_J(f) \cap D_J(f') \neq \emptyset$ for some edge $f' \neq f$ where $f'$ was placed onto $s$ also in line~\ref{line:oespushlastf}.
But then $f$ and $f'$ satisfy the condition in line~\ref{line:oeaiconstructionstart} and would both be contained in some $E_i$, contradicting that $f$ was placed onto $s$ in line~\ref{line:oespushlastf}.\end{proof}

Since $J \in \mathcal{H}(F, \hat{e}, H_1, H_2)\setminus \mathcal{H}^*(F, \hat{e}, H_1, H_2)$, we must have that $\Delta_v(f) \geq 1$ for some $f \in E$. 
Thus, if $\Delta_e(f) = 0$ for all $f \in E$ then \eqref{eq:conclusion2} holds trivially. If $\Delta_e(f) \geq 1$ for some $f \in E$, then $E_1 \neq \emptyset$ and $A_1$ is a non-empty subhypergraph of $H_{\hat{e}}^{-}$. 
Then, by \eqref{eq:m2h1h2} and Claims~\ref{claim:inei} and \ref{claim:notinei}, we have that 
\begin{alignat*}{2}
    \frac{\sum_{f \in E} \Delta_e(f)}{\sum_{f \in E} \Delta_v(f)} & = \frac{\sum_i\sum_{f \in E_i} \Delta_e(f) + \sum_{f \in E\setminus \cup_i E_i} \Delta_e(f)}{\sum_{f \in E} \Delta_v(f)} \\
    & < \frac{m_k(H_1, H_2)\sum_i\sum_{f \in E_i} \Delta_v(f)}{\sum_i\sum_{f \in E_i} \Delta_v(f)} \\
    & = m_k(H_1, H_2) \\
    & = \frac{e^{+}(J^*)}{v^{+}(J^*)}.
\end{alignat*}
Thus \eqref{eq:conclusion2} holds and we are done.\qed

\section{The \texorpdfstring{$m_k(H_1) = m_k(H_2)$}{equality} case}\label{app:mh1mh2equal}

In this section (in the first two subsections) we prove Case~2 of Lemma~\ref{lemma:noerror}, that is, when $m_k(H_1) = m_k(H_2)$. 
Our proof follows that of Case~1 significantly, but uses different procedures \textsc{Grow-$\hat{\mathcal{B}}_{\eps}$-Alt}, \textsc{Special-Alt}\footnote{\textsc{Special-Alt} is written verbatim the same as \textsc{Special}, but is morally different since the definition of $\hat{\mathcal{B}}$ differs depending on whether $m_k(H_1) > m_k(H_2)$ or $m_k(H_1) = m_k(H_2)$.} and \textsc{Grow-Alt}.
All definitions and notation are the same as previously unless otherwise stated. 
In the third subsection, we prove Lemma~\ref{lem:finiteequal}, which asserts finiteness of $\hat{\mathcal{B}}(H_1,H_2,\eps)$ for the case $m_k(H_1)=m_k(H_2)$. 

\subsection{Definition of $\hat{\mathcal{B}}_{\eps}$: $m_k(H_1)=m_k(H_2)$ case}\label{sec:growbalt}

Let $F$ be a $k$-graph. We define an edge $f$ to be \emph{open in $F$} if there exists no $L \in \mathcal{L}_{G}$ and $R \in \mathcal{R}_G$ such that $f = E(L) \cap E(R)$. Now, either attach a copy $L$ of $H_2$ to $F$ at some edge $e \in E(F)$ such that $E(L) \cap E(F) = \{e\}$ and $V(L) \cap V(F) = e$ or attach a copy $R$ of $H_1$ to $F$ at $e$ such that $E(R) \cap E(F) = \{e\}$ and $V(R) \cap V(F) = e$. Call the $k$-graph constructed in this process $F'$. We call the $k$-graph induced by the collection of edges in $J = L$ or $J = R$ -- depending which option was taken -- a \emph{fully open flower}\footnote{Using the name `flower' here may seem not as reflective of the structure in this case as in the $m_k(H_1) > m_k(H_2)$ case. We keep the same nomenclature so as to draw parallels with Case $1$.}
\COMMENT{JH:Too pompous? RH: `and so aid the reader in understanding.' deleted, now think it is fine.} and the edges belonging to $J\setminus e$ the \emph{petal edges of $J$}. If the construction of $F'$ from $F$ is different at all from the above, that is, if $J$ intersects $F$ in more than one edge or more than $k$ vertices, then $J$ is not a fully open flower. Later, in Section~\ref{sec:finiteequal}, we will see that the petal edges are in fact all open, hence the name `fully open flower'. For now, all we need observe is that the construction process of $F'$ yields that one can tell by inspection whether an edge belongs to the petal edges of some fully open flower or not without knowing anything about how $F$ and $F'$ were constructed. This may seem enigmatic currently, but this observation will allow us to define a useful function \textsc{Eligible-Edge-Alt}, used in algorithm \textsc{Grow-$\hat{\mathcal{B}}_{\eps}$-Alt} below and \textsc{Grow-Alt} = \textsc{Grow-Alt}$(G, \eps, n)$ in Appendix~\ref{sec:growalt}, in a way that does not increase the number of outputs of either algorithm.\COMMENT{JH (repeat comment): Not sure how happy I am with this paragraph.}

\begin{figure}[!ht]
\begin{algorithmic}[1]
\Procedure{Grow-$\hat{\mathcal{B}}_{\eps}$-Alt}{$H_1, H_2, \eps$}
    \State $F_0 \gets H_1$\label{line:growaltanyrV}
    \State $i \gets 0$\label{line:growalti0V}
    \While {$\forall \tilde{F}\subseteq F_i$:$\lambda(F) > -\gamma$}\label{line:growaltwhileconditionV}
        \If {$\exists e \in E(F_i)$ s.t. $e$ is open in $F_i$}\label{line:growaltifopeneV}
            \State $e \gets \textsc{Eligible-Edge-Alt}(F_i)$\label{line:avalteligible-edge}
            \State $F_{i+1} \gets$ \textsc{Close-$e$-Alt}($F_i, e$)\label{line:F_i+1getsextendaltV}
            \State $i \gets i + 1$\label{line:opengrowaltitoi+1V}
        \Else 
            \State \Return{$F_i$}\label{line:growaltreturnfiV}
            \State $e \gets$ any edge of $F_i$ \label{line:growaltanyedgeV}
            \State $F_{i+1} \gets$ \textsc{Close-$e$-Alt}($F_i, e$)\label{line:growaltifclosedeV}
            \State $i \gets i + 1$\label{line:closedgrowaltitoi+1V}
        \EndIf
    \EndWhile
\EndProcedure
\end{algorithmic}\smallskip\smallskip

\begin{algorithmic}[1]
\Procedure{Close-$e$-Alt}{$F,e$}
    \State $\{L,R\} \gets$ any pair $\{L,R\}$ such that $L \cong H_2$, $R \cong H_1$, $E(L) \cap E(R) = \{e\}$ and either $L \subseteq F \wedge R \not\subseteq F$ or $R \subseteq F \wedge L \not\subseteq F$ \label{line:lrv}
    \If{$L \nsubseteq F$}\label{line:iflv}
        \State $F' \gets F \cup L$\label{line:flv}
    \Else
        \State $F' \gets F \cup R$\label{line:frv}
    \EndIf
    \State \Return {$F'$}\label{line:growaltreturnf'v}
\EndProcedure
\end{algorithmic}\smallskip\smallskip

\caption{The implementation of algorithm \textsc{Grow-$\hat{\mathcal{B}}_{\eps}$-Alt}.}\label{fig:growbaltfig}
\end{figure}

Let us describe procedure \textsc{Grow-$\hat{\mathcal{B}}_{\eps}$-Alt}, as given in Figure~\ref{fig:growbaltfig}. There are a number of similarities between \textsc{Grow-$\hat{\mathcal{B}}_{\eps}$-Alt} and \textsc{Grow-$\hat{\mathcal{B}}_{\eps}$}, but we will repeat details here for clarity. 
Firstly, as with \textsc{Grow-$\hat{\mathcal{B}}_{\eps}$}, observe that \textsc{Grow-$\hat{\mathcal{B}}_{\eps}$-Alt} takes as input two $k$-graphs $H_1$ and $H_2$ with properties as described in Theorem~\ref{thm:colourb} and the constant $\eps \geq 0$. 
Note that there is no underlying $k$-graph determining at step $i$ what the possible structure of $F_i$ could be, as will be the case in \textsc{Grow-Alt} (in Section~\ref{sec:growalt}) later, hence whenever we have a choice in step $i$ of how to attach a $k$-graph we must consider every possible structure $F_{i+1}$ that could possibly result.
Utilising $k$-graph versions of results from \cite{h}, notably Claims~\ref{claim:non-degensym} and \ref{claim:degenfullsym}, we will see that the majority of these choices result in making $F_{i+1}$ less likely to appear in $G^k_{n,p}$ than $F_{i}$. This is the reason behind line~\ref{line:growaltwhileconditionV}:
While $\lambda(F_i) \geq 0$, it is possible that $F_i$ appears in $G^k_{n,p}$ with positive probability, however once $\lambda(F) \leq -\gamma$ for some subhypergraph $F \subseteq F_i$,
we can show immediately using Markov's inequality that $F_i$ does not exist in $G^k_{n,p}$ a.a.s. 

The outputs $F$ of \textsc{Grow-$\hat{\mathcal{B}}_{\eps}$-Alt} have several properties that will be useful to us. 
Firstly, all their edges are \emph{not} open; that is, $F \in \mathcal{C}(H_1, H_2)$. 
Secondly, we will see later that every $k$-graph $\hat{G}$ constructable by \textsc{Grow-Alt}, 
is constructable using \textsc{Grow-$\hat{\mathcal{B}}_{\eps}$-Alt}. 
Notably, as before, we will see that \textsc{Grow-Alt} cannot construct any $k$-graph which is an output of \textsc{Grow-$\hat{\mathcal{B}}_{\eps}$-Alt}. 
This will be particularly useful in the proofs of Claims~\ref{claim:growsymv} and \ref{claim:conclusionsym}.
In some sense, \textsc{Grow-$\hat{\mathcal{B}}_{\eps}$-Alt} outputs those $k$-graphs which the while-loop of \textsc{Grow-Alt} could get stuck on and which are the building blocks for those subhypergraphs of $G^k_{n,p}$ which one cannot trivially colour, using the while-loop of \textsc{Asym-Edge-Col-$\hat{\mathcal{B}}_{\eps}$} or otherwise.
\COMMENT{JH (repeat comment): I'm not sure how happy I am referring so much to \textsc{Grow-Alt($G, \eps$)} here. But this is crucial to understanding why algorithm \textsc{$Grow-\hat{\mathcal{B}}_{\eps}$-Alt} looks like it does.}
Thirdly, we make the following claim.
\begin{claim}\label{claim:mgbound2}
Let $F$ be an output of \textsc{Grow-$\hat{\mathcal{B}}_{\eps}$-Alt}. Then $m(F) \leq m_k(H_1, H_2) + \eps$.    
\end{claim}
The proof of this claim is identical to the proof of Claim~\ref{claim:mgbound}.

We finish this subsection by describing 
\textsc{Grow-$\hat{\mathcal{B}}_{\eps}$-Alt} in more detail. In lines~\ref{line:growaltanyrV} and \ref{line:growalti0V} we initialise the $k$-graph as a copy of $H_1$ and initialise $i$ as $0$.
We have $\lambda(H_1) = v_1 - e_1/m_k(H_1,H_2) = k - 1/m_k(H_2) > 0$, so we enter the while-loop on line~\ref{line:growaltwhileconditionV}. 
In the $i$th iteration of the while-loop, we first check in line~\ref{line:growaltifopeneV} whether there are any open edges in $F_i$. 
If not, then every edge in $F_i$ is not open and we output this $k$-graph in line~\ref{line:growaltreturnfiV}; 
if so, then in lines~\ref{line:avalteligible-edge}-\ref{line:opengrowaltitoi+1V}, \textsc{Grow-$\hat{\mathcal{B}}_{\eps}$-Alt} invokes a procedure \textsc{Eligible-Edge-Alt} which functions similarly to \textsc{Eligible-Edge} as follows:
for a $k$-graph $F$, \textsc{Eligible-Edge-Alt($F$)} returns a specific open edge $e \in E(F)$ that does not belong to a fully open flower,
or, if no such edge exists, just returns a specific edge $e \in E(F)$ which is open. 
\textsc{Eligible-Edge-Alt} selects this edge $e$ to be \textit{unique up to isomorphism of $F_i$}, that is, for any two isomorphic $k$-graphs $F$ and $F'$, there exists an isomorphism $\phi$ with $\phi(F) = F'$ such that \[\phi(\textsc{Eligible-Edge-Alt}(F)) = \textsc{Eligible-Edge-Alt}(F').\] \textsc{Eligible-Edge-Alt} can make sure it always takes the same specific edge of $F$ regardless of how it was constructed using a theoretical look-up table of all possible $k$-graphs,
that is, \textsc{Eligible-Edge-Alt} itself does not contribute in anyway to the number of outputs of \textsc{Grow-$\hat{\mathcal{B}}_{\eps}$-Alt}.
As observed earlier, whether an edge in $F$ is in a fully open flower or not is verifiable by inspection, hence \textsc{Eligible-Edge-Alt} is well-defined.\COMMENT{JH: Need I say more here?}
\COMMENT{(This is probably not necessary, or could be merged with the corresponding footnote from the description of \textsc{Grow-$\hat{\mathcal{B}}_{\eps}$-Alt}. A reader familiar with \cite{h} will observe that this definition of \textsc{Eligible-Edge-Alt} differs from that the procedure of the same name in \cite{h}. Indeed, this is a subtle change we need to introduce, and will discuss this more in Section~\ref{sec:guidetoapp}.}
In line~\ref{line:F_i+1getsextendaltV}, we invoke procedure \textsc{Close-$e$-Alt($F_i, e$)}, using the edge we just selected with \textsc{Eligible-Edge-Alt}.
Procedure \textsc{Close-$e$-Alt($F_i, e$)} attaches either a copy $L$ of $H_1$ or a copy $R$ of $H_2$ to $F_i$ such that $L$ and $R$ have edge-intersection uniquely at $e$.
Whether \textsc{Close-$e$-Alt($F_i, e$)} attaches $L$ or $R$ depends on which $k$-graph is not fully contained in $F$, with priority given to $L$.
Observe that \textsc{Close-$e$-Alt($F_i, e$)} always adds at least one edge to $F_i$ by construction.
In line~\ref{line:opengrowaltitoi+1V}, we iterate $i$ in preparation for the next iteration of the while-loop. 

If we enter line~\ref{line:growaltreturnfiV}, then afterwards we perform procedure \textsc{Close-$e$-Alt} on some (arbitrary, closed) edge of $F_i$. Observe that the last condition in line~\ref{line:lrv} of \textsc{Close-$e$-Alt} guarantees that a new edge is added.

In Section~\ref{sec:finiteequal}, we prove that, for all $\eps \geq 0$, $\hat{\mathcal{B}}_{\eps}$ is finite in this $m_k(H_1) = m_k(H_2)$ case.

\subsection{Algorithms \textsc{Special-Alt} and \textsc{Grow-Alt}}\label{sec:growalt}

As with algorithms \textsc{Special} and \textsc{Grow},
algorithms \textsc{Special-Alt}$=$\textsc{Special-Alt($G', \eps,n$)} and \textsc{Grow-Alt}$=$\textsc{Grow-Alt($G', \eps,n$)} are both defined for any $\eps \geq 0$ (see Figures~
\ref{fig:specialalt} and \ref{growsymfig}) 
and each take input $G'$ to be the $k$-graph which \textsc{Asym-Edge-Col-$\hat{\mathcal{B}}_{\eps}$} got stuck on.
Note that if \textsc{Asym-Edge-Col-$\hat{\mathcal{B}}_{\eps}$} gets stuck, then it did not ever use \textsc{B-Colour($G',\eps$)}, 
so the processes in these algorithms are still well-defined for all $\eps \geq 0$, even though \textsc{B-Colour($G',\eps$)} may not exist for the given value of $\eps$.
\begin{figure}[!ht]
\begin{algorithmic}[1]
\Procedure{Special-Alt}{$G' = (V,E), \eps, n$}
    \If {$\forall e \in E: |\mathcal{S}^B_{G'}(e)| = 1$}\label{line:growaltforalle}
     \State $T \gets$ any member of $\mathcal{T}^B_{G'}$
     \State \Return $\bigcup_{e \in E(T)}\mathcal{S}^B_{G'}(e)$\label{line:growaltreturnsg'e}
 \EndIf
 \If{$\exists e \in E : |\mathcal{S}^B_{G'}(e)| \geq 2$}
     \State $S_1, S_2 \gets$ any two distinct members of $\mathcal{S}^B_{G'}(e)$
     \State \Return $S_1 \cup S_2$
 \EndIf\label{line:endifgrowalt}
\EndProcedure
\end{algorithmic}
\caption{The implementation of algorithm \textsc{Special-Alt}.}\label{fig:specialalt}
\end{figure}


Algorithm \textsc{Special-Alt} operates in the exact same way as algorithm \textsc{Special}. 

\begin{figure}[!ht]
\begin{algorithmic}[1]
\Procedure{Grow-Alt}{$G' = (V,E), \eps$}
 \State $e \gets$ any $e \in E : |\mathcal{S}^B_{G'}(e)| = 0$\label{line:growaltanye}
 \State $F_0 \gets$ any $R \in \mathcal{R}_{G'}:e \in E(R)$\label{line:growaltanyr}
 \State $i \gets 0$\label{line:growalti0}
 \While {$(i < \ln(n)) \land (\forall \tilde{F} \subseteq F_i : \lambda(\tilde{F}) > - \gamma)$}\label{line:growaltwhileconditions}
         \State $e \gets \textsc{Eligible-Edge-Alt}(F_i)$\label{line:eligible}
         \State $F_{i+1} \gets \textsc{Extend}(F_i, e, G')$\label{line:extend}    
         \State $i \gets i + 1$
 \EndWhile 
 \If {$i \geq \ln(n)$}
     \State \Return{$F_i$}\label{line:growaltreturnfi}
 \Else
     \State \Return{\textsc{Minimising-Subgraph}($F_i$)}
 \EndIf
\EndProcedure
\end{algorithmic}\smallskip\smallskip

\begin{algorithmic}[1]
\Procedure{Extend}{$F,e,G'$}
 \State $\{L,R\} \gets$ any pair $\{L,R\}$ such that $L \in \mathcal{L}_{G'}$, $R \in \mathcal{R}_{G'}$, $E(L) \cap E(R) = e$ and either $L \subseteq F$ or $R \subseteq F$ \label{line:lr}
 \If{$L \nsubseteq F$}\label{line:ifl}
 \State $F' \gets F \cup L$\label{line:fl}
 \Else
 \State $F' \gets F \cup R$\label{line:fr}
 \EndIf
 \State \Return {$F'$}\label{line:growaltreturnf'}
\EndProcedure
\end{algorithmic}
\caption{The implementation of algorithm \textsc{Grow-Alt}.}\label{growsymfig}
\end{figure}

Let us now discuss \textsc{Grow-Alt}. Algorithm \textsc{Grow-Alt} operates in a similar way to \textsc{Grow}, with the same inputs.
In line~\ref{line:eligible}, the function \textsc{Eligible-Edge-Alt}, introduced in procedure \textsc{Grow-Alt-$\hat{\mathcal{B}}_{\eps}$} is called. 
Recall that \textsc{Eligible-Edge-Alt} maps $F_i$ to an open edge $e \in E(F_i)$, prioritising open edges that do not belong to fully open flowers.
As with \textsc{Eligible-Edge} in Case~1, this edge $e$ is selected to be unique up to isomorphism. 
We then apply a new procedure \textsc{Extend} which attaches either a $k$-graph $L \in \mathcal{L}_{G'}$ or $R \in \mathcal{R}_{G'}$ that contains $e$ to $F_i$.
As in Case~1, because the condition in line~\ref{line:edgeremoval} of \textsc{Asym-Edge-Col-$\hat{\mathcal{B}}_{\eps}$} fails, $G' \in \mathcal{C}$.\footnote{Note that we could also conclude $G' \in \mathcal{C}^*$, however this will not be necessary in what follows.} 

We now show that the number of edges of $F_i$ increases by at least one in each iteration of the while-loop and that \textsc{Grow-Alt} operates as desired with a result analogous to Claim~\ref{claim:growv}. 

\begin{claim}\label{claim:growsymv}
Algorithm \textsc{Grow-Alt} terminates without error on any input $k$-graph $G' \in \mathcal{C}$
that is not an $(H_1, H_2)$-sparse $\hat{\mathcal{B}}_{\eps}$-graph, for which \textsc{Special-Alt($G',\eps,n$)} did not output a subhypergraph.
\footnote{See Definitions~\ref{def:avgraph} and \ref{def:h1h2sparse}.} 
Moreover, for every iteration $i$ of the while-loop, we have $e(F_{i+1}) > e(F_i)$.
\end{claim}

\begin{proof}
The assignments in lines~\ref{line:growaltanye} and \ref{line:growaltanyr} operate successfully for the exact same reasons as given in the proof of Claim~\ref{claim:growv}. 
Next, we show that the call to \textsc{Eligible-Edge-Alt} in line~\ref{line:eligible} is always successful.
Indeed, suppose for a contradiction that no edge is open in $F_i$ for some $i \geq 0$. Then every edge $e \in E(F_i)$ is the unique edge-intersection of some copy $L \in \mathcal{L}_{F_i}$ of $H_2$ and $R \in \mathcal{R}_{F_i}$ of $H_1$, by definition. 
Hence $F \in \mathcal{C}$. 
Furthermore, observe that $F$ could possibly be constructed by algorithm \textsc{Grow-$\hat{\mathcal{B}}_{\eps}$-Alt}.
Indeed, each step of the while-loop of \textsc{Grow-Alt} can be mirrored by a step in \textsc{Grow-$\hat{\mathcal{B}}_{\eps}$-Alt}, regardless of the input $k$-graph $G' \in \mathcal{C}^*$ to \textsc{Grow-Alt}; 
in particular, \textsc{Eligible-Edge-Alt} acts in the same way in both algorithms.
However, our choice of $F_0$ in line~\ref{line:growaltanyr} guarantees that $F_i$ is not in $\hat{\mathcal{B}}_{\eps}$.
Indeed, the edge $e$ selected in line~\ref{line:growaltanye} satisfying $|\mathcal{S}^B_{G'}(e)| = 0$ is an edge of $F_0$ and $F_0 \subseteq F_i \subseteq G'$. 
Thus, $F$ is not an output of \textsc{Grow-$\hat{\mathcal{B}}_{\eps}$-Alt}.
Together with the fact that $F$ could possibly be constructed by algorithm \textsc{Grow-$\hat{\mathcal{B}}_{\eps}$-Alt}, it must be the case that \textsc{Grow-$\hat{\mathcal{B}}_{\eps}$-Alt} would terminate before it is able to construct $F$.
\textsc{Grow-$\hat{\mathcal{B}}_{\eps}$-Alt} only terminates in some step $i$ when $\lambda(F_i) \leq -\gamma$, hence there exists $\tilde{F} \subseteq F$ such that $\lambda(\tilde{F}) \leq -\gamma$.
Thus \textsc{Grow-Alt} terminates in line~\ref{line:growaltwhileconditions} without calling \textsc{Eligible-Edge-Alt}. Thus every call that is made to \textsc{Eligible-Edge-Alt} is successful and returns an open edge $e$. 
Since $G' \in \mathcal{C}$, the call to \textsc{Extend}$(F_i, e, G')$ is also successful and thus there exist suitable $k$-graphs $L \in \mathcal{L}_{G'}$ and $R \in \mathcal{R}_{G'}$ such that $E(L) \cap E(R) = \{e\}$, that is, line~\ref{line:lr} is successful. 

It remains to show that for every iteration $i$ of the while-loop, we have $e(F_{i+1}) > e(F_i)$.
Since $e$ is open in \textsc{Grow-Alt} for $F_i$ and $E(L) \cap E(R) = \{e\}$, we must have that either $L \nsubseteq F_i$ or $R \nsubseteq F_i$. 
Hence \textsc{Extend} outputs $F' := F \cup L$ such that $e(F_{i+1}) = e(F') > e(F_i)$ or $F' := F \cup R$ such that $e(F_{i+1}) = e(F') > e(F_i)$.\end{proof}

We call iteration $i$ of the while-loop in algorithm \textsc{Grow-Alt} \emph{non-degenerate} if the following hold:

\begin{itemize}
 \item If $L \nsubseteq F_i$ (that is, line~\ref{line:ifl} is true), then in line~\ref{line:fl} we have $V(F_i) \cap V(L) = e$;
 \item If $L \subseteq F_i$ (that is, line~\ref{line:ifl} is false), then in line~\ref{line:fr} we have $V(F_i) \cap V(R) = e$.
\end{itemize}
Otherwise, we call iteration $i$ \emph{degenerate}. 
Note that, in non-degenerate iterations $i$, there are only a constant number of $k$-graphs that $F_{i+1}$ can be for any given $F_i$; 
indeed, \textsc{Eligible-Edge-Alt} determines the exact position where to attach $L$ or $R$ (recall that the edge $e$ found by \textsc{Eligible-Edge-Alt}($F_i$) is \emph{unique up to isomorphism of $F_i$}).

\begin{claim}\label{claim:non-degensym}
For any $\eps \geq 0$, if iteration $i$ of the while-loop in procedure \textsc{Grow-Alt($G',\eps,n$)} is non-degenerate, we have \[\lambda(F_{i+1}) = \lambda(F_i).\]
\end{claim}

\begin{proof}
In a non-degenerate iteration, we either add $v_2 - k$ vertices and $e_2 - 1$ edges for the copy $L$ of $H_2$ to $F_i$ or add $v_1 - k$ vertices and $e_1 - 1$ edges for the copy $R$ of $H_1$ to $F_i$.
In the former case,
\begin{align*}
  \lambda(F_{i+1}) - \lambda(F_i)   & = v_2 - k - \frac{e_2 - 1}{m_k(H_1,H_2)} \\
                                 & = v_2 - k - \frac{e_2 - 1}{m_k(H_2)} \\
                                 & = 0,
\end{align*} where the second equality follows from $m_k(H_1,H_2) = m_k(H_2)$ (see Proposition~\ref{prop:m2h1h2}) and the last equality follows from $H_2$ being (strictly) $k$-balanced.

In the latter case,
\begin{align*}
  \lambda(F_{i+1}) - \lambda(F_i)   & = v_1 - k - \frac{e_1 - 1}{m_k(H_1,H_2)} \\
                                 & = v_1 - k - \frac{e_1 - 1}{m_k(H_1)} \\
                                 & = 0.
\end{align*} where the second equality follows from $m_k(H_1,H_2) = m_k(H_1)$ (see Proposition~\ref{prop:m2h1h2}) and the last equality follows from $H_1$ being (strictly) $k$-balanced.\end{proof} 

As in Case~1, when we have a degenerate iteration $i$, the structure of $F_{i+1}$ depends not just on $F_i$ but also on the structure of $G'$.
Indeed, if $F_i$ is extended by a copy $L$ of $H_2$ in line~\ref{line:fl} of \textsc{Extend}, then $L$ could intersect $F_i$ in a multitude of ways.
Moreover, there may be several copies of $H_2$ that satisfy the condition in line~\ref{line:lr} of \textsc{Extend}.
One could say the same for $k$-graphs added in line~\ref{line:fr} of \textsc{Extend}.
Thus, as in Case~1, degenerate iterations cause us difficulties since they enlarge the family of $k$-graphs that algorithm \textsc{Grow-Alt} can return.
However, we will show that at most a constant number of degenerate iterations can happen before algorithm \textsc{Grow-Alt} terminates, allowing us to bound from above sufficiently well the number of non-isomorphic $k$-graphs \textsc{Grow-Alt} can return.  
Pivotal in proving this is the following claim, analogous to Claim~\ref{claim:degenfull}.

\begin{claim}\label{claim:degenfullsym}
There exists a constant $\kappa = \kappa(H_1,H_2) > 0$ such that, for any $\eps \geq 0$, if iteration $i$ of the while-loop in procedure \textsc{Grow-Alt($G',\eps,n$)} is degenerate then we have \[\lambda(F_{i+1}) \leq \lambda(F_i) - \kappa.\]
\end{claim}

Compared to the proof of Claim~\ref{claim:degenfull}, the proof of Claim~\ref{claim:degenfullsym} is relatively straightforward.

\begin{proof}
Let $F := F_i$ be the $k$-graph before the operation in line~\ref{line:extend} (of \textsc{Grow-Alt}) is carried out and let $F' := F_{i+1}$ be the output from line~\ref{line:extend}. We aim to show there exists a constant $\kappa = \kappa(H_1, H_2) > 0$ such that
\begin{equation*}
    \lambda(F) - \lambda(F') = v(F) - v(F') - \frac{e(F) - e(F')}{m_k(H_1, H_2)} \geq \kappa
\end{equation*} whether \textsc{Extend} attached a $k$-graph $L \in \mathcal{L}_{G'}$ or $R \in \mathcal{R}_{G'}$ to $F$. 
We need only consider the case when $L \in \mathcal{L}_{G'}$ is the $k$-graph added by \textsc{Extend} in this non-degenerate iteration $i$ of the while-loop, as the proof of the $R \in \mathcal{R}_{G'}$ case is identical. 
Let $V_{L'} := V(F) \cap V(L)$ and $E_{L'} := E(F) \cap E(L)$ and set $v_{L'} := |V_{L'}|$ and $e_{L'} := |E_{L'}|$.
Let ${L'}$ be the $k$-graph on vertex set $V_{L'}$ and edge set $E_{L'}$. By Claim~\ref{claim:growsymv}, we have $e(F') > e(F)$, hence ${L'}$ is a proper subhypergraph of $H_2$. Also, observe that
since iteration $i$ was degenerate, we must have that $v_{L'} \geq k+1$. 
Let $F_{\hat{L}}$ be the $k$-graph produced by a non-degenerate iteration at $e$ with a copy $\hat{L}$ of $H_2$, that is, $F_{\hat{L}} := F \cup \hat{L}$ and $V(F) \cap V(\hat{L}) = e$. Our strategy is to compare $F'$ with $F_{\hat{L}}$. 
By Claim~\ref{claim:non-degensym}, $\lambda(F) = \lambda(F_{\hat{L}})$. Thus, since $m_k(H_1, H_2) = m_k(H_2)$, we have 
\begin{equation}\label{eq:f1beginsymL}
 \lambda(F) - \lambda(F') = \lambda(F_{\hat{L}}) - \lambda(F') = v_2 - k - (v_2 - v_{L'}) - \frac{(e_2 - 1) - (e_2 - e_{L'})}{m_k(H_1, H_2)} = v_{L'} - k - \frac{e_{L'} - 1}{m_k(H_1, H_2)}. 
\end{equation}
If $e_{L'} = 1$ (that is, $E(L') = \{e\}$), then since $v_{L'} \geq k+1$ we have $\lambda(F) - \lambda(F') \geq 1$. So assume $e_{L'} \geq 2$. 
Since $H_2$ is \emph{strictly} $k$-balanced and $L'$ is a proper subhypergraph of $H_2$ with $e_{L'} \geq 2$ (that is, $L'$ is not an edge), we have that \begin{equation}\label{eq:f1l'sym}
 \frac{v_{L'} - k}{e_{L'} - 1} > \frac{1}{m_k(H_2)}.
\end{equation}
Using \eqref{eq:f1beginsymL}, \eqref{eq:f1l'sym} and that $e_{L'} \geq 2$ and $m_k(H_1, H_2) = m_k(H_2)$, we have 
\begin{equation*}
 \lambda(F) - \lambda(F') = (e_{L'} - 1)\left(\frac{v_{L'} - k}{e_{L'} - 1} - \frac{1}{m_k(H_2)}\right) > 0.
\end{equation*}
Letting $\alpha := \frac{1}{2} \min\left\{(e_{L'} - 1)\left(\frac{v_{L'} - k}{e_{L'} - 1} - \frac{1}{m_k(H_2)}\right): L' \subsetneq H_2, e_{L'} \geq 2\right\}$, we take $$\kappa := \min\{1, \alpha\}.$$\end{proof}

Together, Claims~\ref{claim:non-degensym} and \ref{claim:degenfullsym} yield the following claim (analogous to Claim~\ref{claim:q_1}). 

\begin{claim}\label{claim:q_2}
There exists a constant $q_2 = q_2(H_1,H_2)$ such that, for any input $G'$,  algorithm \textsc{Grow-Alt($G',\eps^*,n$)} performs at most $q_2$ degenerate iterations before it terminates.
\end{claim}

\begin{proof}
Analogous to the proof of Claim~\ref{claim:q_1}.\end{proof}

For $0 \leq d \leq t \leq \lceil \ln(n) \rceil$, let $\mathcal{F}_{\textsc{Alt}}(t,d)$ denote a family of representatives for the isomorphism classes of all $k$-graphs $F_t$ that algorithm \textsc{Grow-Alt} can possibly generate after exactly $t$ iterations of the while-loop with exactly $d$ of those $t$ iterations being degenerate. 
Let $f_{\textsc{Alt}}(t,d) := |\mathcal{F}_{\textsc{Alt}}(t,d)|$.

\begin{claim}
There exist constants $C_0 = C_0(H_1, H_2)$ and $A = A(H_1,H_2)$ such that, for any input $G'$ and any $\eps \geq 0$, $$f_{\textsc{Alt}}(t,d) \leq \lceil \ln(n) \rceil^{(C_0 +1)d} \cdot A^{t-d}$$ for $n$ sufficiently large.
\end{claim}

\begin{proof}
Analogous to the proof of Claim~\ref{claim:polylog}.\end{proof}

Let $\mathcal{F}_{\textsc{Alt}} = \mathcal{F}_{\textsc{Alt}}(H_1, H_2, n, \eps^*)$ be a family of representatives for the isomorphism classes of \emph{all} $k$-graphs that can be an output of either \textsc{Special-Alt($G',\eps^*,n$)} or \textsc{Grow-Alt($G',\eps^*,n$)}.

Note that the proof of the following claim requires finiteness of $\hat{\mathcal{B}}_{\eps^*}$. 

\begin{claim}\label{claim:conclusionsym}
There exists a constant $b = b(H_1,H_2) > 0$ such that for all $p \leq bn^{-1/m_k(H_1,H_2)}$, $G^k_{n,p}$ does not contain any $k$-graph from $\mathcal{F}_{\textsc{Alt}}(H_1,H_2,n,\eps^*)$ a.a.s.
\end{claim}

\begin{proof}
Analogous to the combination of  proofs of Claims~\ref{claim:conclusion1} and \ref{claim:conclusion2}.\end{proof}

\begin{proofofnoerrorlemmacasetwo}
Suppose that the call to \textsc{Asym-Edge-Col-$\hat{\mathcal{B}}_{\eps}$}$(G)$ gets stuck for some $k$-graph $G$, and consider $G' \subseteq G$ at this moment. 
Then \textsc{Special-Alt}$(G', \eps^*,n)$ or \textsc{Grow-Alt}$(G', \eps^*,n)$ returns a copy of a $k$-graph $F \in \mathcal{F}_{\textsc{Alt}}(H_1, H_2, n, \eps^*)$ that is contained in $G' \subseteq G$. 
By Claim~\ref{claim:conclusionsym}, this event a.a.s.~ does not occur in $G = G^k_{n,p}$ with $p$ as claimed. Thus \textsc{Asym-Edge-Col-$\hat{\mathcal{B}}_{\eps}$} does not get stuck a.a.s.~and, by Lemma~\ref{lemma:avcolour}, finds a valid edge-colouring for $H_1$ and $H_2$ of $G^k_{n,p}$ with $p \leq bn^{-1/m_k(H_1,H_2)}$ a.a.s.\qed
\end{proofofnoerrorlemmacasetwo}

\subsection{Proof of Lemma~\ref{lem:finiteequal}}\label{sec:finiteequal}

We call iteration $i$ of the while loop of algorithm \textsc{Grow-$\hat{\mathcal{B}}_{\eps}$-Alt} \emph{non-degenerate} if all of the following hold:

\begin{itemize}
    \item \textsc{Close-$e$-Alt} is called; 
    \item If line~\ref{line:flv} of \textsc{Close-$e$-Alt} is entered into, then we have $V(F) \cap V(L) = e$;
    \item If line~\ref{line:frv} of \textsc{Close-$e$-Alt} is entered into, then we have $V(F) \cap V(R) = e$.
\end{itemize}

Otherwise, we call iteration $i$ \emph{degenerate}. 
So a non-degenerate iteration in algorithm \textsc{Grow-$\hat{\mathcal{B}}_{\eps}$-Alt} is equivalent to a non-degenerate iteration in algorithm \textsc{Grow-Alt} -- in terms of the exact structure of the attached $k$-graph which, importantly, means the same number of edges and vertices are added -- and a degenerate iteration in algorithm \textsc{Grow-$\hat{\mathcal{B}}_{\eps}$-Alt} is equivalent to a degenerate iteration of \textsc{Grow-Alt}.

As with \textsc{Grow-$\hat{\mathcal{B}}_{\eps}$}, in order to conclude that \textsc{Grow-$\hat{\mathcal{B}}_{\eps}$-Alt} outputs only a finite number of $k$-graphs, 
    it is sufficient to prove that there exists a constant $K$ such that there are no outputs of \textsc{Grow-$\hat{\mathcal{B}}_{\eps}$-Alt} after $K$ iterations of the while-loop.
In order to prove this, we employ Claims~\ref{claim:non-degensym} and \ref{claim:degenfullsym}. Since degenerate and non-degenerate iterations are defined the exact same in algorithms \textsc{Grow-Alt} and \textsc{Grow-$\hat{\mathcal{B}}_{\eps}$-Alt}, we may write Claims~\ref{claim:non-degensym} and \ref{claim:degenfullsym} with reference to \textsc{Grow-$\hat{\mathcal{B}}_{\eps}$-Alt}.

\begin{claim:non-degensym}
For any $\eps \geq 0$, if iteration $i$ of the while-loop in procedure \textsc{Grow-$\hat{\mathcal{B}}_{\eps}$-Alt} is non-degenerate, we have \[\lambda(F_{i+1}) = \lambda(F_i).\]
\end{claim:non-degensym}

\begin{claim:degenfullsym}
There exists a constant $\kappa_2 = \kappa_2(H_1,H_2) > 0$ such that, for any $\eps \geq 0$, if iteration $i$ of the while-loop in procedure \textsc{Grow-$\hat{\mathcal{B}}_{\eps}$-Alt} is degenerate then we have \[\lambda(F_{i+1}) \leq \lambda(F_i) - \kappa_2.\]
\end{claim:degenfullsym}

For any iteration $i$, if \textsc{Close-$e$-Alt} is used, then let $J_i:=L$ if using line~\ref{line:flv}, and $J_i:=R$ if using line~\ref{line:frv}. 

If $i$ is non-degenerate, then we call $J_i$ a \emph{flower}\footnote{We will use the nomenclature from Section~\ref{sec:bfinite} here to draw effective comparisons, even though pictorially the names may not line up with the structures discussed.}, and we call the edge $e$ for which \textsc{Close-$e$-Alt} was run on the \emph{attachment edge} of $J_i$. 
For a flower $J_i$ with attachment edge $e = \{x_1, \ldots, x_k\}$, call $E(J_i) \setminus \{e\}$ the \emph{petal edges} and call $V(J_i) \setminus \{x_1, \ldots, x_k\}$ the \emph{internal vertices}.

Call a flower $J_{i}$ {\it fully open at time $\ell$}
if it was produced in a non-degenerate iteration at time $i$ with $i \leq \ell$, and no internal vertex of $J_i$ intersects with an edge added after time $i$.

\COMMENT{RH: names of things: flowers, petal edges, internal edges, internal vertices, `destroyed', etc, could be better to change these.}

\begin{lem}\label{lem:newfiniteness1a}
Let $G$ be any $k$-graph and $e \in E(G)$. 
Suppose $J_i$ is a flower attached at $e$.
\begin{itemize}
\item[(i)] If $H$ is any copy of $H_1$ or $H_2$ in $G$ containing an internal vertex (or equivalently a petal edge) of $J_i$, then $H=J_i$.
\item[(ii)] Every petal edge of $J_i$ is open at time $i$. 
\end{itemize}
\end{lem}

\begin{proof}
For (i), suppose not, that is there exists such an $H\not=J_i$. 
Since $H_1$ and $H_2$ are both strictly $k$-balanced and $m_k(H_1)=m_k(H_2)$, it follows that neither is a strict subhypergraph of the other, and so $H$ is not a subhypergraph of $J_i$. 
Letting $e=\{x_1, \ldots, x_k\}$, it follows that $H$ must contain at least one vertex of $V(G) \setminus \{x_1, \ldots, x_k\}$. 
Further, by Lemma~\ref{lem:newconnectivity} and Corollary~\ref{col:newconcor}, it must contain every vertex in $e$, i.e. $x_1, \ldots, x_k$.\COMMENT{JH: Do we need to give the explanation of why again since its already appeared twice in previous proofs.}
Let $H_g:=H[V(G)]$ and $H_j:=H[V(J_i)]$.
The above shows that $v_H=v_{H_g}+v_{H_j}-k$ and $v_{H_g},v_{H_j} \geq k+1$.
If $e \notin E(H)$, then add $e$ to $H_j$. Then $e_H=e_{H_g}+e_{H_j}-1$ regardless of whether $e \in E(H)$.
As $H_g$ and $H_j$ are strict subhypergraphs of $H$, and $H$ is strictly $k$-balanced, we have
\begin{align*}
\frac{e_{H_g}-1}{v_{H_g}-k}<m_k(H_1,H_2) \ \quad \ \text{and} \ \quad \ \frac{e_{H_j}-1}{v_{H_j}-k}<m_k(H_1,H_2).
\end{align*}
But this yields a contradiction, as 
\begin{align*}
m_k(H_1,H_2)=m_k(H)=\frac{e_H-1}{v_H-k} = \frac{e_{H_g}-1+e_{H_j}-1}{v_{H_g}-k+v_{H_j}-k}<m_k(H_1,H_2),
\end{align*} where the inequality follows from Fact~\ref{fact:ineq}(i).

For (ii), suppose that there exists a petal edge $f$ of $J_i$ which is closed. 
Being closed implies there exist copies $L$ of $H_2$ and $R$ of $H_1$ such that $E(L) \cap E(R) =\{f\}$.
However (i) implies that the only copy of $H_1$ or $H_2$ which contains $f$ is $J_i$, a contradiction.
\end{proof}

\begin{claim}\label{claim:newfinitenessc1a} 
Suppose $J_i$ is a fully open flower at time $\ell-1$ and $J_{\ell}$ is such that $V(J_{\ell})$ has empty intersection with the internal vertices of $J_i$. 
Then $J_i$ is fully open at time $\ell$.
\end{claim}

\begin{proof}
Let $e$ be the attachment edge of $J_i$. 
Since $J_i$ is a fully open flower at time $\ell-1$, we know that none of the $k$-graphs $J_{i+1}, \dots, J_{\ell}$ touch any of the internal vertices of $J_i$, 
the $k$-graph $G$ with edge set $E(G):=E(F_{\ell}) \setminus (E(J_i) \setminus \{e\})$ fully contains each of $J_{i+1},\dots,J_{\ell}$. 
Therefore, we may apply Lemma~\ref{lem:newfiniteness1a} to $G$ and $J_i$, and deduce that the petal edges of $J_i$ are still open. 
It follows that $J_i$ is fully open at time $\ell$.
\end{proof}

\begin{claim}\label{claim:newfinitenessc2a}
Suppose $J_{\ell}$ is a flower. 
Then at most one fully open flower is no longer fully open.
\end{claim}

\begin{proof}
Let $e$ be the attachment edge of $J_{\ell}$. 
Note that any pair of internal vertices from two different fully open flowers are non-adjacent, and therefore $e$ contains the internal vertices of at most one fully open flower. 
Let $J_i$ be any of the other fully open flowers (clearly $i<\ell$). 
Now we can apply Claim~\ref{claim:newfinitenessc1a}, and deduce that $J_i$ is still fully open, as required.
\end{proof}

\begin{claim}\label{claim:newfinitenessc3a}
Suppose $J_{\ell}$ and $J_{\ell+1}$ are each flowers, and further after $J_{\ell}$ is added, one fully open flower is no longer fully open.
Then after $J_{\ell+1}$ is added, the total number of fully open flowers does not decrease.
\end{claim}

\begin{proof}
Suppose $J_i$, $i<\ell$ is the flower which is no longer fully open after $J_{\ell}$ is added. 

We claim that $J_i$ still contains an open edge in $F_{\ell}$. 

Let $e$ be the attachment edge of $J_{\ell}$. 
We note that $e \in J_i$ and further, since $e$ was an open edge previously, it must be a petal edge of $J_i$.
Now let $f$ be any other petal edge of $J_i$ and suppose for a contradiction that it is closed in $F_{\ell}$.

First, this means that there exist copies $L$ of $H_2$ and $R$ of $H_1$ such that $E(L) \cap E(R) = \{f\}$.
Let $g$ be the attachment edge of $J_i$. 
Since none of the $k$-graphs $J_{i+1}, \dots, J_{\ell-1}$ touch any of the internal vertices of $J_i$, 
the $k$-graph $G$ with edge set $E(G):=E(F_{\ell}) \setminus (E(J_{\ell}) \cup E(J_i) \setminus \{g\})$ fully contains each of $J_{i+1},\dots,J_{\ell-1}$. 
Hence we can apply Lemma~\ref{lem:newfiniteness1a}(i) with $G$ and $J_i$, which shows that the only copy of $H_1$ or $H_2$ containing $f$ in $F_{\ell-1}$ is $J_i$. 
Therefore one of $L$ or $R$ must contain a petal edge of $J_{\ell}$.
However this is not possible, since Lemma~\ref{lem:newfiniteness1a}(i) applied with $G=F_{\ell-1}$ and $J_{\ell}$ 
gives that $L$ and $R$ cannot contain any petal edges of $J_{\ell}$.
\COMMENT{RH: This last paragraph, and the proof of Lemma~\ref{lem:newfiniteness1a} are essentially the only things which are different in the $m_k(H_1)=m_k(H_2)$ case, compared to the $m_k(H_1)>m_k(H_2)$ case.
Also note that this last paragraph is completely skipped in the Nenadov+Steger short proof for the symmetric case; the fact is just stated, as though its obvious!}

Since \textsc{Eligible-Edge-Alt} prioritises open edges which do not belong to fully open flowers, the next iteration will select an edge, possibly of $J_i$, to close which does not destroy\footnote{That is, make \emph{not} fully open.} any fully open copies. 
That is, the attachment edge of $J_{\ell+1}$ must be either an open edge from $J_i$ or some other open edge not belonging to a fully open flower at time $\ell$.
Since $J_i$ has already been destroyed, we have that $J_{\ell+1}$ does not destroy any fully open flowers using the argument from Claim~\ref{claim:newfinitenessc2a}.\COMMENT{JH: Possibly a nicer way of writing this. I don't think it is possible to stipulate that Eligible-Edge-Alt selects an open edge exactly in $J_i$. One could think up very symmetrical looking $k$-graphs and Eligible-Edge-Alt won't be able to distinguish which `$J_i$' we're talking about here.}
\end{proof}

We can now prove Lemma~\ref{lem:finiteequal}.

\begin{proof}
By Claims~\ref{claim:non-degensym} and~\ref{claim:degenfullsym} 
it easily follows that the number of degenerate iterations is finite, specifically at most $X:=(\lambda(H_1) + \gamma)/ \kappa_2$ (this comes from the proof of Claim~\ref{claim:q_2}), to maintain that $\lambda \geq -\gamma$.

Let $f(\ell)$ be the number of fully open flowers at time $\ell$. 
For the algorithm to output a $k$-graph $G$ at time $\ell$, we must have $f(\ell)=0$. 
By Claim~\ref{claim:newfinitenessc1a}, if iteration $i$ is degenerate, then $f(i) \geq f(i-1)-Y$, where $Y:=\max \{v_1,v_2\}$ (note $v(J_i) \leq Y$ for all $i$). 
By Lemma~\ref{lem:newfiniteness1a}, Claim~\ref{claim:newfinitenessc2a} and
Claim~\ref{claim:newfinitenessc3a},
if $J_i$ and $J_{i+1}$ are both
non-degenerate iterations,
then $f(i+1) \geq f(i-1)+1$.

Now assume for a contradiction that there are $2XY+2$ consecutive non-degenerate iterations.
It follows that after all $X$ degenerate iterations are performed, we will still have
$f(\ell) \geq (2XY+2)/2 -XY>0$. 
Since any further iteration cannot decrease $f(\ell)$ while still maintaining that $\lambda \geq -\gamma$, the algorithm will not output any $k$-graph $G$ at any time $\ell' \geq \ell$. 
Thus \textsc{Grow-$\hat{\mathcal{B}}_{\eps}$-Alt} outputs only a finite number of $k$-graphs.
\end{proof}

\end{document}